\documentclass[12pt]{article}
\usepackage{a4wide}
\usepackage{amsmath, amsthm, amsfonts, amssymb, bbm}
\usepackage[mathscr]{eucal}
\usepackage{graphics}
\usepackage{xypic}
\usepackage[all]{xy}\xyoption{rotate}
\usepackage[english]{babel}
\usepackage[font=small,format=plain,labelfont=bf,up]{caption}
\usepackage[bbgreekl]{mathbbol}
\usepackage{tikz}
\usetikzlibrary{matrix,arrows}

\usepackage{hyperref}       
\hypersetup{       
   pdftex,
   colorlinks=true,        
   pdfstartview=FitH,      
   allcolors=[rgb]{0,0.3,0.6},        
   bookmarks=true,          
	 bookmarksdepth=section, 
   bookmarksopen=false,     
   bookmarksnumbered=true, 
	 backref, pagebackref   
}

\allowdisplaybreaks

\theoremstyle{plain}

\newtheorem{theorem}{Theorem}
\newtheorem{lemma}[theorem]{Lemma}
\newtheorem{proposition}[theorem]{Proposition}
\newtheorem{conjecture}[theorem]{Conjecture}
\newtheorem{corollary}[theorem]{Corollary}

\theoremstyle{definition}

\newtheorem{remark}[theorem]{Remark}

\newtheorem{definition}[theorem]{Definition}

\numberwithin{equation}{section}
\numberwithin{theorem}{section}

\newcommand\void[1]{}

\newcommand\be            {\begin{equation}}
\newcommand\ee            {\end{equation}}

\def\ZZ{{\mathbb Z}}
\def\RR{{\mathbb R}}
\def\CC{{\mathbb C}}
\def\NN{{\mathbb N}}

\newcommand{\rmi}{\mathrm{i}}

\newcommand{\tensor}{\otimes}

\newcommand{\Ncop}{\Delta_{t}}
\newcommand{\one}{\boldsymbol{1}}%{1\kern-4pt 1}

	\newcommand{\re}[1]{[\hspace*{-1.7pt}[#1]\hspace*{-1.7pt}]}
\newcommand{\e}[1]{e_{#1}}

\newcommand{\ffrac}[2]{\mbox{\footnotesize $\displaystyle\frac{#1}{#2}$}}
\newcommand{\half}{%
  \mathchoice{\ffrac{1}{2}}{\frac{1}{2}}{\frac{1}{2}}{\frac{1}{2}}}

\newcommand{\SingTyp}[1]{\mathcal{F}_{#1}}  % typical singlet module
\newcommand{\SingAtyp}[1]{\mathcal{M}_{#1}} % atypical singlet module
\newcommand{\SingStag}[1]{\mathcal{P}_{#1}} % typical singlet module
\newcommand{\TripIrr}[1]{\mathcal{W}_{#1}}  % irreducible triplet module
\newcommand{\TripTyp}[1]{\mathcal{V}_{#1}}
\newcommand{\TripStag}[1]{\mathcal{R}_{#1}} % staggered module

\newcommand{\SingAlg}[1]{\mathcal{M} \bigl( #1 \bigr)}                     % singlet algebra
\newcommand{\TripAlg}[1]{\mathcal{W} \bigl( #1 \bigr)}                     % triplet algebra

\newcommand{\UqgTyp}[1]{V_{#1}} %typical unrolled quantum group module
\newcommand{\UqgStag}[1]{P_{#1}} %typical unrolled quantum group module

\newcommand{\UQG}[1]{\overline{U}_{#1}^H s\ell(2)} % Unrolled Quantum Group

\newcommand{\rep}{\mathrm{\bf Rep}\,}
\newcommand{\repsimple}{\mathrm{\bf Rep}_{\langle\mathrm{s}\rangle}\,}
\newcommand{\repsimpleinfinity}{\mathrm{\bf Rep}^{\oplus}_{\langle\mathrm{s}\rangle}\,}
\newcommand{\repfd}{\rep^{\!\mathrm{fd}}}

\newcommand{\repwt}{\rep_{\!\mathrm{wt}\,}}
\newcommand{\repfdwt}{\rep^{\!\mathrm{fd}}_{\!\mathrm{wt}\,}}

\newcommand{\eps}{\epsilon}

\newcommand{\q}{q}

\newcommand{\catU}{\mathcal{C}}

\newcommand{\catUfd}{\catU^{\mathsf{fd}}}

\newcommand{\catQ}{\mathcal{D}}

	\newcommand{\algC}{\Lambda}

\newcommand{\repS}{\mathsf{S}}

\newcommand{\stprp}{\mathsf{a}}
\newcommand{\atprp}{\mathsf{v}}

\newcommand{\Cmod}{{}_\algC\catU_p}
	\newcommand{\Cmodl}{{}_{\algC}(\catU_{\mathit{p}}^\oplus)^\mathrm{fg\text{-}loc}}

\newcommand{\funF}{\mathcal{F}}
\newcommand{\funG}{\mathcal{G}}

\newcommand{\idem}{\boldsymbol{e}}

\newcommand{\catTrip}[1]{\rep\TripAlg{#1}}

\newcommand{\UresSL}[1]{\overline{U}_{q} s\ell(#1)}

\newcommand{\SLiiZ}{SL(2,\ZZ)}

\newcommand{\Q}{\overline{U}^{(\Phi)}_{q} s\ell(2)}  % our quasi-Hopf algebra
\newcommand{\Qb}{\overline{U}^{(\Phi,\zeta)}_{q} s\ell(2)}  % our quasi-Hopf algebra with \beta

\newcommand{\modO}{\mathcal{O}}
\newcommand{\modX}{\mathcal{X}}

\newcommand{\fuse}{\mathbin{\boxtimes}}                                 % fusion
\newcommand{\ribbon}{{\boldsymbol{v}}}

\newcommand{\Salpha}{\boldsymbol{\alpha}}
\newcommand{\Sbeta}{\boldsymbol{\beta}}

 \newcommand{\id}{\mathrm{id}}

\newcommand{\brC}{c}

\newcommand{\dE}{E_{\catQ}}
\newcommand{\dF}{F_{\catQ}}

\newcommand{\lra}{\longrightarrow}
\newcommand{\dses}[5]{0 \lra #1 \overset{#2}{\lra} #3 \overset{#4}{\lra} #5 \lra 0} % displayed ses

\newcommand{\tensorC}{\tensor_\algC}
\newcommand{\tensorL}{\tensor_\Lambda}
\newcommand{\cat}{\mathcal{C}}
\newcommand{\catD}{\mathcal{D}}
\newcommand{\fun}{\mathcal{F}}

\newcommand{\assD}{\alpha^{\catD}}
\newcommand{\flip}{\tau}                    %%% in Vect

\newcommand{\gcg}{\ , \ } % gap-comma-gap 
\newcommand{\gp}{\ .}
\newcommand{\gc}{\ ,}

\newcommand{\assoc}{\alpha} %%% for associators in a category
\newcommand{\repQ}{\catQ_p}

\newcommand{\tfunF}{\tilde\funF}

\newcommand{\eq}{\mathrm{eq}}
\newcommand{\im}{\mathrm{im}}
\newcommand{\Span}{\mathrm{span}}

\newcommand{\Hplus}{\mathcal{H}^\mathrm{\oplus}}
\newcommand{\RS}{\mathsf{S}}

\title{A quasi-Hopf algebra for the \\
triplet vertex operator algebra}

\author{Thomas Creutzig\thanks{Department of Mathematical and Statistical Sciences, University of Alberta,
Edmonton, Alberta  T6G 2G1, Canada. 
}
~~,~~~
Azat M. Gainutdinov\thanks{Laboratoire de Math\'ematiques et Physique Th\'eorique CNRS, Universit\'e de Tours,
Parc de Grammont, 37200 Tours, 
France.
}\, \thanks{Institut f\"ur Mathematik,
Mathematisch-naturwissenschaftliche Fakult\"at,
Universit\"at Z\"urich, Winterthurerstr.\,190, CH-8057 Z\"urich, Switzerland.}\, ${}^{\S}$
~,~~~
Ingo Runkel\thanks{Fachbereich Mathematik, Universit\"at Hamburg, Bundesstr.\,55, 20146 Hamburg, Germany.}
}

\date{}

\begin{document}

\newcommand {\tr}{\text{tr}}
\newcommand {\ch}{\text{ch}}

\maketitle

\begin{abstract}
We give a new factorisable ribbon quasi-Hopf algebra $U$, whose underlying algebra is that of the restricted quantum group for $s\ell(2)$ at a $2p$'th root of unity. The representation category of $U$ is conjecturally ribbon-equivalent to that  of the triplet vertex operator algebra $\TripAlg{p}$. 
We obtain $U$ via a simple current extension from the unrolled restricted quantum group at the same root of unity. The representation category of the unrolled quantum group is conjecturally equivalent to that of the singlet vertex operator algebra $\SingAlg{p}$, and our construction is parallel to extending $\SingAlg{p}$ to  $\TripAlg{p}$.

We illustrate the procedure in the simpler example of passing from the Hopf algebra for the group algebra $\CC\ZZ$ to a quasi-Hopf algebra for $\CC\ZZ_{2p}$, which corresponds to passing from the Heisenberg vertex operator algebra to a lattice extension.
\end{abstract}

\thispagestyle{empty}

\newpage 

\tableofcontents

\newpage

\section{Introduction}

We define a new factorisable ribbon quasi-Hopf algebra,
 called $\Q$, which is a modification of $\UresSL2$, the restricted quantum group for $s\ell(2)$ with $q$ the primitive $2p$'th root of unity $q = e^{i \pi/p}$ and $p \ge 2$ an integer,
and $\Phi$ is a 3-cocycle of $\ZZ_{2p}$. 
The quasi-Hopf algebra $\Q$ is interesting for at least two reasons, which we explain in turn. 

The first point of interest is the relation to $\TripAlg{p}$, the triplet vertex operator algebra (VOA),
which provides  one of the best investigated logarithmic conformal field theories \cite{Kausch:1990vg,GK,FHST,Carqueville:2005nu,AM2}.
The relation between $\UresSL2$ and $\TripAlg{p}$ has first been observed in \cite{FGST}. The equivalence
\be\label{eq:intro-equiv}
	\rep \UresSL2 ~\cong~\rep \TripAlg{p}
\ee
of $\CC$-linear categories was conjectured in~\cite{Feigin:2005xs} (with proof for $p=2$) and shown for general $p$ in \cite{Nagatomo:2009xp}.
 However, the category $\rep \TripAlg{p}$ is a braided tensor category \cite{HLZ,H}, while $\rep \UresSL2$ is not braidable \cite{KS,GR1}. It is then an obvious question if there is some modification of the coproduct of $\UresSL2$ which leaves the algebra structure untouched  and which allows for a universal $R$-matrix, such that \eqref{eq:intro-equiv} has a chance of becoming a braided monoidal equivalence. This is achieved by the quasi-Hopf algebra $\Q$ which we now introduce.

\medskip

	As a $\CC$-algebra, 
$\Q$ is generated by $E$, $F$, $K^{\pm1}$ with relations 
\begin{align}
  &KEK^{-1}=q^2E
  \quad,\quad
  KFK^{-1}=q^{-2}F
  \quad,\quad
  [E,F]=\ffrac{K-K^{-1}}{q-q^{-1}}\ ,
  \nonumber \\
  &E^{p}=0=F^{p}
    \quad,\quad
  K^{2p}=\one 
\end{align}
and has dimension $2p^3$.
Define the central idempotents $\idem_0 = \tfrac12 (\one+K^p)$ and $\idem_1 =\one - \idem_0$. The coproduct is given by $\Delta(K) = K\otimes K$ and
\begin{align}
\Delta(E) &= E\otimes K + (\idem_0+\q\,\idem_1)\otimes E \ ,
\nonumber
\\
\Delta(F) &= F\otimes 1 +  (\idem_0+\q^{-1}  \idem_1)K^{-1}\otimes F \ .\end{align}
The counit is defined by $\eps(E)=0=\eps(F)$ and $\eps(K)=1$.
The coproduct is non-coassociative and we have the non-trivial coassociator \begin{equation}
\Phi = \one\otimes\one\otimes\one +  \idem_1\otimes\idem_1\otimes \bigl(K^{-1}-\one\bigr)\ .
\end{equation}
The antipode is
\be
S(E)= -EK^{-1}(\idem_0 + q\,\idem_1) ~~ , \quad
S(F) = -KF (\idem_0+q^{-1}\idem_1) ~~ , \quad
S(K) = K^{-1}\ .
\ee
and the evaluation and coevaluation elements are $\Salpha=\one$ 
and $\Sbeta=\idem_0 + K^{-1}\idem_1$. 
The universal $R$-matrix and the ribbon element for $\Q$ are 
	given in \eqref{eq:R-quasiH} and \eqref{eq:ribbon-quasiH}
(where one has to set the additional gauge parameter $t$ to $1$).	
Our first main result is (see Theorem~\ref{thm:equiv-1}\,(1)):

\begin{theorem}\label{thm:main1-intro}
$\Q$ is a factorisable ribbon quasi-Hopf algebra.
\end{theorem}

Factorisability means that a certain element in $\Q^{\,\otimes 2}$ is non-degenerate -- for Hopf algebras this reduces to non-degeneracy of the monodromy matrix $R_{21}R$. We refer to \cite{BT} and to \cite[Rem.\,6.6]{FGR1} for details. 
It is natural to expect $\rep \TripAlg{p}$ to be factorisable, too (see e.g.\ \cite[Conj.\,5.7]{GnR2} and \cite[Conj.\,3.2]{CG}).
This leads to the conjecture (Corollary~\ref{cor:main} to Conjectures~\ref{conj:M-ribbon} and~\ref{conj:M-C}):

\begin{conjecture}\label{conj:trip-QG}
$\rep \Q \cong \rep \TripAlg{p}$ as $\CC$-linear ribbon categories.
\end{conjecture}

In the way we presented it, the conjecture comes a bit out of the blue, but we will now provide some support for it by analysing fusion rules and by exploiting the connection to the singlet model (in diagram~\eqref{eq:intro-cat-diag} below).

Fusion rules for $\rep \TripAlg{p}$ (the product in the Grothendieck ring) are known \cite{TW}, see also \cite{FHST, GR2}, and one can compute from these the Perron-Frobenius dimension of all irreducible $\TripAlg{p}$-modules. They are given by the unique one-dimensional 
	representation of the fusion ring, 
for which all irreducibles are represented by positive real numbers. One finds that all Perron-Frobenius dimensions of $\rep \TripAlg{p}$ are positive integers.
Finite tensor categories with this property are necessarily equivalent to representation categories of a quasi-Hopf algebra, see e.g.\ \cite[Prop.\,6.1.14]{EGNO-book}, where dimensions of irreducibles are provided by the Perron-Frobenius dimensions.
Moreover, the Perron-Frobenius dimensions of $\rep \TripAlg{p}$  agree with the corresponding dimensions of the simple $\UresSL2$ modules under the equivalence~\eqref{eq:intro-equiv}.
It is thus clear from the start that Conjecture~\ref{conj:trip-QG} has to be true for some quasi-Hopf algebra with underlying algebra $\UresSL2$.

\begin{remark}\label{rem:intro-first}
For $p=2$, the triplet VOA agrees with the even part of
the super-VOA for a single pair of symplectic fermions. A quasi-Hopf algebra whose representation category is conjecturally ribbon equivalent to representations of the even part of the super-VOA for $N$ pairs of symplectic fermions was given in \cite{GR1,FGR2}, building on previous work on the symplectic fermions category in \cite{Runkel:2012cf,Davydov:2012xg}.
For $N=1$ this quasi-Hopf algebra is isomorphic to the one presented here (with $p=2$), see Remark~\ref{rem:p-odd}.
\end{remark}

The second point of interest is the relation to the unrolled quantum group $\UQG{q}$ and to the singlet VOA $\SingAlg{p}$. In fact, this relation was the main motivation to write this paper, and it can be summarised in the following diagram, whose ingredients we proceed to explain:
\be\label{eq:intro-cat-diag}
\raisebox{\height}{
\xymatrix{
\repwt \UQG{q} ~\ni \algC
\hspace*{2em}
\ar[d]_[right]{\sim}^{~\text{conjecture}} 
\ar@{~>}[rrr]^{\text{\begin{minipage}{5.5em}pass to local\\[-.5em] $\algC$-modules\end{minipage}}}
&&&
\hspace*{2em}
\algC\text{-mod}^\mathrm{loc}
	\ar[d]_[right]{\sim}^{~\text{consequence}} 
&\mbox{}\;\rep \Q
\ar[l]_\sim^{\mathcal{F}} 
\\
\repsimple \SingAlg{p} ~\ni \TripAlg{p}
\hspace*{2em}
\ar@{~>}[rrr]^{\text{\begin{minipage}{7.5em}pass to represen-
\\[-.5em]
tations of $\TripAlg{p}$\end{minipage}}}
&&&
\hspace*{2em}
\rep \TripAlg{p}
}}
\ee
Let us start in the upper left corner. The unrolled quantum group $\UQG{q}$ is defined in terms of generators and relations similar to the restricted quantum group $\UresSL2$, 
but with two important differences. 
Firstly, the relation $K^{2p}=\one$ is not imposed (so that $K$ has infinite order) and there is an additional generator $H$.
For the convention used here
 we refer to 
\cite[Sec.\,6.3]{GPT}, \cite{CGP} and Section~\ref{sec:unrolled-restricted} for details. 
Variants of the unrolled quantum group were also considered in~\cite{Oh} and in the spin chain literature, see e.g.~\cite{KS91} 
where $S_z$ plays the role of $H$ and $q^{S_z}$ that of $K$.
The algebra $\UQG{q}$ is infinite-dimensional, and it turns out that it has a continuum of simple modules. 

$\repwt \UQG{q}$ denotes the category of finite-dimensional modules which are of weight type. A module is of weight type if $H$ is diagonalisable and $K$ acts as $q^H$. This category can be turned into a ribbon category via an $R$-matrix and ribbon element defined directly on such modules
\cite{Oh,CGP} (see Section~\ref{sec:unrolled-restricted}).

The singlet VOA $\SingAlg{p}$ \cite{Fl, Ad}
is a sub-VOA of the triplet VOA $\TripAlg{p}$. 
We denote by $\repsimple\SingAlg{p}$ the smallest full subcategory of the category of $\SingAlg{p}$-modules which contains all simple modules and which is complete with respect to taking tensor products, finite sums and subquotients.
$\repsimple\SingAlg{p}$ is conjecturally a ribbon category
which satisfies \cite{CM, CGP, CMR}:\footnote{\label{thefootnote}
	As far as we are aware, this formulation of the conjecture is due to David Ridout and Simon Wood (private communication at the ESI Workshop ``Modern trends in TQFT'', February 2014).} 

\begin{conjecture}\label{conj:sing-QG}
$\repwt
 \UQG{q}
\cong \repsimple \SingAlg{p}$ as $\CC$-linear ribbon categories.
\end{conjecture}

Conjectures \ref{conj:trip-QG} and \ref{conj:sing-QG} are in fact not independent. This will be the content of the wiggly lines in \eqref{eq:intro-cat-diag} which we explain next. We start with the bottom one. We already said that $\SingAlg{p} \subset \TripAlg{p}$ is a sub-VOA, and so in particular,
$\TripAlg{p} \in \repsimple \SingAlg{p}$.\footnote{
	Actually, this is not quite right. $\TripAlg{p}$ is an infinite direct sum of $\SingAlg{p}$-modules, and it does not satisfy the finiteness conditions we impose on modules in $\repsimple \SingAlg{p}$. Instead, we should replace $\repsimple \SingAlg{p}$ by a certain completion. 
	The same caveat applies to $\algC \in \repwt \UQG{q}$.
We will treat this point carefully in the main text, but for the sake of the introduction we will gloss over it.}
Furthermore, one can think of $\TripAlg{p}$-modules as $\SingAlg{p}$-modules with extra structure. Indeed, one can even give an algebraic description of $\TripAlg{p}$-modules internal to $\repsimple \SingAlg{p}$. To present it, we need a little detour.

Let $\cat$ be a braided monoidal category and let $A \in \cat$ be an algebra, i.e.\ an object together with multiplication morphism $\mu\colon A \otimes A \to A$ and unit morphism $\eta \colon \one \to A$, subject to associativity and unit conditions. An $A$-module is an object $M \in \cat$ together with a morphism $\rho \colon A \otimes M \to M$, again subject to associativity and unit condition. 
We assume that $A$ is commutative, i.e.\ that $\mu \circ c_{A,A} = \mu$, where $c_{U,V}$ denotes the braiding in~$\cat$.
The key notion is this: an $A$-module $M$ is called \textsl{local} (or \textsl{dyslexic}) \cite{Pareigis:1995} if 
\be
	\rho \circ c_{M,A} \circ c_{A,M} ~=~ \rho \ .
\ee
In words, precomposing the action with the double-braiding has no effect.
	The full subcategory of local $A$-modules is a braided monoidal category, 
see \cite{Pareigis:1995} and Section~\ref{sec:local-modules} for more details.

We now return to the bottom wiggly line of \eqref{eq:intro-cat-diag}. It stands for: think of $\TripAlg{p}$ as a commutative algebra in $\repsimple \SingAlg{p}$; consider the braided monoidal category of its local modules; this is precisely $\rep \TripAlg{p}$. This is indeed known to be true for quite a general class of VOA extensions (see \cite{HKL,CKM} and Section~\ref{sec:singlet}), to which $\SingAlg{p} \subset \TripAlg{p}$ conjecturally belongs. 
The algebra $\TripAlg{p}$ has another important property, namely it allows for a non-degenerate invariant pairing, cf.\ Section~\ref{sec:corr-VOA-qHopf}.

\medskip

At this point we need to make no further conjectures to complete the diagram in~\eqref{eq:intro-cat-diag}. 
Let $\algC \in \repwt \UQG{q}$ be the preimage of $\TripAlg{p}$ under the ribbon equivalence in Conjecture~\ref{conj:sing-QG}. Then $\algC$ is again a commutative algebra. The detailed formulation of the conjecture involves stating the equivalence functor on objects,
in Section~\ref{sec:corr-VOA-qHopf},
 and so it produces the underlying object of $\algC$,
\be
	\algC = \bigoplus_{m \in \ZZ} \CC_{2pm} \ ,
\ee
where $\CC_{2pm}$ denotes the one-dimensional $\UQG{q}$-module where $E$ an $F$ act as zero, $K$ acts as the identity and $H$ by multiplication with $2pm$. Note that this is indeed a weight module as $q^H = q^{2pm} = 1 = K$.
We do not a priori know the product on $\algC$, but we show that up to isomorphism there is only one which is commutative and allows for a non-degenerate invariant pairing (Proposition~\ref{prop:comm-alg-unique}).
Commutative algebras which  -- like $\algC$ above --  are direct sums of invertible objects, are often called simple current extensions, cf.\ Remark~\ref{rem:Lamp-algebra}.
Our second main result is (Theorem~\ref{thm:equiv-1}\,(2)):

\begin{theorem}
There is a ribbon equivalence $\mathcal{F}
  \colon  
  \rep\Q \to
  \algC\text{-}\mathrm{mod}^\mathrm{loc}$, where $\algC\text{-}\mathrm{mod}^\mathrm{loc}$ denotes the category of local $\algC$-modules in $\repwt\UQG{q}$.\footnote{
	Here we are again glossing over some detail, as we must only consider finitely generated local modules. This is explained in the main text.}
\end{theorem}

This explains the top wiggly arrow and the arrow labelled $\mathcal{F}$ in \eqref{eq:intro-cat-diag}. In fact, we reverse-engineered the quasi-Hopf algebra $\Q$ based on the restricted quantum group $\UresSL2$ 
(recall the discussion after Conjecture~\ref{conj:trip-QG})
to make the equivalence $\mathcal{F}$ work.

Since the algebras $\algC$ and $\TripAlg{p}$ are related by a	(conjectural)
 ribbon equivalence, so are their categories of local modules. Hence the equivalence given by the right vertical arrow in \eqref{eq:intro-cat-diag} is indeed a consequence. 
But together with the equivalence $\mathcal{F}$, this is precisely the statement of Conjecture~\ref{conj:trip-QG}, so that this is actually a consequence of the other conjectures (and hence it is a corollary in the main text).

\subsubsection*{Outlook}

Here we collect a few direction for further investigation based on our results.

\medskip

\noindent
$\bullet$ 
The intersection of kernels of screening charges acting on the vacuum module 
 of lattice vertex algebras corresponding to rescaled root lattices of
Lie algebras $\mathfrak g$ give 
generalisations of the triplet vertex algebra, and -- after taking appropriate invariants -- of the singlet vertex algebra
\cite{FT, CM2, Lentner:2017dkg, Flandoli:2017}. 
Not much is known about these vertex algebras, but a natural expectation are correspondences to (unrolled) restricted quantum groups of $\mathfrak g$
(as described in \cite{GP,Lentner:2017oxe}).
It would be interesting to construct quasi-Hopf algebras for these, or for more general ``unrolled" Hopf-algebras~\cite{Andrus:2017}, along the lines presented in this paper.

\medskip

\noindent
$\bullet$
In~\cite{GLO} a construction is presented which is in a sense opposite to the
simple current extension we use here, and which applies to general simple Lie algebras $\mathfrak g$. 
In the $\mathfrak g = s\ell(2)$ case, one starts from a rank-one
lattice VOA whose representation category is equivalent to the category of  $\ZZ_{2p}$-graded vector spaces equipped with a certain abelian 3-cocycle (cf.\ Remark~\ref{rem:Lamp-algebra}), and considers ``the algebra of screenings"  -- a rank-one Nichols algebra object in the category of Yetter-Drinfeld modules over $\CC \ZZ_{2p}$.
 A generalisation of  the construction of quantum groups~\cite{ST,Lentner:2017dkg} from Nichols algebras to the quasi-Hopf setting produces then also our quasi-Hopf algebra $\Q$.
For other choices of $\mathfrak g$ one obtains in this way higher rank generalisations of $\Q$.
It would be good to further study the relation between this construction and  the conformal extensions used in our paper. 

\medskip

\noindent
$\bullet$ The representation category of a factorisable ribbon quasi-Hopf algebra is a factorisable finite tensor category (see e.g.\ \cite[Sec.\,4]{FGR1} for definitions) which is in addition ribbon. For such categories one obtains projective actions of surface mapping class groups on certain Hom-spaces \cite{Lyubashenko:1995,Lyubashenko:1994tm}. 
Concrete expressions for the projective action of the generators of the torus mapping class group $SL(2,\ZZ)$ have been given in terms of quasi-Hopf algebra data in \cite[Thm.\,8.1]{FGR1}.
 The action in this case is given on the center of the  algebra.
The families of quasi-Hopf algebras in \cite{FGR2} and in Theorem~\ref{thm:main1-intro} provide examples to which this formalism can be applied.
It would also be interesting to generalise the projective $SL(2,\ZZ)$-action on Hochschild cohomology
(where the center is in the zero degree)
 as developed for factorisable Hopf algebras in \cite{Lentner:2017f} to this quasi-Hopf setting.

\medskip
\noindent
$\bullet$ Verlinde's formula for rational VOAs says that normalised Hopf link invariants of the modular tensor category of the VOA coincide with the corresponding normalised entries of the modular $S$-matrix of torus one-point 
functions \cite{V, MS, Hu1, Hu2}. 
Generalisations of the Verlinde formula to the finitely non-semisimple setting have been discussed e.g.\ in \cite{FHST,Fuchs:2006nx,GnR2,CG}.
The proposal in \cite{CG, GnR2} for $C_2$-cofinite VOAs $V$ relies on a connection between the modular properties of pseudo-trace functions for $V$ and 
certain invariants in $\rep V$,
in particular \cite{GnR2,GnR3} relate it with Lyubashenko's modular group action on $\mathrm{End}$ of the identity functor.
The comparison of the two  $SL(2,\ZZ)$ actions has been done for symplectic fermions in \cite{FGR2} and the results of this paper allow one to test the proposal in the case of the triplet algebra, see also Remark~\ref{rem:p-odd}\,(2).

\medskip

\noindent
$\bullet$  Logarithmic knot invariants are invariants of knots coloured with projective representations of $\UresSL2$ \cite{MN, Mu}. They have been computed by constructing the projective modules as limits of modules of ${U}_{t} s\ell(2)$ as $t\rightarrow q$ \cite{Mu} and it turns out that the same invariants can also be obtained from projective modules of $\UQG{q}$ \cite{CMR}. Recently, in~\cite{BBG} a so-called modified trace together with Henning's type construction  was used to generalise the knot invariants to logarithmic invariants of links colored by projective $\UresSL2$ modules, and eventually to invariants of 3d manifolds with links inside
 and -- in \cite{DeRenzi:2017} --
further to a 3d TQFT
(for the braided case and with a modified cobordism category).
It is a natural task to try to build similarly a TQFT also from $\Q$.

\medskip

\noindent
$\bullet$
After investigating chiral conformal field theory 
	in terms of
a VOA $V$, one  would next like to understand how to combine the chiral halves into a full conformal field theory. 
If $V$ has finitely semisimple representation theory, 
so that $\rep V$ is a modular tensor category \cite{Hu2}, 
one can phrase this as an algebraic question in $\rep V$, see \cite{Runkel:2005qw} for an overview. 
The corresponding program for finitely non-semisimple $\rep V$, so-called logarithmic VOAs, is currently being developed, see \cite{Fuchs:2016wjr} and references therein.
The triplet VOA $\TripAlg{p}$ and symplectic fermions are prominent examples of such logarithmic VOAs and the families of quasi-Hopf algebras presented in this paper and in \cite{FGR2} provide
the necessary data in order to study the construction of
full logarithmic conformal field theories.

\subsubsection*{Organisation of the paper}

In Section~\ref{sec:local-modules} we start by giving our conventions for braided tensor categories and quasi-Hopf algebras. Then the definition of local modules and the associated braided monoidal category is reviewed, with an emphasis of finitely generated local modules. After setting up the formalism, we present in detail the relation between local modules of an algebra in $\CC$-graded vector spaces given by the group algebra $\CC\ZZ$, and a ribbon quasi-Hopf algebra with underlying algebra $\CC\ZZ_{2p}$ (Theorem~\ref{thm:Cartan-funF-H-br-equiv}). This will illustrate the more complicated construction in the case of quantum groups.

\smallskip

In Section~\ref{sec:unrolled-restricted} we introduce the starting point of our construction, the unrolled restricted quantum group $\UQG{q}$. We review its finite-dimensional simple and projective modules of weight type. Then we give the commutative algebra $\algC$ internal to $\repwt^\oplus \UQG{q}$, the category of at most countably-dimensional $\UQG{q}$-modules of weight type, and we study the finitely generated local $\algC$-modules (Proposition~\ref{prop:UHq-local-Lam-modules}).

\smallskip

Section~\ref{sec:qHopfUq} contains our main result. After reviewing the definition of the restricted quantum group $\UresSL2$, we give our quasi-Hopf algebra modification $\Q$, and we show that it is indeed a factorisable 
ribbon quasi-Hopf algebra, and
 that its category of finite-dimensional representations is ribbon-equivalent to the category of finitely generated local $\algC$-modules
 discussed in Section~\ref{sec:unrolled-restricted} (Theorem~\ref{thm:equiv-1}).
The quasi-Hopf algebra $\Q$
is $\ZZ_2$-graded and we finally investigate the orbit of its quasi-triangular quasi-Hopf algebra structures under the action of the third abelian cohomology for $\ZZ_2$ (which is isomorphic to~$\ZZ_4$).

\smallskip

In Section~\ref{sec:singlet} we describe the conjectural relation between our Hopf-algebraic constructions and VOAs. We briefly recall the definition of the singlet VOA $\SingAlg{p}$ and triplet VOA $\TripAlg{p}$, as well as some of their representations. 
We state the conjectural ribbon equivalence $\repsimple \SingAlg{p} \cong \repwt \UQG{q}$ and show how the triplet algebra is transported to $\algC$ from Section~\ref{sec:qHopfUq} along this equivalence. Thus, assuming this conjecture, for the triplet algebra we can conclude $\rep \TripAlg{p} \cong \rep\Q$ (Corollary~\ref{cor:main}).

\smallskip

Finally, Section~\ref{sec:proof} contains the proof of our main result, Theorem~\ref{thm:equiv-1}.

\smallskip

In Appendix~\ref{app:O}  we illustrate in an example how the modified coproduct in $\Q$ solves the problem of non-commutative tensor products one encounters for  $\UresSL2$.

\subsubsection*{Acknowledgements}

The authors thank
Christian Blanchet, 
J\"urgen Fuchs, 
Nathan Geer, 
Shashank Kanade, 
Simon Lentner, 
Antun Milas,
Tobias Ohrmann, 
Hubert Saleur,
Christoph Schweigert
	and
Alyosha Semikhatov
for helpful discussions.
We are grateful to
	Christian Blanchet,
	Simon Lentner
and	Antun Milas
for 
important
comments on a draft version of this paper.	
The authors are grateful to the organisers of ``The Mathematics of Conformal Field Theory" at ANU, Canberra in 2015 where part of this work  was done.
AMG thanks the 
Department of Mathematical and Statistical Sciences at the
University of Alberta for kind hospitality during a stay in 
February 2016.
AMG is supported by CNRS and thanks the Humboldt Foundation for a partial financial support.
TC is supported by NSERC discovery grant $\#$RES0020460.

\section{Local modules and quasi-Hopf algebras}
\label{sec:local-modules}

In this section we describe the notion of local modules for a commutative algebra in a braided monoidal category. 
	After setting out our conventions,
we review the general setup and then consider the example of $\mathbb{C}$-graded vector spaces in some detail. This example illustrates in a much simpler setting
the construction of the equivalence $\mathcal{F}$ in \eqref{eq:intro-cat-diag} carried out in Sections~\ref{sec:Lam-and-local-Uq} and \ref{sec:proof}. 
In essence, the example treats the (rank one) Cartan part 
	of $\UQG{q}$ and $\Q$. 
In the language of vertex operator algebras, it describes the extension of the rank-one Heisenberg VOA to a lattice VOA.

\subsection{Conventions}\label{sec:conventions}

In order to set out our notation, here we collect standard categorical and Hopf algebraic notions. For details and definitions see 
e.g.\ \cite{EGNO-book,ChPr,Yorck}. 

In a monoidal category $\cat$ with tensor unit $\one \in \cat$, the associator is a natural family of isomorphisms $\alpha_{U,V,W} : U\otimes (V\otimes W) \to (U\otimes V)\otimes W$.
If $\cat$ is braided, we denote the natural family of braiding isomorphisms by $c_{U,V} : U \otimes V \to V \otimes U$.
In a braided monoidal category $\mathcal{B}$, an object $T \in \mathcal{B}$ is called \textsl{transparent} if it has trivial monodromy with every other object of $\cat$, that is, if $c_{U,T} \circ c_{T,U} = \id_{T \otimes U}$ for all $U \in \mathcal{B}$.

A braided monoidal category $\mathcal{B}$ is called \textsl{balanced} \cite{Shum} if there is a natural isomorphism $\theta_U : U \to U$ of the identity functor on $\cat$, called \textsl{balancing isomorphism} or \textsl{twist}, such that for all $U,V \in \mathcal{B}$,
\be\label{eq:balancing-def}
	\theta_{U \otimes V} = (\theta_U \otimes \theta_V) \circ c_{V,U} \circ c_{U,V} 
\ee
The relation $\theta_{\one} = \id_{\one}$ is a consequence of this axiom (see e.g.\ \cite[Lem.\,XIV.3.3]{Kassel}).

Let $\cat$ be a monoidal category. We say that an object $U \in \cat$ has a \textsl{left dual} $U^*$ if there are \textsl{duality morphisms} $\mathrm{ev}_U : U^* \otimes U \to \one$ and $\mathrm{coev}_U : \one \to U \otimes U^*$ satisfying the zig-zag identities, see e.g.~\cite[Sec.\,2.10]{EGNO-book}. We think of $\mathrm{ev}_U$ and $\mathrm{coev}_U$ as part of the data defining a dual object. 
There is an analogous notion of a right dual. If every object of $\cat$ has a left (resp.\ a left and right) dual, we say that \textsl{$\cat$ has left duals} (resp.\ \textsl{$\cat$ has duals}).

Let $\mathcal{B}$ be a balanced braided monoidal category which has left duals. Then $\mathcal{B}$ is a \textsl{ribbon category} if the balancing isomorphism satisfies the following compatibility condition with the duals: for all $U\in\mathcal{B}$,
\be\label{eq:ribbontwist-def}
	(\theta_U)^* = \theta_{U^*} \ .
\ee
In this case we call $\theta$ a \textsl{ribbon twist}.
	A ribbon category automatically also has right duals.
We note that often in the literature, \eqref{eq:balancing-def} and \eqref{eq:ribbontwist-def} are assumed together, and the notion of ``balanced category'' is not separately introduced (and the notion is not
standard).
We will, however, deal with categories where not every object has a dual, but which are nonetheless balanced, hence the distinction.

In the category of vector spaces over some field (for us always $\CC$) we will denote the braiding by $\tau$ instead of $c$. The braiding is symmetric and is given by the flip of tensor factors, 
\begin{equation}\label{vecflip}
\flip_{X,Y}\colon X \tensor Y \longrightarrow Y \tensor X
\quad , \quad
\flip_{X,Y}(x \tensor y) =  y \tensor x \ ,
\end{equation}
where $X,Y$ are vector spaces. 

\medskip

A \textsl{quasi-bialgebra} $A$ over a field  is an associative unital algebra with product $\mu$, together with a counital coproduct $\Delta$, such that the coproduct is a coalgebra map, and it is coassociative up to conjugation with a coassociator $\Phi \in  A^{\otimes 3}$. A \textsl{quasi-Hopf algebra} $A$ is a quasi-bialgebra that is equipped with an antipode structure, that is, an algebra-anti-automorphism $S \colon A \to A$, an evaluation element $\Salpha$ and a coevaluation element $\Sbeta$.
A quasi-Hopf algebra $A$ is \textsl{quasi-triangular} if it is equipped with a universal $R$-matrix $R \in A^{\otimes 2}$, and it is \textsl{ribbon} if it is in addition equipped with a ribbon element $\ribbon \in A$.
For the list of conditions in the notation we use here, we refer to \cite[Sec.\,6]{FGR1}.
We stress that when working with a quasi-Hopf algebra $A$ we assume that the unit isomorphisms in $\rep A$ are those of the underlying vector spaces. This avoids having to introduce two additional invertible elements in $A$ and simplifies the conditions.

For a ribbon quasi-Hopf algebra $A$, the category of (possibly infinite dimensional) $A$-modules is a balanced braided tensor category with structure maps
\begin{align} 
\assoc_{M,N,K} & ~\colon~ & M\tensor(N\tensor K) &\longrightarrow (M\tensor N)\tensor K \gcg	&
m\tensor n\tensor k &\longmapsto \Phi. (m\tensor n\tensor k) \gc
\nonumber \hspace*{3em}~\\ 
  c_{M,N} & ~\colon~ &  M\tensor N &\longrightarrow N\tensor M 
  \gcg &
  m\tensor n &\longmapsto 
  \flip_{M,N}
  (R.(m\tensor n)) \gc  
  \nonumber \\
	\theta_M & ~\colon~ & M &\longrightarrow M  
	\gcg &
	m &\longmapsto \ribbon^{-1}.m  \ ,  \label{qHopf-cat-data}
\end{align} 
where $M,N,K$ are $A$-modules.
The category of finite-dimensional $A$-modules is a ribbon category.

A finite-dimensional quasi-triangular quasi-Hopf algebra $A$ is called \textsl{factorisable} if the element $\hat{\mathcal{D}} \in A \otimes A$ defined below provides a non-degenerate copairing \cite[Rem.\,6.6]{FGR1}. 
$\hat{\mathcal{D}}$ is given by
\be\label{eq:Q-for-fact}
\hat{\mathcal{D}} ~= 
\sum_{(X),(W)} S(W_3 X_2') W_4 X_2'' \otimes S(W_1 X_1') W_2 X_1''
\ ,
\ee
where 
$X \in A^{\otimes 2}$ and $W \in A^{\otimes 4}$ are defined as
\begin{align}
X &~=~ \sum_{(\Phi)} \Phi_1 \otimes \Phi_2 \Sbeta S(\Phi_3)\ ,
\nonumber \\ 
W&~=~(\one\tensor\Salpha\tensor\one\tensor\Salpha)\cdot
(\one\tensor\Phi^{-1})\cdot(\one\tensor M
\tensor\one)\cdot(\one\tensor\Phi)
\cdot(\id\tensor\id\tensor\Delta)(\Phi^{-1}) \ ,
\label{eq:Rem-X}
\end{align}
and $M = R_{21}R \in A\otimes A$ is the \textsl{monodromy matrix}. 
The equivalence to the original definition of factorisability for quasi-triangular quasi-Hopf algebras in~\cite{BT} is shown in~\cite[Cor.\,7.6]{FGR1}. 
For a Hopf algebra we have $\hat{\mathcal{D}} = \sum_{(M)} S(M_2) \otimes M_1$, so that we obtain the usual factorisability condition that $M$ is a 
non-degenerate copairing.

We note that factorisability can be defined directly for braided finite tensor categories in several equivalent ways, we refer to~\cite{Shimizu:2016} or e.g.~\cite[Sec.\,4]{FGR1} for details. By~\cite[Sec.\,7.3]{FGR1}, $A$ is factorisable iff $\repfd A$ is factorisable. For us the following equivalent characterisation (which follows from~\cite{Shimizu:2016}) will be important: 
$A$ is factorisable iff $\repfd A$ has no transparent objects apart from direct sums of the trivial representation.

\subsection{The category of finitely generated local modules}

Let $\mathcal{B}$ be a braided monoidal category and $A \in \mathcal{B}$ a commutative associative unital algebra (we will just say ``commutative algebra'' form now on). A \textsl{(left) $A$-module in $\mathcal{B}$} is an object $M \in \mathcal{B}$ together with a morphism $\rho_M \colon A \otimes M \to M$ satisfying an associativity and unit condition (see e.g.\ \cite{Fuchs:2001qc,Ostrik:2001} for details). The category of $A$-modules and $A$-module intertwiners in $\mathcal{B}$ will be denoted by ${}_A\mathcal{B}$.

Local modules (also called dyslectic modules) where introduced in \cite{Pareigis:1995} and further studied e.g.\ in \cite{KO,Frohlich:2003hm}. They are defined as follows.

\begin{definition}\label{def:local-module}
An $A$-module $M$ in $\mathcal{B}$ is called \textsl{local} if 
\be
	\rho_M \circ c_{M,A} \circ c_{A,M} ~=~ \rho_M \ .
\ee
We write ${}_A\mathcal{B}^\mathrm{loc}$ for
the full subcategory of local $A$-modules in ${}_A\mathcal{B}$. 
\end{definition}

In words, for a local module the monodromy endomorphism of $A \otimes M$ can be omitted when composed with the action map $A \otimes M \to M$. If $\mathcal{B}$ is symmetrically braided, every module is local.

In a braided category, a left module over a commutative algebra\footnote{ 
	If $A$ is not commutative, this procedure results in a right module over the corresponding version of the opposite algebra $A^{\mathrm{op}}$, i.e.\ the algebra  which has the multiplication of $A$ composed with the braiding or its inverse, respectively.
}
can be turned into a right module in two ways, using the braiding or the inverse braiding to define the right action \cite{Pareigis:1995}:
\be\label{eq:right-action-from-left-action}
	\rho_M^{r,\pm} : M \otimes A \to M \quad , \quad
	\rho_M^{r,+} = \rho_M \circ c_{M,A} ~~,~~~
	\rho_M^{r,-} = \rho_M \circ c_{A,M}^{-1} \ .
\ee
An equivalent way of phrasing Definition~\ref{def:local-module} is that these two right actions agree:
\be\label{eq:alt-local-def}
	\text{$M$ local} \quad \Leftrightarrow \quad
	\rho_M^{r,+} = \rho_M^{r,-} \ ,
\ee
providing thus a preferred way to turn left modules into bimodules. 

The tensor product of two $A$-$A$-bimodules $M,N$ is defined as the coequaliser
\be\label{eq:tensor-bimod-def}
\xymatrix{
 M \otimes (A \otimes N) 
 \ar@<-.5ex>[r] \ar@<.5ex>[r] & 
  M \otimes N \ar[r]^{\pi_\otimes} & M \otimes_A N
  }
   \ .
\ee
The initial bracketing is
not important,
 but for this choice the two arrows are the left action $\id_M \otimes \rho_N^l$
and the right action composed with the associator, $(\rho_M^r \otimes id_N) \circ \alpha_{M,A,N}$, respectively.

We will need that the tensor product $\otimes_A$ exists for all modules. One simple way to ensure this is to demand $\mathcal{B}$ to be abelian, which we will do from now on. 
(Another way is to take $A$ separable and $\mathcal{B}$ idempotent complete, see e.g.\ \cite{Frohlich:2003hm,Carqueville:2012dk}).

By \eqref{eq:right-action-from-left-action}, there are two ways to turn a left $A$-module (not necessarily local) into an $A$-$A$-bimodule, and hence  two ways to turn ${}_A\mathcal{B}$ into a monoidal category, see \cite{Pareigis:1995} (this procedure is also called $\alpha$-induction \cite{Bockenhauer:1999wt,Ostrik:2001}).
We choose the convention that the right action on a left $A$-module $M$ is $\rho^{r,+}_M$.
The tensor product functor of ${}_A\mathcal{B}$ is $\otimes_A$ and the tensor unit is $A$. The associativity and unit isomorphisms are induced from those of $\mathcal{B}$.

Induction provides a tensor functor
\be
	\mathrm{Ind} : \mathcal{B} \longrightarrow {}_A\mathcal{B}
	\quad, \quad X \mapsto A \otimes X ~~,~~ f \mapsto \id_A \otimes f \ ,
\ee
where the $A$-action on $A \otimes X$ is by left multiplication (see e.g.\ \cite[Sec.\,4]{Fuchs:2001qc} for details).

The next lemma shows that local modules form a full tensor subcategory (whose tensor product does not depend on the initial choice of $\rho^{r,+}_M$ over $\rho^{r,-}_M$ by \eqref{eq:alt-local-def}).

\begin{lemma}[{\cite[Lem.\,1.3\,\&\,Prop.\,2.4]{Pareigis:1995}}]
For local $A$-modules $M,N$, the tensor product $M \otimes_A N$ is again a local $A$-module, and its $A$-$A$-bimodule structure is induced from its left module structure via \eqref{eq:right-action-from-left-action}.
\end{lemma}

The key property of local modules is that the braiding descends to $\otimes_A$  \cite[Prop.\,2.2]{Pareigis:1995},
	i.e.\ there exists a unique $\tilde c_{M,N}$ such that the following diagram commutes:
\be\label{eq:braiding-descends}
\xymatrix{ 
M \otimes (A \otimes N)
\ar@<-.5ex>[r] \ar@<.5ex>[r] 
% \ar@<-.5ex>[d] \ar@<.5ex>[d]
	\ar[d]
& M \otimes N \ar[r]^{\pi_\otimes} \ar[d]^{c_{M,N}}
& M \otimes_A N
\ar@{-->}[d]^{\exists!\, \tilde c_{M,N}}
\\
N \otimes (A \otimes M)
\ar@<-.5ex>[r] \ar@<.5ex>[r] & N \otimes M \ar[r]^{\pi_\otimes} & N \otimes_A M
}
\ee
The precise braiding for $A$ in the
left-most vertical arrow turns out to not matter due to the locality condition. 
When verifying that the squares for the two maps entering the coequaliser commute (after composition with $\pi_\otimes$) one needs to make use of \eqref{eq:alt-local-def}, so that this diagram is specific to local modules.
One then proves:

\begin{proposition}[\cite{Pareigis:1995,KO}]
For an abelian braided monoidal category $\mathcal{B}$ and a commutative algebra $A \in \mathcal{B}$, ${}_A\mathcal{B}^\mathrm{loc}$ is a braided monoidal category.
\end{proposition}

Suppose the category $\mathcal{B}$ is in addition balanced with balancing isomorphism $\theta$.
As a direct consequence of Definition~\ref{def:local-module} and \eqref{eq:balancing-def}
we get

\begin{lemma}
Let $M$ be a local $A$-module.
If $\theta_A = \id_A$, then $\theta_M$ is an $A$-module intertwiner.
\end{lemma}

Hence in case $\theta_A = \id_A$, the balancing isomorphisms on local modules are morphisms in ${}_A\mathcal{B}^\mathrm{loc}$, and one checks that they turn ${}_A\mathcal{B}^\mathrm{loc}$ into a balanced braided monoidal category.

The next lemma addresses the relation between induction and local modules, see e.g.~\cite{CKM}.

\begin{lemma}\label{lem:local-vs-induction}
Let  $X \in \mathcal{B}$.
\begin{enumerate}
\item The induced module $\mathrm{Ind}(X)$ is local if and only if $X$ is transparent to $A$, i.e.\ $c_{X,A} \circ c_{A,X} = \id_{A \otimes X}$.
\item Denote by $\mathcal{B}' \subset \mathcal{B}$ the full subcategory of objects which are transparent to $A$. Then $\mathcal{B}'$ is braided monoidal and $\mathrm{Ind}\colon\mathcal{B}' \to {}_A\mathcal{B}^{\mathrm{loc}}$ is a braided monoidal functor. If $\theta_A=\id_A$, $\mathrm{Ind}$ is in addition balanced.
\end{enumerate}
\end{lemma}

\begin{remark}
A second reason why local modules are of special relevance is the theory of extensions of VOAs \cite{KO, HKL, CKM}. 
	It is shown in \cite{HKL}
that a VOA extension of a given VOA $V$ corresponds to a commutative, associative algebra $A$ with trivial twist in
 the vertex tensor category $\rep V$ 
	(see \cite{HKL,CKM} for the precise conditions on $V$ and on vertex tensor categories of grading-restricted generalised $V$-modules).
 Moreover, the vertex tensor category of
modules over the extended VOA  
  is braided equivalent to the category of local $A$-modules  \cite{CKM}. We will come back to this in two examples: in Section~\ref{sec:ex-Cgr}
 we will recall the well-known statement that lattice vertex algebras are extensions of the Heisenberg (or free boson) VOA and
in Section~\ref{sec:Triplet} we discuss the triplet VOA algebras $\TripAlg{p}$ as extensions of the singlet VOA $\SingAlg{p}$.
\end{remark}

In a braided category, if an object $X$ 
has
 a left dual $X^*$, that same object is also a right dual of $X$, and vice versa. Hence we just say ``$X$ has a dual''.

\begin{definition}
\label{def:fg-module}
A left $A$-module $M$ in $\mathcal{B}$ is called \textsl{finitely generated} if there is an object $X \in \mathcal{B}$ which has a dual, and a surjective $A$-module intertwiner from the induced module $\mathrm{Ind}(X)$ to $M$.
\end{definition}

This definition provides us with another full tensor subcategory, as the following lemma shows.

\begin{lemma}
If the tensor product functor of $\mathcal{B}$ is right exact, then the tensor product of finitely generated modules is again finitely generated.
\end{lemma}

\begin{proof}
Since the tensor product of $\mathcal{B}$ is right exact, so is $\otimes_A$ on ${}_A\mathcal{B}$. Let $M,N$ be finitely generated 
	from objects $X,Y \in \mathcal{B}$ with duals and
with surjections $f\colon A \otimes X \to M$ and $g\colon A \otimes Y \to N$. Then by right exactness,
\be
	A \otimes (X \otimes Y)
	\xrightarrow{\sim} (A \otimes X) \otimes_A (A \otimes Y)
	\xrightarrow{f \otimes_A g} M \otimes_A N 
\ee
is a surjection. Since also $X \otimes Y$ has a dual, this shows that $M \otimes_A N$ is finitely generated.
\end{proof}

\newcommand{\fgloc}[2]{{}_{#1}\mathcal{#2}^\mathrm{fg\text{-}loc}}

We denote the category of modules which are both, finitely generated and local, by $\fgloc AB$. 
{}From Lemma~\ref{lem:local-vs-induction} we have the following easy source of finitely generated local modules.

\begin{corollary}\label{cor:fglocal-vs-induction}
Let $X \in \mathcal{B}$ have a dual and be transparent to $A$. Then $\mathrm{Ind}(X) \in \fgloc AB$.
\end{corollary}

Note, however, that in general for a finitely generated local module we do \textsl{not} require the induced module $\mathrm{Ind}(X)$ covering it to be local.

We collect the above constructions in the following proposition.

\begin{proposition}\label{prop:B-fgloc-braidedmon}
Let $\mathcal{B}$ be an abelian braided monoidal category with right exact tensor product, and let $A \in \mathcal{B}$ be a commutative algebra. Then ${}_A\mathcal{B}^\mathrm{loc}$ and $\fgloc AB$ are braided monoidal categories. If $\mathcal{B}$ is balanced, so are ${}_A\mathcal{B}^\mathrm{loc}$ and $\fgloc AB$.
\end{proposition}

\begin{remark}
In the setting of Proposition~\ref{prop:B-fgloc-braidedmon},
$\fgloc AB$ is in general neither rigid nor abelian. For example, take $\mathcal{B}$ to be the symmetric monoidal category of $k$-vector spaces.
Recall that for symmetric categories, every module is local. For $A = k[X]$, the $A$-module $k$ is finitely generated but has no dual (indeed, $X$ acts trivially on $M \otimes_A k$ for any $A$-module $M$, and hence $\mathrm{Hom}_A(M \otimes_A k,A) = \{0\}$). 
Thus in this case, $\fgloc AB$ is not rigid.
Or take $A$ to be the polynomial ring with a countable number of generators, $A = k[X_1,X_2,\dots]$. Clearly $A$ is finitely generated as an $A$-module. But the kernel of the $A$-module map $f\colon A \to k$, $X_i \mapsto 0$ is the ideal generated by all $X_i$, and hence is not finitely generated 
	over~$A$.
Thus for this choice of $A$, $\fgloc {A}B$ is not abelian.
\\
Under suitable extra conditions on $\mathcal{B}$ and $A$ one can show that $\fgloc AB$ is abelian and rigid. In particular, $A$ should satisfy the descending chain condition in ${}_A\mathcal{B}$ and the tensor product $\tensor_A$ should be exact on 
	${}_A\mathcal{B}^\mathrm{fg}$. 
We will not pursue this here as rigidity will be evident in the concrete cases we consider.
\end{remark}

After these general preparations, let us turn to the example which illustrates the construction of the equivalence $\mathcal{F}$ in \eqref{eq:intro-cat-diag}. We first describe the relevant category and algebra (Section~\ref{sec:ex-Cgr}) and the equivalence to the representations of a quasi-Hopf algebra (Section~\ref{sec:ex-qHopf}).

\subsection{Example: $\mathbb{C}$-graded vector spaces}\label{sec:ex-Cgr}

We denote by $\mathcal{H}^\mathrm{fd}$ the category of finite-dimensional $\mathbb{C}$-graded complex vector spaces, and by $\mathcal{H}^\oplus$ its completion with respect to countable direct sums. Equivalently, $\mathcal{H}^\oplus$ consists of $\mathbb{C}$-graded vector spaces of overall finite or countable dimension.\footnote{
	We could just take all $\mathbb{C}$-graded vector spaces, but we prefer $\mathcal{H}^\oplus$ as this is the smallest setting in which the example works. Another natural looking choice is to take $\mathbb{C}$-graded vector spaces with finite-dimensional graded components. However, this will not work for us as it does not close under tensor products.}
	
Given $V \in \mathcal{H}^\oplus$, we write $V_\alpha$ for the graded component of degree $\alpha \in \mathbb{C}$, so that $V = \bigoplus_{\alpha \in \mathbb{C}} V_\alpha$. By our assumption on the overall dimension of $V$, at most countably many components $V_\alpha$ can be non-zero.
Denote by $H_V\colon V \to V$ the degree map on $V$, that is, $H_V|_{V_\alpha} = \alpha \, \id_{V_\alpha}$. The collection $\{H_V\}_{V \in \mathcal{H}^\oplus}$ is a natural transformation of the identity functor on $\mathcal{H}^\oplus$.
The category $\mathcal{H}^\oplus$ is semisimple and we denote by $\mathbb{C}_\alpha$ the simple object of 
$\mathcal{H}^\oplus$
 which is $\mathbb{C}$ in degree $\alpha$ and zero else; every simple object is isomorphic to exactly one of these. 

\medskip

For any abelian group $\Gamma$, braided monoidal structures on the category of $\Gamma$-graded vector spaces over a field $k$ are parametrised (up to braided equivalences whose underlying functor is the identity) by the third abelian group cohomology $H_\mathrm{ab}^3(\Gamma,k^\times)$, see \cite{Joyal:1993} and \cite[Sec.\,8.4]{EGNO-book}. The cohomology group $H_\mathrm{ab}^3(\Gamma,k^\times)$
in turn is canonically isomorphic to the group of
 quadratic forms on $\Gamma$ (\cite{Eilenberg:1954} and \cite[Thm.\,8.4.9]{EGNO-book}). Recall that a quadratic form is a function $\kappa : \Gamma \to k^\times$ such that $\kappa(a) = \kappa(-a)$, and such that the function $\rho : \Gamma \times \Gamma \to k^\times$, $(a,b) \mapsto \kappa(a+b) \kappa(a)^{-1} \kappa(b)^{-1}$ is bilinear (it is called the \textsl{associated bilinear form}). 

In our example, $\Gamma = \mathbb{C}$ and we choose the quadratic form 
\be\label{eq:C-graded-qf-choice}
a \mapsto \exp(\tfrac{\pi i r}2 a^2) \ , 
\ee
where $r \in \mathbb{C}^\times$ is a normalisation constant. The factor $\tfrac{\pi i r}2$ is there to match the convention for $\UQG{q}$ in \eqref{eq:R-H} below, in which case $r = 1/p$.
We will see in Proposition~\ref{prop:all-cont-braidings} below that all choices of $r \in \mathbb{C}^\times$ lead to equivalent braided monoidal categories.
The associated bilinear form is $\rho(a,b) = e^{\pi i r ab}$. 
The resulting braided monoidal structure on $\mathcal{H}^\oplus$ can be described as follows. The associator is just the standard one of graded vector spaces. The braiding isomorphism is
\be\label{eq:C-graded-vec-braiding}
	c_{U,V} = \tau_{U,V} \circ e^{\frac{\pi i r}2 H_U \otimes H_V} \ ,
\ee
where $\tau_{U,V}$ is the flip of tensor factors from \eqref{vecflip}.
On an element $u \otimes v \in U_\alpha \otimes V_\beta$ the braiding isomorphism acts as $c_{U,V}(u \otimes v) = e^{\frac{\pi i r}2\alpha\beta} v \otimes u$.
{}From the explicit form of the braiding it is immediate that the transparent objects in $\mathcal{H}^\oplus$ are precisely those $\mathbb{C}$-graded vector spaces concentrated in degree $0$.

The braided monoidal category $\mathcal{H}^\oplus$ is balanced with balancing isomorphism
\be\label{eq:C-graded-vec-twist}
	\theta_U = e^{\frac{\pi i r}2 (H_U)^2} \ ,
\ee
so that for $u \in U_\alpha$ we have $\theta_U(u) = e^{\frac{\pi i r}2\alpha^2} u$.
The full subcategory $\mathcal{H}^\mathrm{fd}$ of $\mathcal{H}^\oplus$ is rigid, while $\mathcal{H}^\oplus$ itself is not. The dual of $V$ is given by $V^*$ with degree map $-(H_V)^*$. That is, the dual of a vector space concentrated in degree $\alpha$ is concentrated in degree $-\alpha$. With the above choice of balancing, $\mathcal{H}^\mathrm{fd}$ becomes a ribbon category.

\medskip

There are infinitely many quadratic forms on $\mathbb{C}$ which are not of the form \eqref{eq:C-graded-qf-choice}, for example one can think of $\mathbb{C}$ as a $\mathbb{Q}$-vector space and choose the diagonal part of an arbitrary symmetric bilinear form on this uncountably infinite dimensional $\mathbb{Q}$-vector space. 
Still, under some natural assumptions our choice of braiding on $\mathcal{H}^\oplus$ is the generic one, as we now illustrate.

Given a braided monoidal structure on $\mathcal{H}^\oplus$, 
the corresponding quadratic form 
$\kappa\colon \mathbb{C} \to \mathbb{C}^\times$ is extracted from the self-braiding of simple objects: 
$c_{\mathbb{C}_\alpha,\mathbb{C}_\alpha} = \kappa(\alpha) \, \id_{\mathbb{C}_\alpha \otimes \mathbb{C}_\alpha}$. 
We stress that in all our considerations the tensor product functor on $\mathcal{H}^\oplus$ remains fixed.
We have:

\begin{proposition}\label{prop:all-cont-braidings}
Let $\mathcal{B}$ be equal to $\mathcal{H}^\oplus$ equipped with some braided monoidal structure and let $\kappa_{\mathcal{B}}$ be the corresponding quadratic form. If $\kappa_{\mathcal{B}}$ is continuous, 
then 
\be\label{eq:continous-quadratic-forms}
	\kappa_{\mathcal{B}}(z) = e^{\zeta z^2 + \chi z\bar z + \xi \bar z^2} \ ,
\ee
for an appropriate choice of $\zeta,\chi,\xi \in \mathbb{C}$. 
If $\chi=\xi=0$ and $\zeta \neq 0$, then
$\mathcal{B}$ is braided monoidally equivalent to $\mathcal{H}^\oplus$ equipped with the braided monoidal structure for the quadratic form $\kappa(z) = e^{z^2}$, i.e.\ we may set $\zeta=1$.
\end{proposition}

\begin{proof}
Let $\rho\colon \mathbb{C} \times \mathbb{C} \to \mathbb{C}^\times$ be the  (symmetric) bilinear form associated to $\kappa_{\mathcal{B}}$. By assumption, $\rho$ is continuous on $\mathbb{C} \times \mathbb{C}$. 
As $\mathbb{C} \times \mathbb{C}$ is simply connected we can find a unique continuous lift to the universal cover $\mathbb{C}$ of $\mathbb{C}^\times$ (i.e.\ a logarithm) $\tilde\rho \colon \mathbb{C} \times \mathbb{C} \to \mathbb{C}$. The map $\tilde\rho$ is again bilinear 
	(i.e.\ a group homomorphism for $(\CC,+)$ in each argument), and 
it is determined on rational points by $\tilde\rho(1,1)$, $\tilde\rho(1,i)$ and $\tilde\rho(i,i)$. 
By continuity one finds, for $a,b,c,d \in \mathbb{R}$,
\be
\tilde\rho(a+ib,c+id) = rac + s(ad+bc) + tbd\gc 
\ee
 with $r,s,t \in \mathbb{C}$ given by
  $r = \tilde\rho(1,1)$, $s = \tilde\rho(1,i) = \tilde\rho(i,1)$, $t = \tilde\rho(i,i)$. 
Correspondingly, for $\rho$ itself we get 
\be
\rho(a+ib,c+id) = e^{rac + s(ad+bc) + tbd}\gp
\ee
The quadratic form is then uniquely fixed to be $\kappa_\mathcal{B}(a+ib) = e^{\frac12(ra^2 + 2sab + tb^2)}$.
This is precisely \eqref{eq:continous-quadratic-forms} if one sets 
$z = a + ib$, 
$\zeta = \frac18r - \frac18t - \frac{i}4s$,
$\chi = \frac14r + \frac14t$,
$\xi = \frac18r - \frac18t + \frac{i}4s$.

	It remains to show that if $\chi=\xi=0$, up to braided equivalence we can set $\zeta=1$.
	Let $F$ be an invertible $\mathbb{R}$-linear endomorphism of $\mathbb{C}$ (not necessarily an algebra map) and define the quadratic form $\tilde \kappa(z) := \kappa_\mathcal{B}(F(z))$.
Consider the $\mathbb{C}$-linear endofunctor $\mathcal{F}$ of $\mathcal{H}^\oplus$ given on objects by $\mathcal{F}(V) = \bigoplus_\alpha \mathcal{F}(V)_\alpha$ with $\mathcal{F}(V)_\alpha := V_{F(\alpha)}$. On morphisms, $\mathcal{F}$ acts as the identity. 
Write $\mathcal{C}$ for the braided monoidal category obtained from $\mathcal{H}^\oplus$ with quadratic form $\tilde\kappa$. The functor $\mathcal{F}$ transports the self-braiding from $\mathcal{C}$ to that of $\mathcal{B}$
	(note that this is independent of the choice of coherence isomorphisms for $\mathcal F$):
\be
	\mathcal{F}(c^{\mathcal{C}}_{\mathbb{C}_\alpha,\mathbb{C}_\alpha})
	= \tilde\kappa(\alpha) 
	\id_{\mathbb{C}_{F(\alpha)} \otimes \mathbb{C}_{F(\alpha)}}
	\quad , \quad
	c^{\mathcal{B}}_{\mathcal{F}(\mathbb{C}_\alpha),\mathcal{F}(\mathbb{C}_\alpha)}
	=
	\kappa_{\mathcal{B}}(F(\alpha)) \,
	\id_{\mathbb{C}_{F(\alpha)} \otimes \mathbb{C}_{F(\alpha)}}
	\ ,
\ee
so that by construction the two expressions are equal. 
Since the quadratic form determines the braided monoidal structure up to a braided monoidal isomorphism whose underlying functor is the identity, it is  possible to find coherence isomorphisms for $\mathcal{F}$ such that it becomes a braided monoidal functor $\mathcal{C} \to \mathcal{B}$.

If $\chi=\xi=0$, we can absorb the factor $\zeta$ into $F$ by choosing it to be an appropriate rescaling and rotation. That is, we can find $F$ such that $\kappa_{\mathcal B}(F(z)) = e^{z^2}$, and by the above argument we obtain the braided monoidal equivalence claimed in the proposition.
\end{proof}

The above proposition shows in particular that up to braided monoidal equivalence
there is a \textsl{unique} braided monoidal structure on $\mathcal{H}^\oplus$ for which 
\begin{enumerate}
\item the braiding is non-symmetric,  
\item the self-braiding of simple objects depends holomorphically on
 their degree.
\end{enumerate}
This is the braided monoidal structure we have chosen on $\mathcal{H}^\oplus$ in \eqref{eq:C-graded-vec-braiding} (in a normalisation convenient to us -- all choices of $r \in \mathbb{C}^\times$ give equivalent braided monoidal categories).

\begin{remark}\label{rem:HeisenbergVOA}~
\begin{enumerate}\setlength{\leftskip}{-1em}
\item
Denote by $\mathsf{H}$ the rank one Heisenberg VOA. Another name for $\mathsf{H}$ is the affine VOA of the Lie algebra $\mathfrak{gl}(1)$ at non-zero real level $k$ (one can fix the level to any non-zero value as one can always rescale). 
By $\rep \mathsf{H}$ we denote the category of Fock modules, 
i.e.\ countable direct sums of highest-weight modules of $\widehat{\mathfrak{gl}}(1)$ of level $k$.
 Inequivalent simple modules are exactly the Fock modules $\mathcal F_\lambda$
	of weight $\lambda \in \CC$ \cite{LW}.
$\rep \mathsf{H}$ is thus $\CC$-linear equivalent to $\mathcal{H}^\oplus$. Moreover, 
	their tensor products are just
$\mathcal F_\lambda \boxtimes_{\mathsf{H}} \mathcal F_\mu \cong \mathcal F_{\lambda+\mu}$
	\cite{DL}, see also the proof of \cite[Thm.\,2.3]{CKLR}.
Applicability of the vertex tensor theory of \cite{HLZ} requires the conformal weight to be real and so let Rep$(\mathsf{H})_\mathbb R$ be the subcategory of Rep$(\mathsf{H})$ whose simple objects are the Fock modules $\mathcal F_\lambda$ of real weight $\lambda$. It is a vertex tensor category \cite[Thm.\,2.3]{CKLR}. The operator product algebra of intertwining operators of Fock modules is analytic in the weight labels (see e.g.~\cite[Eq.\,(5.2.9)]{FB}) and so especially the
 braiding  depends analytically on the weight $\lambda$ of a Fock module $\mathcal F_\lambda$. It is a computation to verify that the self-braiding 
is just 
$e^{\pi i \lambda^2/k} \, \id_{\mathcal F_{\lambda}\boxtimes_{\mathsf H} \mathcal F_{\lambda}}$. 
This tells us that Rep$(\mathsf{H})_\mathbb R$ is braided equivalent to the full tensor subcategory of $\mathcal{H}^\oplus$ of $\RR$-graded vector spaces with $r$ in \eqref{eq:C-graded-qf-choice} chosen to be
	$r = 2/k \in \mathbb{R}$.
\item
	Demanding a vector space to be $\CC$-graded, as we do, is equivalent to requiring the degree map $H$ to act semisimply. 
The action of $H$ corresponds to the action of the zero-mode of the Heisenberg field of the Heisenberg VOA $\mathsf{H}$. However, $\mathsf{H}$ also allows for self-extensions of Fock modules on which this zero-mode does not act semisimply. There are even extensions of infinite Jordan-H\"older length of Fock modules. As this is rather unpleasant, one usually restricts to $\mathsf{H}$-modules that carry a semi-simple action of the zero-mode of the Heisenberg vertex algebra $\mathsf{H}$. 
\end{enumerate}
\end{remark}

\medskip

Now we turn to algebras $A$ in $\mathcal{H}^\oplus$. We need the following notion:
a \textsl{non-degenerate invariant pairing} on an algebra $A$ in a monoidal category is a morphism 
$\varpi : A \otimes A \to \one$ 
such that
\begin{enumerate}
\item \textsl{(invariance)} we have the identity
\begin{align}
&\big[
A \otimes (A \otimes A) \xrightarrow{\id \otimes \mu} A \otimes A \xrightarrow{\varpi} \one \big]
\nonumber\\
&=
\big[
A \otimes (A \otimes A) \xrightarrow{\sim}
(A \otimes A) \otimes A \xrightarrow{\mu \otimes \id} A \otimes A \xrightarrow{\varpi} \one \big] \ ,
\label{eq:cat-pairing-inv}
\end{align}
\item \textsl{(non-degeneracy)}
for all objects $U$ and morphisms $f\colon U \to A$ the equalities
\be\label{eq:cat-pairing-nondeg}
	\varpi \circ (f \otimes \id_A) = 0
	\quad \text{or} \quad
	\varpi \circ (\id_A \otimes f) = 0
\ee
imply that $f=0$. 
\end{enumerate}
Condition 2 may look a bit opaque. 
In the case of $\mathcal{H}^\oplus$, $\varpi$ is in particular a linear map $A \otimes A \to \mathbb{C}$, and condition 2 just means that it is non-degenerate as a pairing of vector spaces: $\varpi(a,b)=0$ for all $a$ implies $b=0$ and vice versa.
If $A \in \mathcal{H}^\mathrm{fd}$, conditions 1 and 2 are equivalent to $A$ being a Frobenius algebra. However, we will specifically be interested in algebras $A$ that are not finite-dimensional.

Since finite-dimensional (ungraded) vector spaces form
 a subcategory of $\mathcal{H}^\oplus$, one cannot hope to classify all algebras in $\mathcal{H}^\oplus$. But one can classify all algebras subject to the following three conditions:
\begin{enumerate}
\item $A$ is commutative,
\item $A$ does allow for a non-degenerate invariant pairing,
\item the graded components of $A$ have dimension 0 or 1.
\end{enumerate}
It is easy to write examples of such algebras. 
Fix $p \in \mathbb{Z} \setminus \{0\}$, 
abbreviate 
$\gamma = 2 \sqrt{p/r}$ 
(the choice of square root is immaterial) and set 
\be\label{eq:free-boson-extension-alg}
\Lambda_p := \bigoplus_{m \in \mathbb{Z}} \mathbb{C}_{m \gamma} \ .
\ee
Write $1_\alpha$ for the element $1 \in \mathbb{C}_\alpha$.
Then $\Lambda$ has unit $\eta(1) = 1_0$ and product 
\be\label{eq:Lam_p-product}
	\mu(1_{m\gamma} \otimes 1_{n\gamma}) = 1_{(m+n)\gamma} \ .
\ee
Associativity and unitality are clear, and commutativity amounts to the observation that $\exp(\frac{\pi i r}2  m n \gamma^2) = 1$. 
	The non-degenerate invariant pairing is given by $\mu$ composed with the projection to grade 0.
Note that $p$ and $-p$ give two rank one lattices in $\mathbb{C}$ which are rotated by $i$ relative to each other.

\begin{remark}\label{rem:Lamp-algebra}~
\begin{enumerate}\setlength{\leftskip}{-1em}
\item
The algebras $\Lambda_p$ (for $p>0$ but without the commutativity requirement, so without $p$ having to be an integer) were also considered in \cite{Buecher:2012ma}, where the monoidal category of their bimodules was determined. 
\item
Algebras $A$ in braided monoidal categories of the form $A = \bigoplus_{g \in G} L_g$, where $G$ is a finite group and $L_g$ are invertible simple objects whose tensor product agrees with the product in $G$, have been studied in \cite{tft3}. 
These algebras originate from so-called simple current extensions in conformal field theory \cite{Schellekens:1989am,Intriligator:1989zw,Schellekens:1990xy}
	and have  been studied in the VOA setting e.g.\ in \cite{DLM, Y, CKL}.
\item 
We have discussed in Remark~\ref{rem:HeisenbergVOA}\,(1)  that Rep$(\mathsf{H})_\RR$ is braided equivalent to the full subcategory of $\RR$-graded vector spaces in $\mathcal{H}^\oplus$. 
The algebra object $\Lambda_p$ is thus the simplest example of an infinite order simple current extension and the extended vertex algebra is the lattice VOA $V_{\sqrt{2p}\ZZ}$ of the lattice $\sqrt{2p}\ZZ$, 
see e.g.\ \cite{D,DLM} and \cite{AR} for details.
We note that $p<0$ gives a sensible vertex algebra; the literature however seems to prefer lattices of positive definite signature. 
\end{enumerate}
\end{remark}

It turns out that with one exception\footnote{
	The exception is just notational. One can include the case $p=0$ by setting $S = 2 \sqrt{p/r} \mathbb{Z}$ and $\Lambda_p = \bigoplus_{\alpha \in S} \mathbb{C}_\alpha$. Then for $p=0$, $S$ is the one-element set $\{0\}$.
}, the algebras $\Lambda_p$ provide all examples of algebras in $\mathcal{H}^\oplus$ which satisfy 
conditions 1--3 above:

\begin{proposition}\label{prop:comm-alg-unique}
Any algebra $A \in \mathcal{H}^\oplus$ that satisfies
conditions 1--3 
is isomorphic to either $\mathbb{C}_0$ or to $\Lambda_p$ for some $p \in \mathbb{Z} \setminus \{0\}$.
\end{proposition}

This proposition can be shown by adapting the statements and proofs in \cite{tft3} from finite groups and Frobenius algebras to the present situation. Alternatively, one can give a straightforward direct proof, which we do in Appendix~\ref{app:proof-Prop}.
  Note that the proposition is false if we do not demand existence of a non-degenerate invariant pairing. For example, one could then just take the sum over $m \in \mathbb{Z}_{\ge 0}$ in \eqref{eq:free-boson-extension-alg}.

Next we look at local $\Lambda_p$-modules. We have the following explicit description:

\begin{proposition}\label{prop:Lambda-loc-mod}
Let $p \in \mathbb{Z} \setminus \{0\}$. A $\Lambda_p$-module is local if and only if the underlying $\mathbb{C}$-graded vector space is non-zero only for degrees in $\frac{1}{\sqrt{rp}} \mathbb{Z}$.
A local $\Lambda_p$-module is finitely generated if and only if all its graded components are finite dimensional.
\end{proposition}

We remark that for a non-local $\Lambda_p$-module,  finite-dimensionality of its eigenspaces does not imply that it is finitely generated.
We also note that there are local $\Lambda_p$-modules which are not finitely generated, e.g.\ those induced from $U\notin \mathcal{H}^\mathrm{fd}$.

In what follows, we will abbreviate $\Lambda = \Lambda_p$.

\begin{proof}
As in the proof of Proposition~\ref{prop:comm-alg-unique} we set $r=1$. 

Let $M$ be a local $\Lambda$-module with action $\rho\colon \Lambda \otimes M \to M$. For $m \in \mathbb{Z}$ consider the linear maps
$\rho(1_{\gamma m}) := \rho(1_{m\gamma} \otimes -) \colon M \to M$. On the 
graded components, these act as $\rho(1_{\gamma m}) \colon M_\alpha \to M_{\alpha + m \gamma}$ (and so these are not morphisms in $\mathcal{H}^\oplus$). The action property can be stated as, for all $m,n \in\mathbb{Z}$.
\be
	\rho(1_{0}) = \id_M \quad , 
	\quad
	\rho(1_{\gamma m}) \circ \rho(1_{\gamma n}) = \rho(1_{\gamma (m+n)}) \ .
\ee
In particular, all $\rho(1_{\gamma m})$ are isomorphisms.

The locality property $\rho \circ c_{M,\Lambda} \circ c_{\Lambda,M} = \rho$, evaluated on $1_{m \gamma} \otimes x_\alpha$ for some $m \in \mathbb{Z}$ and $x_\alpha \in M_\alpha$, reads
\be
	e^{\pi i m \gamma \alpha} \rho(1_{\gamma m})(x_\alpha) = \rho(1_{\gamma m})(x_\alpha) \ .
\ee
Since $\rho(1_{\gamma m})$
is an isomorphism and $x_\alpha$ was arbitrary, we must have $e^{\pi i m \gamma \alpha} = 1$ for all $m \in \mathbb{Z}$ and $\alpha \in \mathbb{C}$ such that $M_\alpha \neq \{0\}$. Equivalently, we need $\alpha 2 \sqrt{p} \in 2 \mathbb{Z}$, showing the first statement.

For the second statement, first note that for $V \in \mathcal{H}^\oplus$ finite dimensional, $\Lambda \otimes V$ has finite-dimensional graded components. 
Conversely, pick a 	set $\RS$ 
of representatives of the cosets in 
$\frac{1}{\sqrt{p}} \mathbb{Z} / 2\sqrt{p} \mathbb{Z} 
	\cong \mathbb{Z}_{2p}$. 
Then $\RS$ has $2p$ elements, and for every homogeneous $u \in M$ one can find $\alpha \in \RS$, $x \in M_\alpha$ and $m \in \mathbb{Z}$ such that $\rho(1_{\gamma m})(x) = u$, 
	i.e.\ we get a surjection $\rho : \mathrm{Ind}(\bigoplus_{\alpha \in \RS} M_{\alpha}) \to M$ of $\Lambda$-modules.
Since $\bigoplus_{\alpha \in \RS} M_\alpha$ is finite-dimensional by assumption, this shows that $M$ is finitely generated.
\end{proof}

It is now straightforward to write down the braided monoidal category  
\be\label{eq:C-gr-example-Mdef}
\mathcal{M} := \fgloc{\Lambda_p}{(\mathcal{H}^\oplus)}
\ee
of finitely generated local modules directly. 
One first convinces oneself that every $M \in \mathcal{M}$ is isomorphic to an induced module $\mathrm{Ind}(U)$ for some (non-unique) 
$U \in \mathcal{H}^\mathrm{fd}$. 
In particular, $\mathcal{M}$ is again semisimple and since $\mathrm{Ind}$ is a tensor functor 
	(cf.\ Lemma~\ref{lem:local-vs-induction}),
$\mathcal{M}$ is pointed. The relevant group is the quotient of $\frac{1}{\sqrt{rp}} \mathbb{Z}$ by $2\sqrt{p/r} \mathbb{Z}$, i.e.\ $\mathbb{Z}_{2p}$. 
The corresponding quadratic form is read off from that of $\mathcal{H}^\mathrm{fd}$
using~\eqref{eq:C-graded-qf-choice}.
 One finds that, 
for $a \in \mathbb{Z}_{2p}$, $\kappa(a) = \exp(\frac{\pi i }{2p} a^2 )$.
Note that this is independent of $r$.
The ribbon category $\mathcal{M}$ is equivalent to the ribbon category of modules over the lattice VOA from Remark~\ref{rem:Lamp-algebra}\,(3), see also Theorem~\ref{thm:VOA-alg-in-cat-corr} below.

Now comes the simple but key observation of this example. 
The quadratic form $\kappa$ does \textsl{not} arise as the diagonal of a bilinear form on $\mathbb{Z}_{2p}$. 
Consequently, $\mathcal{M}$ is \textsl{not} monoidally equivalent to $\mathbb{Z}_{2p}$-graded vector spaces with trivial associator, even
though the $\mathbb{C}$-graded vector spaces we started from had a trivial associator.
Explicit choices of associator and braiding for the quadratic form $\kappa$ have been given for example in 
\cite[Sec.\,4.2]{Brunner:2000wx}, \cite[Sec.\,2.5.1]{tft1} or
\cite[App.\,B.1]{Fuchs:2007tx}.

\begin{remark}
The above discussion generalises to algebras in 
$\mathbb{C}^n$-graded vector spaces for some $n \in \mathbb{Z}_{\ge 1}$ which satisfy the conditions 1--3. We only need the rank one case, and so we can avoid some difficulties occurring for higher rank, such as the appearance of signs in the product of the corresponding algebras. 
Namely, for higher rank it is no longer possible to set all structure constants $t_{a,b}$ to $1$ as we do in the proof of 
Proposition~\ref{prop:comm-alg-unique},
 but one can at best achieve $t_{a,b} = \pm 1$. 
See e.g.\ \cite[Sec.\,2.2]{Sch} for a discussion of these signs in the VOA context. 
\end{remark}

\subsection{Quasi-Hopf algebra for local $\algC_p$-modules}
\label{sec:ex-qHopf}

For this section we fix $p \in \mathbb{Z}_{>0}$ 
and we set the parameter $r$ in \eqref{eq:C-graded-vec-braiding} to be 
\be \label{eq:r=1/p-convention}
r = \tfrac 1p \ . 
\ee
This will be the choice relevant for Sections~\ref{sec:unrolled-restricted} and~\ref{sec:qHopfUq}.
Recall from Proposition~\ref{prop:all-cont-braidings} that different values of $r$ give equivalent braided monoidal categories 
$\mathcal{H}^\oplus$.

{}From the above discussion we see that $\mathcal{M}$ as defined in \eqref{eq:C-gr-example-Mdef} is braided monoidally equivalent to modules over the quasi-Hopf algebra
$\CC \ZZ_{2p}^{\omega}$ -- the group algebra of $\ZZ_{2p}$ equipped with the co-associator and $R$-matrix described by an abelian 3-cocycle $\omega$. The class $[\omega]\in H_\mathrm{ab}^3(\ZZ_{2p},\CC^\times)$ corresponds to the quadratic form  $\kappa(a) = \exp(\frac{\pi i }{2p} a^2 )$ discussed above.

\begin{definition} \label{def:transport}
Let $\cat,\catD$ be categories with tensor product 
functors $\tensor_\cat,\tensor_\catD$ and tensor units $\one_\cat,\one_\catD$. A functor $\fun\colon \cat\to\catD$ is called \emph{multiplicative}
if there exists a family of natural isomorphisms 
$\fun_{U,V}\colon \fun(U)\tensor_{\catD} \fun(V) \to \fun(U\tensor_{\cat} V)$ 
and an isomorphism $\fun_{\one}\colon \one \to \fun(\one)$. 
\end{definition}

If $\cat$ and $\catD$ in the above definition are monoidal categories (i.e.\ equipped with associator and unit isomorphisms
	satisfying the relevant coherence conditions), 
a multiplicative functor is also called a \textsl{quasi-tensor functor}, see e.g.\ \cite[Def.\,4.2.5]{EGNO-book},
i.e.\ one does not demand the
	compatibility
conditions with the associator and unit isomorphisms from the definition of the tensor functor. 

\medskip

In this section we illustrate the method 
we use below
to find the coassociator and braiding for the restricted 
quasi-quantum group 
$\Q$ (where $\CC \ZZ_{2p}$ plays the role of the Cartan subalgebra)
 by applying it to the much simpler example of $\CC \ZZ_{2p}^{\omega}$. 
It consists of the following steps:
\begin{enumerate}
\item\label{plan-step1}
 For a given subset $\RS\subset\ZZ$, 
 such that $\pi_\RS\colon \RS\to \ZZ_{2p}$ is a bijection, we  construct a $\CC$-linear equivalence 
 $\funF_\RS: \repfd\, \CC \ZZ_{2p} \to \mathcal{M}$.
Different choices $\RS$ and  $\RS'$ lead to naturally isomorphic functors $\funF_\RS$ and $\funF_{\RS'}$. Let us abbreviate $\funF:= \funF_\RS$.
 
\item\label{plan-step2}
We equip $\funF$ with 
 a multiplicative  structure. 
 That is, we give $\CC$-linear natural isomorphism $\funF_{M,N}\colon\funF(M) \tensor_{\Lambda} \funF(N) \xrightarrow{ \sim } \funF(M\tensor N)$, where $M \tensor N$ is defined via the standard coproduct in  $\CC \ZZ_{2p}$.\footnote{
	This step is more complicated for $\Q$ as there the coproduct has to be modified, too.} 
 
\item\label{plan-step3}
Abbreviate 
$\catD =\repfd\, \CC \ZZ_{2p}$.
 The multiplicative functor $\funF\colon \catD \to \mathcal{M}$ is made monoidal by computing the unique natural isomorphism
$\assD$ in $\catD$ which makes the diagram
\begin{equation}\label{eq:transport-assoc-diag}
\xymatrix@R=32pt@C=72pt{
\fun (U)\tensor_{\mathcal M}\bigl(\fun(V)\tensor_{\mathcal M}\fun(W)\bigr)\ar[d]_{\id\tensor\funF_{V,W}}\ar[r]^{\alpha^{\mathcal{M}}_{\fun(U),\fun(V),\fun(W)}}
&\bigl(\fun (U)\tensor_{\mathcal M} \fun(V)\bigr)\tensor_{\mathcal M}\fun(W)\ar[d]^{\funF_{U,V}\tensor\id}\\
\fun (U)\tensor_{\mathcal M} \fun(V\tensor_{\catD} W)\ar[d]^{\funF_{U,V\tensor_{\catD} W}}
&\fun (U\tensor_{\catD} V)\tensor_{\mathcal M} \fun(W)\ar[d]^{\funF_{U\tensor_{\catD} V,W}}\\
\fun \bigl(U\tensor_{\catD} (V\tensor_{\catD} W)\bigr)\ar[r]^{\fun(\assD_{U,V,W})}
&\fun \bigl((U\tensor_{\catD} V)\tensor_{\catD} W\bigr)
}
\end{equation}
commute for all $U,V,W \in \catD$, and where $\tensor_{\mathcal M}$ stands for $\tensor_{\Lambda}$.
 This results in a group 3-cocycle for $\ZZ_{2p}$ with values in $\mathbb{C}^\times$.

\item\label{plan-step4}
To make $\funF$  braided by finding the unique
	natural isomorphism
$c^{\catD}$ in $\catD$ which solves the
 commutativity condition
\begin{equation}\label{eq:transport-braiding-via-functorequiv}
\xymatrix@R=22pt@C=42pt{
&\fun(U)\tensor_{\mathcal{M}}\fun(V)\ar[r]^{\;c^{\mathcal{M}}_{\fun(U),\fun(V)}\;}\ar[d]^{\funF_{U,V}}&\fun(V)\tensor_{\mathcal{M}}\fun(U)\ar[d]^{\funF_{V,U}}&\\
&\fun(U\tensor_{\catD} V)\ar[r]^{\fun(c^{\catD}_{U,V})}&\fun(V\tensor_{\catD} U)&
}  
\end{equation}
 for all $U,V\in\catD$. 
This amounts to finding a 2-cochain on $\ZZ_{2p}$ which together with the group 3-cocycle from step~\ref{plan-step3} defines an abelian 3-cocycle for $\ZZ_{2p}$.
 
\item\label{plan-step5} 
Finally we make the braided monoidal functor $\funF$ balanced
by finding the unique 
	natural
isomorphism $\theta^\catD$ on $\catD$ such that
\be\label{eq:transport-twist}
	\theta^\mathcal{M}_{\fun(U)} = \fun(\theta^\catD_U)
\ee
is satisfied for all $U \in \catD$.
 \end{enumerate}

\begin{remark}\label{rem:solve-pent-hex-autom}
It is important to stress that since $\mathcal{M}$ is balanced braided monoidal, the natural isomorphisms 
$\alpha^\catD$, $c^\catD$, $\theta^\catD$ computed in steps \ref{plan-step3}--\ref{plan-step5}, and which satisfy conditions \eqref{eq:transport-assoc-diag}--\eqref{eq:transport-twist}, \textsl{automatically} solve the pentagon and hexagon condition for $\catD$, as well as \eqref{eq:balancing-def}, so that they turn $\catD$ into a balanced braided monoidal category. 
\\
If $\mathcal{M}$ has duals and is ribbon (as is the case in this example), then, as $\fun$ is an equivalence, also $\catD$ has duals. In this case, the balancing $\theta^\catD$ also automatically satisfies \eqref{eq:ribbontwist-def}, i.e.\ it gives a ribbon structure on $\catD$. To see this, first recall that all choices of left duals are equivalent (i.e.\ the corresponding functors $(-)^*$ are naturally isomorphic).
Then check that the validity of condition \eqref{eq:ribbontwist-def} is independent of the particular choice of left duality. 
\\
Since $\catD = \repfd\, \CC\ZZ_{2p}$ now is a ribbon category, 
the coassociator, $R$-matrix and ribbon element for $\CC \ZZ_{2p}$ which correspond to $\alpha^\catD$, $c^\catD$, $\theta^\catD$ via \eqref{qHopf-cat-data} also 
automatically solve the pentagon, hexagon and ribbon conditions for quasi-Hopf algebras. 
\\
The proof in Section~\ref{sec:proof} that $\Q$ as given in Section~\ref{sec:qHopfUq} is a ribbon quasi-Hopf algebra will use the above reasoning.
The same procedure of ``transport of structure'' as outlined above was already used in~\cite{GR1, FGR2} to obtain a quasi-Hopf algebra related to the symplectic fermions category.
\end{remark}

To implement our plan, we begin with conventions for the  algebra $\CC \ZZ_{2p}$. It is generated by $K$ with $K^{2p}=\one$. Let 
\be
	q=e^{i\pi/p} \ . 
\ee	
The 
full list of primitive idempotents of $\CC \ZZ_{2p}$ is provided by
\be\label{eq:idemp-prim}
\e{n} = \frac{1}{2p}\sum_{l=0}^{2p-1} q^{-nl} K^l ,\qquad n\in \ZZ_{2p} \ .
\ee
Indeed, they satisfy
\be\label{eq:en-ortho-idem-prop}
\e{n} \e{k} = \delta_{n,k}\, \e{k}\qquad \text{and}\qquad \sum_{n=0}^{2p-1} \e{n} = \one\ .
\ee
In the left regular representation, these idempotents are projectors onto $K$-eigenspaces of the eigenvalue~$q^{n}$,
\be\label{eq:K-en}
K\cdot \e{n} = q^{n} \e{n}\ ,
\ee
and any $\CC \ZZ_{2p}$-module $M$ is decomposed onto  $K$-eigenspaces as 
\be\label{eq:M-decomp}
M	=
\bigoplus_{n\in\ZZ_{2p}} M_n 
\quad \text{where} \quad
	M_n  :=  \e{n} M\ .
\ee
We will write $\e{n}m$ for the $n$th weight component of an element $m\in M$.

For 	step~\ref{plan-step1}
of our plan, we need to construct a functor from 
$\repfd\, \CC \ZZ_{2p}$ to 
$\Hplus$,
 the category of
$\mathbb{C}$-graded complex vector spaces. 
To fix the $H$-action
we thus have to ``lift'' the 
	$\ZZ_{2p}$-grading on objects from 
$\repfd\, \CC \ZZ_{2p}$ 
to a $\CC$-grading, 
which amounts to choosing a branch of the logarithm. This choice can be parametrised as follows.
Let $\pi\colon \ZZ \to \ZZ_{2p}$ be the canonical surjection and choose
 \be\label{eq:S}
 \RS\subset \ZZ \qquad \text{such that } \qquad \pi|_\RS\colon\; \RS\xrightarrow{\,\sim\,} \ZZ_{2p} 
 \ee
is a bijection. We will denote by $\re{x}_\RS$ the representative of $x\in\ZZ$ in $\RS$, so $\re{x}_\RS-x=0$ modulo $2p$, and in particular $\re{x}_\RS=x$ if $x\in\RS$.
 We also recall the decomposition~\eqref{eq:free-boson-extension-alg} for our algebra $\Lambda:=\Lambda_p$, however for brevity we will denote the basis element in degree $2pk$ by $1_k$ instead of $1_{2pk}$.
 We define then the $\CC$-linear functor
\be\label{eq:funF-Z2p}
\funF_\RS\colon\; \repfd\, \CC \ZZ_{2p} \longrightarrow  \Hplus \ , \qquad
 M\mapsto \Lambda\otimes_{\CC} M
\ee
with the action, 
for all $k,a \in \mathbb{Z}$,
\be\label{eq:C-gr-ex_K-action_to_H-action}
 H (1_k\otimes \e{\pi(a)}m) := \bigl(2pk + \re{a}_\RS\bigr) \cdot 1_k\otimes \e{\pi(a)}m\ .
\ee
Below we will abbreviate $\e{a} := \e{\pi(a)}$.
The  object $\funF_\RS(M)=\Lambda\otimes_{\CC} M$ has a natural action of $\Lambda$ by multiplication on the left:
\be\label{eq:rho-LambdaM}
\rho_{\Lambda\otimes_{\CC} M}\colon \; \Lambda\otimes \Lambda\otimes_{\CC} M \xrightarrow{\;\mu_{\Lambda}\otimes\id_M\;} \Lambda\otimes_{\CC} M\ .
\ee
It clearly intertwines the $H$-action and the pair $\bigl(\funF_\RS(M), \rho_{\Lambda\otimes_{\CC} M}\bigr)$ is a
finitely-generated local module due to Proposition~\ref{prop:Lambda-loc-mod}.  $\funF_\RS$ sends a morphism $f\colon M \to N$ to the morphism $\id_\Lambda\otimes f$ in $\mathcal{M}$.
This finally defines the functor 
$\funF_\RS\colon \repfd\, \CC \ZZ_{2p} \to \mathcal{M}$. 
It turns out that $\funF_\RS$ depends on $\RS$ only up to natural isomorphism:

\begin{lemma}\label{lem:etaSS}
For two sets $\RS$ and $\RS'$ satisfying~\eqref{eq:S}, $\funF_\RS$ and $\funF_{\RS'}$ are naturally isomorphic via $\eta_{\RS,\RS'}: \funF_\RS(M)\xrightarrow{\sim} \funF_{\RS'}(M)$ given by, for $k,a\in\ZZ$,
\be\label{eq:natis-FS-FS'}
\eta_{\RS,\RS'}\colon \quad 1_k\otimes \e{a}m \;\mapsto\; 1_{k+ \frac{\re{a}_\RS-\re{a}_{\RS'}}{2p}}\otimes \e{a}m \ .
\ee
\end{lemma}

\begin{proof}
It is clear that $\eta_{\RS,\RS'}$ is natural and that it is an isomorphism (its inverse is simply $\eta_{\RS',\RS}$). {}From \eqref{eq:C-gr-ex_K-action_to_H-action} one verifies that $\eta_{\RS,\RS'}$ preserves the $H$-action. To be a morphism in~$\mathcal{M}$, $\eta_{\RS,\RS'}$ also has to intertwine the 
	$\Lambda$-action. 
By \eqref{eq:rho-LambdaM}, this amounts to the identity $(\mu_\Lambda \otimes \id_M) \circ (\id_\Lambda \otimes \eta_{\RS,\RS'}) = \eta_{\RS,\RS'} \circ (\mu_\Lambda \otimes \id_M)$, which is immediate from the explicit form of $\mu_\Lambda$ in \eqref{eq:Lam_p-product} when evaluating on $1_k \otimes 1_l \otimes m$.
\end{proof}
 
The choice of the set $\RS$ is therefore not important and the result will not depend on it, up to an equivalence.
We will then use the shorthand 
\be\label{eq:abbrev-x-xS}
\re{x}:= \re{x}_\RS
\ee
 for brevity in many places below.

To show that $\funF_\RS$ is an equivalence, let us introduce a functor from $\mathcal{M}$ to 
$\repfd\, \CC \ZZ_{2p}$:
\be\label{eq:funG-Cartan}
\funG_\RS\colon\;  \mathcal{M}\to\repfd\, \CC \ZZ_{2p} 
\quad , \quad
 V=\bigoplus_{k\in\ZZ}\bigoplus_{a\in \RS} V_{2pk+a} \; \mapsto \; V_{(0)}:=\bigoplus_{a\in \RS} V_{a}\ ,
\ee
in terms of the corresponding graded spaces (so, we take the $k=0$ component of $V$), and the action of $\ZZ_{2p}$  on $V_{(0)}$ is given by
 $K=q^H$. 
On morphisms $f\colon V\to W$, we define 
$
\funG_\RS\colon f \mapsto  f_{(0)} :=  \bigoplus_{a\in \RS} f_a
$
 with $f_{(0)}$ the restriction of $f$ to the ``fundamental'' component $V_{(0)}\subset V$. We have $f_{(0)}\colon V_{(0)}\to W_{(0)}$ because morphisms in $\mathcal{M}$  respect the grading by~$H$.
 
It is now easy to see that the composition $\funG_\RS\circ\funF_\RS$ is naturally isomorphic to the identity functor 
$\id_{\repfd \CC \ZZ_{2p}}$
via, for all 
$M \in \repfd\, \CC \ZZ_{2p}$,
\be \label{eq:GF-nat_Z2p}
\funG_\RS\circ\funF_\RS(M)\to M \quad , \quad
1_0\tensor m \mapsto m \ .
\ee
Similarly we have a natural isomorphism between $\funF_\RS\circ\funG_\RS$ and $\id_{\mathcal{M}}$.
 On objects, the composition is
\be\label{eq:FG-Cartan}
\funF_\RS\circ\funG_\RS: \quad V \mapsto
	\Lambda
 \tensor_{\CC} V_{(0)} \ .
\ee
	Similar as in the proof of Proposition~\ref{prop:Lambda-loc-mod}, for $(V,\rho_V)\in \mathcal{M}$ and $k \in \ZZ$ we define the isomorphisms of $\CC$-vector spaces
\be\label{eq:rho1t}
	\rho(1_k):=\rho_V(1_k\otimes -): \quad V\to V 
\ee
(we  call it  \textit{the action of the basis element $1_k$ of~$\Lambda$}).
The natural transformation to $\id_{\mathcal{M}}$ can be then written as
\begin{align}\label{eq:H-FG-nat}
\funF_\RS\circ\funG_\RS \xrightarrow{\;\cdot\;} \id_{\mathcal{M}}: \quad &\funF_\RS\circ\funG_\RS(V)\to V \nonumber\\
&1_k \tensor v
  \mapsto \rho(1_k)(v)\ , \qquad k\in\ZZ, \; v\in V_{(0)}\ .
\end{align}
It is a straightforward check that this map is a morphism in $\mathcal{M}$,
 i.e.\ a grade preserving map that also intertwines the $\Lambda$ action. 
The map in~\eqref{eq:H-FG-nat}  is  an isomorphism and its naturality is also straightforward. We have thus shown that both $\funF_\RS$ and $\funG_\RS$ are equivalences.

We summarise the discussion above as:

\begin{proposition}\label{prop:funF-H}
For a set $\RS$ satisfying~\eqref{eq:S}, $\funF_\RS$ is a $\CC$-linear equivalence from   
$\repfd\, \CC \ZZ_{2p}$ to the  category $\mathcal{M}$ of  finitely-generated local left $\Lambda$-modules internal to the category of $\CC$-graded vector spaces $\mathcal{H}^\oplus$.
\end{proposition}

We now turn to step~\ref{plan-step2}, equipping $\funF$ with a multiplicative structure. That is, we give a family of isomorphisms 
\be
	\funF_{M,N}\colon\; \funF(M) \tensor_{\Lambda} \funF(N) \xrightarrow{ \sim } \funF(M\tensor N) \ ,
\ee
natural in $M,N \in \repfd\, \CC \ZZ_{2p}$. 
	Recall from Definition~\ref{def:transport} that
in contrast to a monoidal structure on $\funF$, the $\funF_{M,N}$ are not required to satisfy the coherence condition \eqref{eq:transport-assoc-diag}
for the standard monoidal structure on $\repfd\, \CC \ZZ_{2p}$.
We begin with introducing  a family of $\CC$-linear maps
\begin{align}
\tfunF_{M,N} : \quad \funF(M) \tensor_\CC \funF(N) 
&\longrightarrow \funF(M\tensor N)
\nonumber \\
 (1_k\tensor \e{a}m)\tensor (1_l\tensor \e{b}n)
 &\longmapsto 
(-1)^{al} \zeta_{a,b}\, 1_{k+l + \kappa(a,b)}\tensor (\e{a}m\tensor \e{b}n)\ ,
\label{eq:isotF-MN}
\end{align}
where the	constants $\zeta_{a,b}\in \CC$ and $\kappa(a,b) \in \ZZ$, for $a,b \in \RS$, will be subject to certain conditions which we establish below.

Next we  show that the maps $\tfunF_{M,N}$ factor through $\funF(M) \tensorL \funF(N)$.
Recall from~\eqref{eq:tensor-bimod-def} that $\otimes_\Lambda$ is defined via a coequaliser or explicitly as the quotient
\be\label{eq:tensorA-quotient}
\funF(M)\otimes_\Lambda \funF(N) = \funF(M)\tensor \funF(N) / \ker\pi_\otimes
\ee
where 
\be\label{eq:ker-pi}
 \ker\pi_\otimes = \mathrm{im} \big(\id_{\funF(M)} \otimes \rho_{\funF(N)}^l - \rho_{\funF(M)}^r\otimes\id_{\funF(N)} \big)
\ee
and we also used here that the associator $\alpha_{\funF(M),\Lambda,\funF(N)}$ in our case is trivial.

\begin{lemma}\label{lem:tfun-fact}
We have
	the
factorisation
\be\label{eq:tfunF-fact}
\tfunF_{M,N} = \funF_{M,N}  \circ \pi_\otimes
\ee
where
\be
\funF_{M,N}: \quad \funF(M) \tensor_{\Lambda} \funF(N) \xrightarrow{\quad \sim \quad} \funF(M\tensor N)
\ee
are $\mathbb{C}$-linear isomorphisms.
\end{lemma}

In other words, the maps $\tfunF_{M,N}$ are well defined on elements from $\funF(M)\otimes_\Lambda \funF(N)$ and do not depend on the choice of representatives in $\funF(M)\otimes \funF(N)$.
We stress that in this lemma, the $\funF_{M,N}$ are not yet required to be morphisms in $\mathcal{M}$, just linear isomorphisms between the underlying vector spaces.

\begin{proof}
Following the definition~\eqref{eq:tensorA-quotient} we first have to  calculate the image of the equaliser map 
\be
\eq:= \id_{\funF(M)} \otimes \rho_{\funF(N)}^l - \rho_{\funF(M)}^r\otimes\id_{\funF(N)} 
\ee
and to show that $\tfunF_{M,N}$ is zero on it. Recall that $\rho_{\funF(M)}^l$ is given by~\eqref{eq:rho-LambdaM}, while following~\eqref{eq:right-action-from-left-action} the right $\Lambda$-action is the composition
\begin{align}
\rho_{\funF(M)}^r\colon \;  &\Lambda\otimes_{\CC} M \otimes \Lambda \xrightarrow{\; \brC_{ \Lambda\otimes_{\CC} M, \Lambda}\;} \Lambda\otimes \Lambda\otimes_{\CC} M \xrightarrow{\;\mu_{\Lambda}\otimes\id_M\;} \Lambda\otimes_{\CC} M\nonumber\\
&  (1_k\otimes \e{a}m) \otimes 1_s \mapsto (-1)^{as}1_{k+s}\otimes \e{a}m\ ,
\end{align}
where we also used~\eqref{eq:alt-local-def} for the local module $\funF(M)$, i.e.\ the choice of $\pm$ is irrelevant. Then the image of $\eq$ is
\begin{equation*}
\im(\eq) =  \Span \bigl\{ (1_k\otimes \e{a}m)  \otimes (1_{l+s}\otimes \e{b}n) - (-1)^{as}  (1_{k+s}\otimes \e{a}m)  \otimes (1_l\otimes \e{b}n) \, | \, k,l,s\in \ZZ, a,b\in \RS\bigr\} .
\end{equation*}
It is now straightforward to check that $\tfunF_{M,N}$ is zero on $\im(\eq)=\ker (\pi_{\tensor})$ for any $\zeta_{a,b}$ and $\kappa(a,b)$,
so~\eqref{eq:tfunF-fact} holds.
It remains to
show that $\funF_{M,N}$ are invertible. Indeed the maps
\be
\funF_{M,N}^{-1} \colon \; 1_k\otimes (\e{a}m \otimes \e{b}n) \mapsto \zeta_{a,b}^{-1} (1_{k-\kappa(a,b)}\otimes \e{a}m)\otimes_\Lambda (1_0 \otimes \e{b}n)
\ee
are left and right inverses to $\funF_{M,N}$.
\end{proof}

Because of our convention (see Section~\ref{sec:conventions}) that the representation category of a quasi-Hopf algebra uses the unit isomorphisms of the underlying category of vector spaces, it is natural to demand the isomorphism $\one \to \funF(\one)$ to be the identity. This implies the constraint
\be
	\zeta_{a,\re{0}} = 1 = \zeta_{\re{0},a} 
	\qquad \text{for all}
	\quad a \in \RS
\ee
for the constants $\zeta_{a,b}$ in \eqref{eq:isotF-MN}.
For $\kappa$ in \eqref{eq:isotF-MN} we obtain the following condition:

\begin{proposition}\label{prop:funF-kappa}
The maps $\funF_{M,N}$ are morphisms in $\mathcal{M}$ iff 
\be\label{eq:kappa}
\kappa(a,b) = \ffrac{a+b - \re{a+b}}{2p}\; \in \; \ZZ \ .
\ee
	The $\funF_{M,N}$ act on elements as
\be\label{eq:isoF-MN-Cartan}
\funF_{M,N}: \quad 
 (1_k\tensor \e{a}m)\tensor_{\Lambda} (1_l\tensor \e{b}n)\mapsto 
(-1)^{al}
 \zeta_{a,b}\, 1_{k+l + \kappa(a,b)}\tensor (\e{a}m\tensor \e{b}n)\ .
\ee
\end{proposition}

\begin{proof}
First note that the action in~\eqref{eq:isoF-MN-Cartan} is a simple consequence of the factorisation~\eqref{eq:tfunF-fact} and the definition of $\tfunF_{M,N}$.
	Next we show
that $\funF_{M,N}$ is a $\CC$-grade preserving map
iff \eqref{eq:kappa} holds.
	One can check from the definition in \eqref{eq:tensorA-quotient} that 
	the $H$-action on $\funF(M) \tensorL \funF(N)$ is given by $H\tensorL \one + \one\tensorL H$.
Then,  on the one hand we have,
	for $a,b \in \RS$ and $k,l \in \ZZ$,
\be\label{eq:H-acts-1}
H\colon\;  (1_k\tensor \e{a}m)\tensorL (1_l\tensor \e{b}n) \mapsto 
\bigl(2p(k+l) + a + b\bigr) (1_k\tensor \e{a}m)\tensorL (1_l\tensor \e{b}n)
\ .
\ee
On the other hand,
\be\label{eq:H-acts-2}
H\colon\; 1_{k+l + \kappa(a,b)}\tensor (\e{a}m\tensor \e{b}n) \mapsto   
\Bigl(2p\bigl(k+l +  \kappa(a,b)\bigr) + \re{a + b}\Bigr)  1_{k+l + \kappa(a,b)}\tensor (\e{a}m\tensor \e{b}n)  \ ,
\ee
where we used that $K$ acts on $\e{a}m\tensor \e{b}n$ by  $q^{a+b}$, so that  by \eqref{eq:C-gr-ex_K-action_to_H-action}, for the $H$-action we must use the representative $\re{a + b}$ in $\RS$. 
Comparing~\eqref{eq:H-acts-1} and~\eqref{eq:H-acts-2} 
we conclude that the $H$ action  commutes with $\funF_{M,N}$ iff the equality~\eqref{eq:kappa} holds.

It remains to show that the maps $\funF_{M,N}$ intertwine the $\Lambda$-action, or equivalently, that  $\funF_{M,N}$ commute with the action $\rho(1_s)$ of the basis elements of $\Lambda$
defined in~\eqref{eq:rho1t}. On one side, we have the composition $\funF_{M,N}\circ\rho(1_s)$, for $s\in\ZZ$, in the basis
\begin{align}
(1_k\tensor \e{a}m)\tensorL (1_l\tensor \e{b}n)
 &\xrightarrow{\;\rho(1_s)\;} (1_{k+s}\tensor \e{a}m)\tensorL (1_l\tensor \e{b}n) \\
 &\xrightarrow{\;\funF_{M,N}\;} (-1)^{al} \zeta_{a,b}\, 1_{k+s+l + \kappa(a,b)}\tensor (\e{a}m\tensor \e{b}n) \ .\nonumber
\end{align}
On the other hand:
\begin{align}
(1_k\tensor \e{a}m)\tensorL (1_l\tensor \e{b}n)
 &\xrightarrow{\;\funF_{M,N}\;} (-1)^{al} \zeta_{a,b}\, 1_{k+l + \kappa(a,b)}\tensor (\e{a}m\tensor \e{b}n) \\
 &\xrightarrow{\;\rho(1_s)\;} (-1)^{al} \zeta_{a,b}\, 1_{k+s + l + \kappa(a,b)}\tensor (\e{a}m\tensor \e{b}n) \ .\nonumber
\end{align}
We have therefore the equality $\funF_{M,N}\circ\rho(1_s) = \rho(1_s)\circ \funF_{M,N}$ for $s\in\ZZ$.
\end{proof}

As a corollary of this proposition and Lemma~\ref{lem:tfun-fact}, we conclude that the equivalence functor $\funF$ is multiplicative. This completes step~\ref{plan-step2}.

\bigskip

We now turn to	step~\ref{plan-step3} of
our plan. 
The  commutativity condition of the diagram~\eqref{eq:transport-assoc-diag} reads 
 \be\label{eq:transport-assoc-diag-eq}
 \funF_{U\tensor_{\catD} V,W} \circ \bigl(\funF_{U,V}\tensor_\Lambda\id\bigr) =  \fun(\Phi\, . -) \circ \funF_{U,V\tensor_{\catD} W} \circ \bigl(\id\tensor_\Lambda\funF_{V,W}\bigr)\ ,
 \ee
 where we used that $\alpha^{\mathcal{M}}_{\fun(U),\fun(V),\fun(W)}$ is trivial, i.e.\ as for vector spaces, and we realised the associator $\assD_{U,V,W}$ by the action with $\Phi\in \CC\ZZ_{2p}^{\otimes 3}$.  We first calculate the LHS of~\eqref{eq:transport-assoc-diag-eq}, starting from an element in the left-top corner of~\eqref{eq:transport-assoc-diag} of the form, for $a,b,c \in \RS$, 
 \be\label{eq:left-top-el}
 (1_k\tensor \e{a} u)\tensorL \bigl((1_l\tensor \e{b} v)\tensorL (1_t\tensor\e{c} w)\bigr) \; \in  \;
 \fun (U)\tensor_{\mathcal M}\bigl(\fun(V)\tensor_{\mathcal M}\fun(W)\bigr).
 \ee
  It is mapped by $\alpha^{\mathcal{M}}$ to $\bigl((1_k\tensor \e{a} u)\tensorL (1_l\tensor \e{b} v)\bigr)\tensorL (1_t\tensor\e{c} w)$. Then 
 \begin{align}\label{eq:transport-assoc-diag-LHS}
 \text{LHS of}~\eqref{eq:transport-assoc-diag-eq} \colon \quad
& \bigl((1_k\tensor \e{a} u)\tensorL (1_l\tensor \e{b} v)\bigr)\tensorL (1_t\tensor\e{c} w) 
\nonumber\\
& \xrightarrow{\funF_{U,V}\tensor\id}
(-1)^{al} \zeta_{a,b} \bigl(1_{k+l+\kappa(a,b)}\tensor
	(\e au\tensor \e bv)
\bigr)\tensorL (1_t\tensor \e{c}w)
\nonumber\\
&\xrightarrow{\funF_{U\tensor_{\catD} V,W}} 
(-1)^{al + t(a+b)} \zeta_{a,b} \zeta_{\re{a+b},c} \,1_{k+l+t+ \kappa(a+b,c)}\tensor
	\bigl((\e au\tensor \e bv)\tensor  \e{c}w\bigr)\ , 
 \end{align}
where in the second line we used 
that $\e{a}u\tensor \e{b}v$ has $\ZZ_{2p}$-degree $a+b$ and this is why  
the representative $\re{a+b}_\RS$ appears in the third line.
 We similarly calculate   the RHS of~\eqref{eq:transport-assoc-diag-eq} on the element~\eqref{eq:left-top-el} with the result
  \begin{align}
 \text{RHS of}~\eqref{eq:transport-assoc-diag-eq} \colon \quad
& (1_k\tensor \e{a} u)\tensorL \bigl((1_l\tensor \e{b} v)\tensorL (1_t\tensor\e{c} w)\bigr)\nonumber\\
 & \mapsto 
 (-1)^{a(l + t + \kappa(b,c)) + bt} \zeta_{b,c} \zeta_{a,\re{b+c}} \,1_{k+l+t+ \kappa(a+b,c)}\tensor
   \Phi\bigl(\e{a}u\tensor (\e bv\tensor \e cw)\bigr)\ . \label{eq:transport-assoc-diag-RHS}
 \end{align}
 Comparing the two sides, i.e.\ equating~\eqref{eq:transport-assoc-diag-LHS} to~\eqref{eq:transport-assoc-diag-RHS}, we conclude that the solution of~\eqref{eq:transport-assoc-diag-eq} is,
	for $a,b,c \in \RS$,
 \be
 \Phi\bigl(\e{a}u\tensor(\e{b}v\tensor \e{c}w)\bigr) = (-1)^{a\kappa(b,c)} \ffrac{ \zeta_{a,b} \zeta_{\re{a+b},c}}{\zeta_{b,c} \zeta_{a,\re{b+c}} }\bigl((\e{a}u\tensor\e{b}v)\tensor \e{c}w\bigr) \ .
 \ee
Taking here the sum over $a,b,c\in\RS$, we finally get the co-associator for $\CC \ZZ_{2p}$:
 \be\label{eq:Phi-Z2p}
 \Phi \,= 	\sum_{a,b,c\in \RS}  
 (-1)^{a\kappa(b,c)} \ffrac{ \zeta_{a,b} \zeta_{\re{a+b},c}}{\zeta_{b,c} \zeta_{a,\re{b+c}}} \,\e{a}\tensor\e{b}\tensor \e{c} \ .
 \ee
 
 \medskip
 
We now turn to 	step~\ref{plan-step4} of
our plan. The  commutativity condition of the diagram~\eqref{eq:transport-braiding-via-functorequiv} reads 
 \be\label{eq:transport-br-diag-eq}
 \funF_{V,U}\circ c^{\mathcal{M}}_{\fun(U),\fun(V)} = \fun(\tau_{U,V}\circ R\, . - )\circ\funF_{U,V} 
 \ee
 where we realised the braiding $c^{\catD}_{U,V}$ by the action with $R$-matrix $R\in \CC\ZZ_{2p}^{\otimes 2}$ followed by the flip of vector spaces. 
	Recall that the braiding in $\mathcal{M}$ is inherited from the one in $\Hplus$ defined in~\eqref{eq:C-graded-vec-braiding} (given by $\tau_{X,Y} \circ q^{\frac{1}{2} H_X \tensor H_Y}$ for our choice of $r$) via \eqref{eq:braiding-descends}.  
We first calculate the LHS of~\eqref{eq:transport-br-diag-eq}:
	for $a,b \in \RS$ we get
 \begin{align}
 (1_k\tensor \e{a} u)\tensorL (1_l\tensor \e{b} v)& 
	\xrightarrow{c^{\mathcal{M}}_{\funF(U),\funF(V)} } 
 q^{(2pk + a)(pl + b/2)} (1_l\tensor \e{b} v)\tensorL (1_k\tensor \e{a} u) \nonumber\\
& \xrightarrow{ \funF_{V,U} } q^{(2pk + a)(pl + b/2)} (-1)^{bk}\zeta_{b,a} 1_{k+l + \kappa(b,a)}\tensor ( \e{b} v\tensor  \e{a} u)\ .
 \end{align}
The result for  RHS of~\eqref{eq:transport-br-diag-eq} is
\be\label{eq:br-diag-RHS}
 (1_k\tensor \e{a} u)\tensorL (1_l\tensor \e{b} v) \mapsto (-1)^{al} \zeta_{a,b} 1_{k+l +\kappa(a,b)} \tensor \bigl(
\tau_{U,V}
\circ R . (\e{a}u\tensor \e{b}v)\bigr)\ .
\ee
Comparing the both sides we find the $R$-matrix:
\be\label{eq:R-Z2p}
R  =
 \sum_{a,b\in \RS}
 q^{\half ab} \ffrac{\zeta_{b,a}}{\zeta_{a,b}}\,  \e{a}\tensor \e{b} \ .
\ee

\medskip

Finally, for  step~\ref{plan-step5} we write the twist on a module $U \in \repfd\, \CC \ZZ_{2p}$ as $\theta_U = \ribbon^{-1}.(-)$ for an element $\ribbon \in \CC \ZZ_{2p}$ to be computed. 
Evaluating 
the required equation \eqref{eq:transport-twist}
on $1_k \otimes \e au$ gives (recall \eqref{eq:C-graded-qf-choice})
\be
	q^{\frac12 (2pk+\re{a})^2} \, 1_k \otimes \e au = 1_k \otimes \ribbon^{-1}.\e au \ .
\ee
This can be solved for $\ribbon$ as
\be\label{eq:ribbon-Z2p}
	\ribbon = \sum_{a \in \ZZ_{2p}} q^{-\frac12  a^2} \, \e a 
	\,\overset{(*)}=\,
	\ffrac{1-i}{2 \sqrt{p}} \sum_{l \in \ZZ_{2p}} q^{\frac12 l^2} K^l \ .
\ee
In step $(*)$ we substituted the definition of $e_a$ in \eqref{eq:idemp-prim} and carried out the Gauss sum using
\be\label{eq:gauss-sum}
\sum_{a=0}^{2p-1} e^{-\frac{i\pi}{2p} a^2} =  (1-i)\sqrt{p}\ .
\ee

	By Remark~\ref{rem:solve-pent-hex-autom}, $\CC\ZZ_{2p}$ together with $\Phi$ is a quasi-bialgebra.  
To turn it into a quasi-Hopf algebra, we need to find an antipode structure,
and, again by  Remark~\ref{rem:solve-pent-hex-autom}, different such structures are equivalent. One can check that the following choice for the antipode $S$, the evaluation element $\Salpha$ and the coevaluation element $\Sbeta$ works:
\be\label{eq:CZ_2p-antipode-struc}
	S(K) = K^{-1}
	\quad , \quad
	\Salpha = \one
	\quad , \quad
	\Sbeta = \sum_{a\in \RS}
		q^{\half a\re{0}}
	 \, \ffrac{\zeta_{a,-a}}{\zeta_{-a,a}}\e{a} \ .
\ee
Thus $\CC\ZZ_{2p}$ is a quasi-Hopf algebra.
Since $\mathcal{M}$ is in addition ribbon, by construction the elements $R$ and $\ribbon$ turn $\CC\ZZ_{2p}$ into a ribbon quasi-Hopf algebra.

\begin{proposition}\label{prop:Z-2p-quasiH}
The data  $(\CC\ZZ_{2p}, \Phi, R, \ribbon)$ with
the standard group-like coproduct, 
	antipode structure \eqref{eq:CZ_2p-antipode-struc},
the coassociator $\Phi$ defined in~\eqref{eq:Phi-Z2p}, the $R$-matrix in~\eqref{eq:R-Z2p} and the ribbon element in~\eqref{eq:ribbon-Z2p} 
is a factorisable ribbon quasi-Hopf algebra.
\end{proposition}

\begin{proof}
It only remains to  show factorisability. 
Recall the definition of a factorisable quasi-Hopf algebra in Section~\ref{sec:conventions}. Here, we will use the second criteria of factorisability formulated in the very end of Section~\ref{sec:conventions} as absence of non-trivial transparent objects.
	Since $\CC\ZZ_{2p}$ is semisimple it is enough to check this condition on simple modules.
Using~\eqref{eq:R-Z2p} we compute the double braiding for a pair of simple 
	$\CC\ZZ_{2p}$-modules 
$\CC_\alpha$ and $\CC_\beta$, where $\e{a}$ acts as identity on $\CC_\alpha$ iff $\alpha=a$ and zero otherwise. Let $1_\alpha\in\CC_\alpha$ then
\be
(c\circ c)(1_\alpha\tensor 1_\beta) =  q^{\half \alpha \beta} \ffrac{\zeta_{\beta,\alpha}}{\zeta_{\alpha,\beta}}\, c (1_\beta\tensor 1_\alpha)  = q^{ \alpha \beta}\cdot 1_\alpha\tensor 1_\beta\ .
\ee
Therefore we look for such $\alpha$ that $q^{ \alpha \beta}=1$ for all $\beta\in\ZZ_{2p}$. Obviously there is only one solution $\alpha=0$.
\end{proof}

Denote the ribbon quasi-Hopf algebra from Proposition~\ref{prop:Z-2p-quasiH} by $\CC \ZZ_{2p}^{\omega}$, where $\omega$ stands for the abelian 3-cocycle determined by $\Phi$ and $R$ (we do not make the ribbon element explicit in the notation), and let $\repfd\, \CC \ZZ_{2p}^{\omega}$ 
be the ribbon category of its finite-dimensional
representations.
Combining all the steps, we can now state the following refinement of Proposition~\ref{prop:funF-H}.

\begin{theorem}\label{thm:Cartan-funF-H-br-equiv}
The functor $\funF\colon \repfd\, \CC \ZZ_{2p}^{\omega} \to \mathcal{M}$
 equipped with the isomorphisms $\funF_{M,N}$ in~\eqref{eq:isoF-MN-Cartan} is a ribbon equivalence.
\end{theorem}

\begin{remark}
The associator, braiding and ribbon twist on $\repfd\, \CC \ZZ_{2p}$ given in \cite[Sec.\,4.2]{Brunner:2000wx} and
\cite[App.\,B.1]{Fuchs:2007tx}
use the choice $\zeta_{a,b}=1$ (and the inverse braiding and twist). 
In Section~\ref{sec:proof} a different choice for $\zeta_{a,b}$ will be more convenient, see \eqref{eq:zeta-def} below.
\end{remark}

\section{The unrolled restricted quantum group $\UQG{q}$}\label{sec:unrolled-restricted}

In this section, we recall the definition of the unrolled quantum group for $s\ell(2)$ and standard facts on its representation theory, mainly following~\cite{CGP}.
We then give an algebra object~$\algC$ in the category of $\UQG{q}$-modules and  study its local modules.

\subsection{Quantum group and weight modules}\label{sec:unrolled+weight}

Let $\UQG{q}$ denote the unrolled restricted
quantum group introduced in~\cite{GPT,CGP}. 
It is defined as the Hopf algebra generated by $E$, $F$, $K^{\pm1}$, and $H$ with relations
\begin{equation}\label{Uq-relations}
  KEK^{-1}=q^2E\ ,\quad
  KFK^{-1}=q^{-2}F\ ,\quad
  [E,F]=\ffrac{K-K^{-1}}{q-q^{-1}}\ ,
\end{equation}
as well as
\be\label{eq:H-rel}
[H,E] = 2E, \qquad [H,K] = 0,\qquad [H,F] = -2F,
\ee
and
\begin{equation}\label{EpFp-rel}
  E^{p}=F^{p}=0 \ .
\end{equation}
The comultiplication, counit and antipode for $E$, $F$, and $K$ are given by
\begin{align}
  \Delta(E)&=\one\otimes E+E\otimes K\ ,\quad &
  \Delta(F)&=K^{-1}\otimes F+F\otimes\one\ ,\quad &
  \Delta(K)&=K\otimes K\ ,
  \label{eq:coprod} \\ 
\eps(E)&=0 \ ,& \eps(F)&=0\ ,\quad & \eps(K)&=1\ ,\nonumber
\\
S(E)&=-EK^{-1}\ ,\quad& S(F)&=-KF\ ,\quad& S(K)&=K^{-1}\ ,\nonumber
\end{align}
and for $H$ as
 \be\label{eq:H-coprod}
 \Delta(H) = H\tensor \one + \one \tensor H, \qquad \eps(H)=0, \qquad S(H) = -H \ . 
 \ee
 The subalgebra generated by $E$, $F$, and $K^{\pm1}$ clearly has a standard PBW basis   $E^a F^b K^k$ with $0 \le a,b <p$, and $k \in \ZZ$. Then any word in the algebra $\UQG{q}$ can be reduced to a linear combination of elements of the form $E^a F^b H^c K^k$ by successfully applying relations~\eqref{eq:H-rel}, this can be proven by induction on the length of the word. We therefore see that
$\UQG{q}$ has a PBW type basis:
\be\label{eq:Uq-basis}
 \UQG{q} =  \bigl\langle E^a F^b H^c K^k\; | \; 0 \le a,b <p\, , \, c \in \ZZ_{\ge 0}\, , \, k \in \ZZ\bigr\rangle\gp
 \ee 

\medskip

We now turn to representations of $\UQG{q}$. We will only consider a particular class of modules, defined as follows.

\begin{definition}
A $\UQG{q}$-module is of \textit{weight type} if the action of $H$ is diagonalisable 
and the $K$-action satisfies
\be
	K=q^H
	\qquad \text{where} \quad
	q^H := e^{\pi i H/p} \ .
\ee
 \end{definition}

We use the standard convention for $q$-numbers and $q$-factorials, for $n \in \ZZ_{\ge 0}$,
\be
	[n] = \frac{q^n - q^{-n}}{q-q^{-1}}
	\quad , \quad
	[n]! = \begin{cases} 1 &; n=0 \\
	[n][n{-}1] \cdots [1] &; n>0 
	\end{cases}
	\qquad .
\ee
Given weight modules $U,V$ of $\UQG{q}$, we define the linear endomorphisms $R$ of $U \otimes V$ and $\ribbon$ of $U$ as, for $x \in U$, $y\in V$,
\begin{align}
R(x \otimes y)  &\,=\, q^{\frac{H\otimes H}{2}} \sum_{n=0}^{p-1} \frac{(q-q^{-1})^n}{[n]!}q^{\half n(n-1)} E^n\otimes F^n . (x \otimes y)
\label{eq:R-H}
\\
\ribbon(x) &\,=\, 
 K^{p-1} \sum_{n=0}^{p-1}  \frac{(q-q^{-1})^{n}}{[n]!} 
 q^{\half n(n-1)} S(F^n)q^{-\half{H^2}}E^n . \, x\ .
\nonumber
\end{align}
Here $q^{\frac{H\otimes H}{2}}$ and $q^{-{H^2}/{2}}$ are defined via their action on $H$-eigenvectors. For example, if $Hx = \mu x$ and $Hy = \nu y$, then $q^{\frac{H\otimes H}{2}} x \otimes y = \exp(\frac{\pi i}{2p} \mu \nu) x \otimes y$. As $H$ is assumed to be diagonalisable on $U$ and $V$, the maps $R$ and $\ribbon$ are well-defined. Since $q^{\frac{H\otimes H}{2}}$ and $q^{-{H^2}/{2}}$ are not contained in (the tensor square of) $\UQG{q}$, we cannot express $R$ and $\ribbon$ as the action of a universal $R$-matrix and a ribbon element. Nonetheless, 
it is shown in~\cite{Oh,CGP}
that the following linear maps define a braiding and a balancing on weight modules:
\be\label{eq:UHq-braid-bal}
	 c_{U,V} = \tau_{U,V} \circ (R.-)
 \quad , \qquad
 \theta_U = \ribbon^{-1}(-) 
\ee
	(cf.\ our conventions in Section~\ref{sec:conventions}).

\medskip

We denote the category of finite dimensional weight modules of $\UQG{q}$ by
\be
\catUfd_p:=\repfdwt \UQG{q} \ .
\ee

\begin{proposition}[{\cite{Oh,CGP}}]
$\catUfd_p$ is a ribbon category.
\end{proposition}

The simple modules in $\catUfd_p$ are explicitly known 
\cite[Lem.\,5.3]{CGP}.
They are of two types:
\begin{enumerate}
\item
 $p$-dimensional \textsl{typical modules} 
 $\UqgTyp{\alpha}$ for 
	$\alpha\in (\CC \setminus \ZZ)\cup p\ZZ$. 
They have a basis $\atprp_{0}, \dots, \atprp_{p-1}$ with the action
\be\label{eq:atprp}
H \atprp_{n} = \left( \alpha+p-1-2n\right)\atprp_{n} \ , 
\qquad 
E\atprp_n = [n] [n{-}\alpha]
 \atprp_{n-1}, \qquad F\atprp_n = \atprp_{n+1} \ ,
\ee
 where we set $\atprp_{-1}=\atprp_{p}=0$. For complex $z$ the $q$-number $[z]$ is evaluated by setting $q^z := e^{\frac{\pi i }{p} z}$.

 \item
The $s$-dimensional \textsl{atypical modules}
 $\repS_{s,k}$
are labelled by
pairs $(s,k)$, with $1\leq s\leq p-1$ and $k\in\ZZ$, and have the highest
$H$-weight $(s-1+kp)$, and so the $K$-weight is $(-1)^k \q^{s-1}$.
They have basis vectors $\stprp_{n}$, $0\leq n\leq s{-}1$,
where $\stprp_{0}$ is the highest-weight
vector and the left action of the algebra is
\begin{equation}\label{eq:UHq-action-S_sk}
  H \stprp_{n} =
 (s-1-2n + kp) \stprp_{n},\qquad
  E \stprp_{n} =
  (-1)^k [n][s - n]\stprp_{n - 1},\qquad
  F \stprp_{n}= \stprp_{n + 1},
\end{equation}
where we set $\stprp_{-1}
=\stprp_{s}=0$.
\end{enumerate}

\begin{figure}[tb]
\begin{center}
\begin{tikzpicture}[scale=0.70][thick,>=latex,
nom/.style={circle,draw=black!20,fill=black!20,inner sep=1pt}
]
\node(left0) at  (-1,0) [] {$\UqgStag{s,k}$:}; 
\node (top1) at (5,2.5) [] {$\repS_{s, k}$};
\node (left1) at (2.5,0) [] {$\repS_{p-s, k+1}$};
\node (right1) at (7.5,0) [] {$\repS_{p-s, k-1}$};
\node (bot1) at (5,-2.5) [] {$\repS_{s, k}$};
\draw [->] (top1) -- (left1);
\draw [->] (top1) -- (right1);
\draw [->] (left1) -- (bot1);
\draw [->] (right1) -- (bot1);
\end{tikzpicture}
\captionbox{\label{fig:Loewy_UHq} The Loewy diagram of the projective cover  $\UqgStag{s,k}$ of the simple module $\repS_{s,k}$.}{\rule{12cm}{0cm}}
\end{center}
\end{figure}

In \cite[Sec.\,6]{CGP} it is shown that these simple modules have projective covers in $\catUfd_p$. 
The modules $V_\alpha$ are themselves projective, and the Loewy diagram of the projective cover of $\repS_{s,k}$ is given in Figure~\ref{fig:Loewy_UHq}.
In particular, we get:

\begin{proposition}\label{prop:Cfdp-has-enough-projectives}
The category $\catUfd_p$ has enough projectives.
\end{proposition}

We write $\catU_p^\oplus$ for the completion of $\catUfd_p$ with respect to countably infinite direct sums
 and subquotients.\footnote{
 	By this we mean the smallest (i.e.\ the universal) category which contains a) $\catUfd_p$, b) countable direct sums of objects in $\catU_p^\oplus$, c) all kernels and cokernels of morphisms in $\catU_p^\oplus$ (and so in particular subobjects and their quotients, i.e.\ subquotients). Since countable sums of countable sums are again countable, one can restrict oneself to countable sums of objects in $\catUfd_p$ and their subquotiens.
} 
 To give a more direct description, let $\repwt \UQG{q}$ denote the category of weight modules which are finite dimensional or countably infinite dimensional. We have:

\begin{lemma}\label{lem:direct-sum-completion}
$\catU_p^\oplus = \repwt \UQG{q}$.
\end{lemma}

\begin{proof}
Clearly, $\catU_p^\oplus$ is a full subcategory of $\repwt \UQG{q}$. Let now $M \in \repwt \UQG{q}$ be arbitrary. We need to show $M \in  \catU_p^\oplus$. If $M$ is finite dimensional there is nothing to do. For countably-infinite dimensional $M$ we will establish the existence of finite-dimensional projective modules $P_k$, $k \in \mathbb{N}$,  and a surjection 
\be\label{eq:sum-completion-aux1}
	\bigoplus_{k \in \mathbb{N}} P_k \to M
\ee
of $\UQG{q}$-modules, proving the claim. 
To verify \eqref{eq:sum-completion-aux1} we proceed in three steps. 

\smallskip

\noindent
1.~Let $M \in \repwt \UQG{q}$ and $m \in M$. Then the submodule $\langle m \rangle \subset M$ generated by $m$ is finite-dimensional. To see this, first note that $m$ is a finite sum of $H$-homogeneous components.
 Assume now that $m$ itself is $H$-homogeneous. 
Since $\UQG{q}$ has	the PBW-type basis given in~\eqref{eq:Uq-basis},
 and since $H^c K^k$ acts on $m$ by a constant, clearly $\langle m \rangle$ is finite dimensional.

\smallskip

\noindent
2.~Let $P \in \catUfd_p$ be projective. Then $P$ is also projective in $\repwt \UQG{q}$. Indeed, let $f\colon M \to N$ be a surjection between modules in $\repwt \UQG{q}$. We may assume that $P$ is indecomposable and generated by a $H$-homogeneous element $p \in P$. Given $g\colon P \to N$, pick $m \in M$ such that $f(m) = g(p)$.
 By part 1, $\langle m  \rangle \subset M$ and $\langle g(p) \rangle \subset N$ are finite dimensional. 
 By construction, $g(P)
	= 
\langle g(p) \rangle$ and $f(\langle m \rangle) = \langle g(p) \rangle$. Since $P$ is projective for finite dimensional modules, there is an intertwiner $\tilde g\colon P \to \langle m \rangle$ such that $f \circ \tilde g = g$.

\smallskip

\noindent
3.~Let $M \in  \catU_p^\oplus$ be infinite and let $m_k$, $k \in \mathbb{N}$, be a generating set. By step 1, each $\langle m_k \rangle$ is finite dimensional. Since $\catUfd_p$ has enough projectives (Proposition~\ref{prop:Cfdp-has-enough-projectives}), there are $P_k \in \catUfd_p$ projective and surjections $P_k \to \langle m_k \rangle$ for all $k \in \mathbb{N}$.
By step 2, $\bigoplus_{k \in \mathbb{N}} P_k$ is projective in $\repwt \UQG{q}$ (direct sums of projectives are projective) and surjects onto $M$. This establishes \eqref{eq:sum-completion-aux1}.
\end{proof}

\begin{remark}~
\begin{enumerate}\setlength{\leftskip}{-1em}
\item
Lemma~\ref{lem:direct-sum-completion} would be false
 without the relations $E^p=0=F^p$, as Verma modules 
(i.e.\ the modules induced from one-dimensional representations of the sub-algebra $\langle E,K,H\rangle$)
 would be missing from the completion with respect to direct sums and subquotients.
\item 
Note that a simple $\UQG{q}$-module is a quotient of $\UQG{q}$ and hence in particular of at most countable dimension. 
As a corollary to Lemma~\ref{lem:direct-sum-completion} we see that all simple $\UQG{q}$-modules of weight type (without restriction on their dimension) are in fact finite dimensional.
\item The category $\catU_p^\oplus$ is still a $\CC$-linear abelian braided 
	 and balanced
tensor category. However, it is no longer rigid (and hence not ribbon) as infinite direct sums will have no duals.
\end{enumerate}
\end{remark}

The next proposition shows that the braiding is in a certain sense maximally non-degenerate, using the notion of transparent objects (cf.\ Section~\ref{sec:conventions}).

\begin{proposition}\label{prop:C-transp-obj}
Both $\catUfd_p$ and $\catU_p^{\oplus}$ have no non-trivial transparent objects, i.e.\ all transparent objects are isomorphic to direct sums of the tensor unit~$\one = \repS_{1,0}$.
\end{proposition}
\begin{proof}
We first prove the statement for $\catUfd_p$. We start by showing that
the tensor unit $\one = \repS_{1,0}$ is the only simple transparent object. To this end we compute the double braiding  for a pair of  simple modules. 
The braiding 
$\brC$ in~$\catU_p$ is defined via the $R$-matrix in~\eqref{eq:R-H}.
For a pair $V_\alpha$ and $V_\beta$ of typical simple modules the action 
of $c\circ c$ on  the highest-weight vectors $\atprp^{(\alpha)}_0\in V_\alpha$
is given by (recall~\eqref{eq:atprp}),
\be
(c\circ c)\bigl(\atprp^{(\alpha)}_0\tensor \atprp^{(\beta)}_0\bigr) =   q^{\half(\alpha+p-1)(\beta+p-1)} c \bigl( \atprp^{(\beta)}_0\tensor\atprp^{(\alpha)}_0\bigr)  = 
q^{(\alpha+p-1)(\beta+p-1)} \bigl(\atprp^{(\alpha)}_0\tensor \atprp^{(\beta)}_0\bigr) \ ,
\ee
where we used that $E \atprp^{(\alpha)}_0 = 0$.
Therefore we look for such $\alpha\in(\CC \setminus \ZZ)\cup p\ZZ$ that $q^{(\alpha+p-1)(\beta+p-1)}=1$ for all $\beta\in(\CC \setminus \ZZ)\cup p\ZZ$. 
Obviously there are no solutions among typical modules. 
Similarly we see that the only atypical simple module that is transparent to all typicals $V_\beta$ is the tensor unit $\one$, recall the action in~\eqref{eq:UHq-action-S_sk}.

Next we show that all transparent objects are isomorphic to $\one^{\oplus m}$ for some $m$. It is easy to see that if an object is transparent, so are all its subquotients. In particular, all composition factors of transparent objects are again transparent. 
But since $\one$ is the only transparent simple object, all other transparent objects must have composition series only involving $\one$'s. 
The structure of the projective cover $P_{\one}$ of $\one$ (cf.\ Figure~\ref{fig:Loewy_UHq}) shows that $\one$ has no self-extensions, and hence any object all of whose composition factors are $\one$ is isomorphic to $\one^{\oplus m}$ for some $m$.

Finally, since $\catU_p^{\oplus}$ is the completion of $\catUfd_p$ with respect to countable direct sums and subquotients, the result for $\catUfd_p$ implies that also in $\catU_p^{\oplus}$ all transparent objects are (now possibly infinite) direct sums of $\one$.
\end{proof}

\subsection{The algebra $\algC$ and its local modules}\label{sec:Lam-and-local-Uq}

Recall the definition of the category of $\mathcal{H}^{\oplus}$ from Section~\ref{sec:ex-Cgr}. We choose $r=1/p$ (cf.\ \eqref{eq:r=1/p-convention}) and define the full subcategory $\mathcal{H}^{\oplus}_{p\ZZ} \subset \mathcal{H}^\oplus$ to consist of objects whose $\CC$-grades are contained in $p\ZZ$. This is in fact a monoidal subcategory. Note that the commutative algebra $\Lambda_p$ from \eqref{eq:free-boson-extension-alg} lies in $\mathcal{H}^{\oplus}_{p\ZZ}$.

The only one-dimensional modules of $\UQG{q}$ are $\repS_{1,k}$ with $k \in \mathbb{Z}$. By \eqref{eq:UHq-action-S_sk}, $E$ and $F$ act trivially on $\repS_{1,k}$ while $H$ acts as $kp$. We thus obtain a $\CC$-linear fully faithful functor
\be
	\mathcal{J} : \mathcal{H}^{\oplus}_{p\ZZ}
	\longrightarrow \catU_p^{\oplus} \ ,
\ee
which sends a $p\ZZ$-graded vector space $U$ to the $\UQG{q}$-module with trivial $E$ and $F$ action and $H$ action given by the grading. We have:

\begin{lemma}\label{lem:H-in-Cp-embed}
The fully faithful functor $\mathcal{J}$, together with identity coherence isomorphisms, is braided monoidal.\footnote{ 
We note that for $p$ even, the functor $\mathcal{J}$ does not respect the balancing isomorphisms. Comparing the action of $\theta$ in \eqref{eq:C-graded-vec-twist} and \eqref{eq:UHq-braid-bal} on the $H$-grade $pk$ gives a sign difference by $(-1)^{(p-1)k}$. This is related to the freedom of choosing a group homomorphism $G \to \{\pm 1\}$ which multiplies quadratic form to give the balancing in a ribbon category.
For $G = \CC$ there is no non-trivial continuous such homomorphism, while for $G = \ZZ$ there is $m \mapsto (-1)^{m}$. }
\end{lemma}

\begin{proof}
Monoidality is clear as the associator is that of vector spaces on both sides, and the $H$ action on the tensor product is given by $H \otimes \id + \id \otimes H$ in both cases.
To check that $\mathcal{J}$ respects the braiding note that in \eqref{eq:R-H} only the $n=0$ summand contributes, so that the resulting formula agrees with \eqref{eq:C-graded-vec-braiding}. 
\end{proof}

Let $\algC\in\catU_p^\oplus$ denote the infinite direct sum
\begin{equation}\label{eq:algC}
	\algC = \bigoplus_{m\in\ZZ} \repS_{1,2m} = \bigoplus_{m\in\ZZ}\CC\cdot 1_{m} \ ,
\end{equation}
with $\repS_{1,2m} = \CC$ the $1$-dimensional simple $\UQG{q}$-module of $H$-eigenvalue $2pm$ from above. 
We also introduced the notation $1_{m} = \stprp_0$ for the basis $\stprp_0 \in \repS_{1,2m}$ as in \eqref{eq:UHq-action-S_sk}.

Now observe that by construction of $\mathcal{J}$ we have $\algC = \mathcal{J}(\Lambda_p)$. Since $\mathcal{J}$ is braided monoidal, this endows $\Lambda$ with the structure of a commutative algebra with a non-degenerate invariant pairing, and this structure is unique up to isomorphism by Proposition~\ref{prop:comm-alg-unique} (this uses that $\mathcal{J}$ is fully faithful). 
Explicitly, the product on $\algC$ is given by
\begin{equation}\label{eq:muC}
\mu_{\algC}: \; \algC \otimes \algC \longrightarrow \algC
\quad , \quad
1_m \otimes 1_n \longmapsto 1_{m+n} \ .
\end{equation}

The subcategory of $\catU_p^\oplus$ which admits duals is $\catUfd_p$. 
By Definition~\ref{def:fg-module}, the finitely generated $\algC$-modules are hence those 
$M \in {}_\algC\catU_p^\oplus$ for which there is a surjection of 
$\algC$-modules $\mathrm{Ind}(X) \to M$ for some $X \in \catUfd_p$. This implies the inclusion
\be\label{eq:inclusing-fingen-findim}
	{}_\algC(\catU_p^\oplus)^\mathrm{fg}
	~\subset~
	\big\{ M \in {}_\algC\catU_p^\oplus
	\,\big|\,
	M \text{ has finite-dimensional $H$-eigenspaces } 
	\big\} \ .
\ee
Note that equality does not hold here  since there are simple objects in $\catUfd_p$ for a continuous range of $H$-eigenvalues. 
Hence the RHS contains an infinite direct sum of simple 
$\algC$-modules
with the property that the $H$-eigenvalues do not differ by an integer between any pair in the sum, while the LHS does not.
But for local modules equality does hold
(compare with Proposition~\ref{prop:Lambda-loc-mod} for $r=1/p$):

\begin{proposition}\label{prop:UHq-local-Lam-modules}
Let $M \in {}_\algC\catU_p^\oplus$. We have:
\begin{enumerate}
\item
$M$ is local if and only if all its $H$-eigenvalues are in $\mathbb{Z}$.
\item
Suppose $M$ is local. Then it is finitely generated if and only if all of its $H$-eigenspaces are finite-dimensional.
\end{enumerate}
\end{proposition}

\begin{proof}
\textsl{Part 1:} Recall the notation $1_k$ for homogeneous basis elements of $\algC$ from \eqref{eq:algC}. As in \eqref{eq:rho1t} we have that $\rho(1_k) : M \to M$ is invertible for each $k \in \ZZ$. 
It follows that the locality condition 
$\rho_M \circ c_{M,\algC} \circ c_{\algC,M} (1_k \otimes -) = \rho_M(1_k \otimes -)$
 for all $k\in\ZZ$ is equivalent to
\be
	c_{M,\algC} \circ c_{\algC,M} = \id_{\algC\otimes M} \ .
\ee
Let $v \in M$ be an $H$-eigenvector of eigenvalue $h$. 
As the action of $E$ and $F$ is trivial on~$\algC$, from the expression for $R$ in \eqref{eq:R-H} we see that only the Cartan part
$q^{H\otimes H}$ contributes when computing $c_{M,\algC}\circ c_{\algC,M}$.
  We find
\be
	c_{M,\algC} \circ c_{\algC,M} (1_k \otimes v)
	= q^{H \otimes H} (1_k \otimes v)
	= e^{ 2 \pi \rmi k h} 1_k \otimes v \ .
\ee
Clearly, this RHS is equal to $1_k \otimes v$ for all $k \in \ZZ$ if and only if $h \in \ZZ$, proving part 1.

\medskip

\noindent
\textsl{Part 2:} 
One direction follows from  \eqref{eq:inclusing-fingen-findim} even without using that $M$ is local. Suppose now that $M$ is local with finite-dimensional $H$-eigenspaces. By part 1, $M = \bigoplus_{k \in \ZZ} M_k$, where $M_k$ is the $H$-eigenspace for eigenvalue $k$. As in \eqref{eq:S} pick a set $S \subset \ZZ$ of representatives of $\ZZ_{2p}$ and consider the subspace $N = \bigoplus_{k \in S} M_k$ of $M$. Then $N$ is finite-dimensional,
and the $\CC$-linear map $\rho_M|_{\algC \otimes N}\colon \algC \otimes N \to M$ is surjective (as the $\rho(1_k)$ are isomorphisms). 
However, $N$ is typically not a $\UQG{q}$-module, i.e.\ not an object of 
$\catUfd_p$. 
Instead we need to replace $N$ by $\langle N \rangle$, 
the $\UQG{q}$-submodule of $M$ generated by $N$. We have seen in part 1 of the proof of Lemma~\ref{lem:direct-sum-completion} that
 $\langle N \rangle$ is finite-dimensional. Thus 
$\langle N \rangle \in \catUfd_p$ and 
 $\rho_M|_{\algC \otimes \langle N \rangle}\colon \algC \otimes \langle N \rangle \to M$ is clearly still surjective. 
\end{proof}

\section{A quasi-Hopf algebra modification of $\UresSL2$}\label{sec:qHopfUq}

In this section, we first recall the definition of the so-called restricted quantum group $\UresSL2$, where $q = e^{i\pi/p}$ is 
a primitive $2p$th root of unity.
Then we give our new quasi-Hopf algebra $\Q$, which is a slight modification of $\UresSL2$: while the algebra structure is the same, the coalgebra structure is modified by a central element and is endowed with a coassociator~$\Phi$.
In contrast to $\UresSL2$, the new quasi-Hopf algebra does possess an $R$-matrix.

\subsection{Restricted quantum group  for $sl(2)$}\label{sec:restricted-QG-noR}

The \textit{restricted}
quantum group $\UresSL2$, where $q = e^{i\pi/p}$ and $p\geq2$ is  an integer~\cite{ChPr,FGST}, is a Hopf algebra
 with generators $E$, $F$, and
$K^{\pm1}$ satisfying the defining  relations~\eqref{Uq-relations} together with
\begin{equation}\label{EpFpK2p-rel}
  E^{p}=F^{p}=0\ ,\quad K^{2p}=\one \ .
\end{equation}
The coalgebra structure and the antipode are as in~\eqref{eq:coprod}.
This defines a Hopf algebra of dimension $\mathrm{dim}\,\UresSL2 = 2p^3$.

As in Section~\ref{sec:ex-qHopf} we introduce the complete 
orthonormal set of idempotents $\e{n}$, $n\in \ZZ_{2p}$ defined by \eqref{eq:idemp-prim}	(which in this case are no longer primitive). 
They satisfy \eqref{eq:en-ortho-idem-prop}
and~\eqref{eq:K-en}:
they project onto $K$-eigenspaces, $K  \e{n} = q^{n} \e{n}$, and hence allow one to decompose $\UresSL2$-modules $M$ into $K$-eigenspaces as 
\be\label{eq:M-decomp-repeat}
M :=  \bigoplus_{n\in\ZZ_{2p}} M_n
\qquad 
\text{where} ~~ M_n := \e{n} M \ .
\ee
As in Section~\ref{sec:ex-qHopf} we have that $\e{n}m$ gives the $n$th weight component of an element $m\in M$.

Let $\repfd\, \UresSL2$ 
denote the abelian $\CC$-linear
	monoidal
category of finite-dimensional  $\UresSL2$-modules.
The category 
$\repfd\, \UresSL2$  is known to be non-braidable. Indeed,
for  $p>2$ one finds that the tensor product functor $\otimes$ is not commutative:
there are representations $U,V$ such that $U\tensor V$ is not isomorphic to $V\tensor U$ \cite{KS}. For $p=2$ the tensor product is commutative but there is no $R \in (\UresSL2)^{\otimes 2}$ which satisfies the conditions of a universal $R$-matrix~\cite{GR1}.
We  demonstrate the non-commutativity of the tensor product 	for $p>2$ in an example in Appendix~\ref{app:O}.

\subsection{Quasi-Hopf modification of $\UresSL2$}\label{sec:Q-def}
In this section we define the ribbon quasi-Hopf algebra $\Q$. We first give all its defining data and then state that this is indeed a ribbon quasi-Hopf algebra in Theorem~\ref{thm:equiv-1}. The proof is given in Section~\ref{sec:proof}.

Fix an odd integer $t$.
$\Q$ is a  quasi-Hopf algebra with the same algebra structure as $\UresSL2$, while the coalgebra structure is given by
  the modified coproduct 
\begin{align}\label{eq:cop-new}
\Ncop(E) &= E\otimes K + (\idem_0+\q^{t}\idem_1)\otimes E \ ,
\nonumber \\
\Ncop(F) &= F\otimes 1 +  (\idem_0+\q^{-t}\idem_1)K^{-1}\otimes F \ 
\\
\Ncop(K)&=K\otimes K \ ,
\nonumber
\end{align}
where we used the central idempotents
\be\label{eq:idem-def}
\idem_0 = \ffrac{\one+K^p}{2} 
\quad , \qquad \idem_1 =\one - \idem_0 \ .
\ee
In terms of the idempotents $e_n$, they are given by 
\be\label{eq:idem-e}
\idem_0 = \sum_{a \in \ZZ_p} \e{2a}\ , \qquad \idem_1 = \sum_{a \in \ZZ_p} \e{2a+1}\ .
\ee

The counit $\eps$ is the same as for $\UresSL2$.
The coproduct is non-coassociative 
in general (see, however, Remark~\ref{rem:p-odd}\,(3) below)
and is equipped with the co-associator
\begin{equation}\label{eq:Phi-t}
\Phi_{t}= \one\otimes\one\otimes\one +  \idem_1\otimes\idem_1\otimes \bigl(K^{-t}-\one\bigr)\ .
\end{equation}
The antipode is also modified by central elements:
\be\label{eq:antipode_St}
S_t(E)= -EK^{-1}(\idem_0 + q^t\idem_1) ~~ , \quad
S_t(F) = -KF (\idem_0+q^{-t}\idem_1) ~~ , \quad
S_t(K) = K^{-1}\ .
\ee
and the (co)evaluation elements are
\be
\Salpha_t = \one \quad , \qquad
\Sbeta_t = \idem_0 + K^{-t}\idem_1\ .
\ee

Let 
\be\label{eq:catQ-def}
\catQ_p \,:=\, \repfd\, \Q
\ee
denote the abelian $\CC$-linear rigid 
monoidal category of finite-dimensional  $\Q$-modules. 
In Appendix~\ref{app:O} we illustrate in an example how the modified coproduct $\Ncop$ resolves the problem with non-commutativity we saw above in $\repfd\, \UresSL2$. 
And indeed it is now possible to 
endow $\catQ_p$ with a braiding and a ribbon structure.
Let us introduce the two special elements: 
\begin{align}\label{eq:R-quasiH}
R_t &= \ffrac{1}{4p}\sum_{n=0}^{p-1} \sum_{s,r=0}^{2p-1}\ffrac{(q-q^{-1})^n}{[n]!}q^{\frac{n(n-1)}{2}-2sr}\bigl(1 +q^{t r} + q^{-t(n+s)} + q^{\half t^2 +t r -t(n+s)} \bigr) K^sE^n\otimes K^r F^n 
\ , \\
\label{eq:ribbon-quasiH}
\ribbon_t &= \ffrac{1-\rmi}{2\sqrt{p}}  \sum_{n=0}^{p-1} \sum_{j\in\ZZ_{2p}} \ffrac{(q-q^{-1})^{n}}{[n]!} 
 q^{n(j-\half) + \half(j+p+1)^2} F^nE^n K^j\ .
\end{align}
Note that $R_t$ is an element of
$\Q^{\geq0}\otimes \Q^{\leq0}$, i.e.\ each factor sits in an opposite Borel half of $\Q$. This is in contrast to the $p=2$ case in \cite{GR1}, where the coproduct was kept the same as in $\UresSL2$ at the price of a much more complicated coassociator and $R$-matrix.
The two quasi-Hopf algebras are however twist-equivalent (see also Remark~\ref{rem:p-odd}\,(2) below).
The ribbon element does not actually depend on $t$, however to emphasise that the whole ribbon structure does depend on $t$ and to distinguish it from the ribbon element $\ribbon$ of $\UQG{q}$ we use $\ribbon_t$. 

\medskip
Recall from Section~\ref{sec:Lam-and-local-Uq} the category $\fgloc{\algC}{{(\catU_{\mathit{p}}^\oplus)}}$ of finitely-generated local $\Lambda$-modules in $\catU_p^\oplus = \repwt \UQG{q}$ (cf.\ Lemma~\ref{lem:direct-sum-completion}).
By the next theorem we realise the (rather abstract) category $\fgloc{\algC}{{(\catU_{\mathit{p}}^\oplus)}}$  in terms 
of representations of the restricted quantum group.

\begin{theorem}\label{thm:equiv-1}
\mbox{}
\begin{enumerate}
\item
  $\Q$ with the data $(\cdot, \one, \Ncop, \eps, \Phi_t, R_t, \ribbon_t)$ introduced  above is a factorisable
   ribbon quasi-Hopf algebra, and two quasi-Hopf algebras for different $t$ are twist-equivalent.

\item  There is an equivalence of ribbon categories between 
	$\catQ_p=\repfd\, \Q$ 
and 
	$\fgloc{\algC}{{(\catU_{\mathit{p}}^\oplus)}}$
	where $\catU_p^\oplus = \repwt \UQG{q}$.	
\end{enumerate}
\end{theorem}

The proof is based on an explicit construction of the functors
in the diagram
\begin{equation}
\xymatrix@C=55pt@M=9pt@W=9pt
{
	\catQ_p
	\ar@/^2ex/[]!<-5ex,-2ex>;[r]!<-5ex,0.8ex>^{{}\qquad\qquad\funF}
	& 
	\fgloc{\algC}{{(\catU_{\mathit{p}}^\oplus)}}
	\ar@<-5pt>@/^2ex/[]!<-5ex,-2ex>;[l]!<-5ex,0.8ex>^{{}\qquad\qquad\funG}
}
\end{equation}
and is delegated to Section~\ref{sec:proof}, as it is rather long.
The proof follows steps \ref{plan-step1}--\ref{plan-step5} in Section~\ref{sec:ex-qHopf}:
we provide an explicit formulation of the equivalence $\fun$ and the coherence isomorphisms $\fun_{M,N}$
 and then verify conditions \eqref{eq:transport-assoc-diag}--\eqref{eq:transport-twist}. 
Since we know that $\fgloc{\algC}{{(\catU_{\mathit{p}}^\oplus)}}$ is ribbon, this implies at the same time that $\Phi_t$, $R_t$ and $\ribbon_t$ satisfy all conditions for $\Q$ to be a ribbon quasi-Hopf algebra
 and that $\funF$ is a ribbon equivalence,	
 cf.\ Remark~\ref{rem:solve-pent-hex-autom}.
 
\medskip
For later reference we also compute the monodromy matrix.

\begin{proposition}
The monodromy matrix $M_t$ of $\Q$ is given by
\begin{align}
  & M_t = (R_t)_{21} R_t
\nonumber\\
  &= \ffrac{1}{2p}
  \sum_{m,n=0}^{p-1}
  \sum_{i,j=0}^{2p-1}
  \ffrac{(\q - \q^{-1})^{m + n}}{[m]! [n]!}\,
  \q^{\half m(m - 1) + \half n(n - 1)}
   \q^{- m^2  + m(j-i) - i j} 
\nonumber\\
  &\hspace{4em}\times \left(
  \tfrac12(1+(-1)^{i+m}) + \tfrac12(1-(-1)^{i+m})q^{t(m-n)} 
  \right)
  K^{j} F^{m} E^{n} \tensor K^{i} E^{m} F^{n} .
\label{eq:M}
  \end{align}
\end{proposition}

\begin{proof}
Substituting the expression for $R_t$ in \eqref{eq:R-quasiH} immediately gives
\begin{align}
	(R_t)_{21} \, R_t
	= 
	\ffrac{1}{(4p)^2}
  \sum_{m,n=0}^{p-1}
  \sum_{i,j=0}^{2p-1}
  &
  \ffrac{(\q - \q^{-1})^{m + n}}{[m]! [n]!}\,
  \q^{\half m(m - 1) + \half n(n - 1)}
\nonumber\\  
  & \hspace{2em} \times Y_{m,n,i,j} ~ K^{j} F^{m} E^{n} \tensor K^{i} E^{m} F^{n} \ ,
  \label{eq:M-proof-aux1}
\end{align}
where
\begin{align}
Y_{m,n,i,j} = 
\sum_{s,r=0}^{2p-1}
q^{2(-2r+j+m)s}
\left( AC + q^{-ts}(AD+BC) + 
q^{-2ts}BD \right) E
\end{align}
and
\begin{align}
A &= 1 +q^{t r} \ ,
&
B &= q^{-tm}\left( 1 + q^{t(\half t + r)} \right) \ ,
\nonumber\\  
C &= 1+q^{t(r-n-j)} \ ,
&
D &=  q^{ti}\left( 1 + q^{t(\half t + r - n-j)} \right) \ ,
&E &= q^{2(- ij + ri + m(j-i-r))} \ .
\end{align}
In the expression for $Y_{m,n,i,j}$ we displayed all $s$-dependent terms. To carry out the sum over~$s$ we use $\sum_{s=0}^{2p-1} q^{sx} = 2p\,\delta[x \equiv 0 \,(\mathrm{mod}\, 2p)]$, where $\delta[x \equiv y \,(\mathrm{mod}\,z)]$ denotes the Kronecker-delta modulo $z$. Since $t$ is odd, the $\delta$-function arising from the summand $q^{-ts}$ does not contribute. We arrive at
\be
Y_{m,n,i,j} = 
2p\sum_{r=0}^{2p-1}
\bigl( AC \delta[2r \equiv j+n \,(\mathrm{mod}\, p)] + 
 BD \delta[2r+t \equiv j+n \,(\mathrm{mod}\, p)] \bigr) E \ .
\ee
Note that if $r$ satisfies the condition of either $\delta$-function, the so does $r+p$. Hence terms of the form $q^{tr}\cdot (\text{coeff. indep. of $r$})$ occur twice with opposing sign and cancels. Removing these terms one finds that only $2r$ occurs in the various exponents of $q$. In particular, we can restrict the sum to  run over $r=0,\dots,p-1$ at the cost of an overall factor of 2.

Next we replace $\delta[2r \equiv j+n \,(\mathrm{mod}\, p)] = \delta[2r \equiv j+n \,(\mathrm{mod}\, 2p)]+\delta[2r \equiv j+n+p \,(\mathrm{mod}\, 2p)]$ and similarly for the other $\delta$-function. Since $t$ is odd, with the new range for $r$, the value $2r$ covers all even values in $0,1,\dots,2p-1$ while $2r+t$ covers all odd values. We can therefore replace $2r$ and $2r+t$ by a new variable $x$ with range $0,\dots,2p-1$. The overall result is
\begin{align}
Y_{m,n,i,j} &= 
4p \, q^{2(-ij+m(j-i))} \sum_{x=0}^{2p-1} q^{x(i-m)}(1+q^{t(x-n-j)})
\nonumber\\
& \hspace{7em}
\times 
\big( \delta[x \equiv j+m \,(\mathrm{mod}\, 2p)] + \delta[x \equiv j+m+p \,(\mathrm{mod}\, 2p)] \big) \ .
\end{align}
Carrying out the sum over $x$ and substituting the resulting expression for $Y_{m,n,i,j}$ into \eqref{eq:M-proof-aux1} results in \eqref{eq:M}.
\end{proof}

\begin{remark}~\label{rem:p-odd}
\begin{enumerate}\setlength{\leftskip}{-1em}
\item
We require $t$ to be odd in the definition of $\Q$, but let us nonetheless formally set $t=0$ for the moment.
In this case, the coproduct \eqref{eq:cop-new} and the antipode \eqref{eq:antipode_St} reduce to those of the restricted quantum group $\UresSL2$ from Section~\ref{sec:restricted-QG-noR}, and the coassociator and (co)evaluation elements become identities. 
The ribbon element $\ribbon_t$ in \eqref{eq:ribbon-quasiH} is the same as the one found in~\cite[Sec.\,4.6]{FGST} for $\UresSL2$ (which, however, does not allow for an $R$-matrix). 
Similarly, the monodromy matrix $M_t$ in \eqref{eq:M} reduces to the one obtained for $\UresSL2$ in~\cite[Sec.\,4.2]{FGST}.

One can also verify that the Grothendieck ring of $\repfd\Q$ is equal to that of $\repfd\UresSL2$, even though the categories themselves are not tensor equivalent (one being braided and the other not braidable).
In fact, even the decomposition of the tensor product of two simple modules into indecomposable summands is the same in both cases.
(The distinction is however in decompositions of more complicated modules as outlined in Appendix~\ref{app:O}.)

\item 
Let $\mathcal{V}_\mathrm{ev}(N)$ be the even part of the super-VOA of $N$ pairs of symplectic fermions with total central charge $c=-2N$ \cite{Abe:2005}. 
In \cite[Sec.\,3]{FGR2} 
a ribbon quasi-Hopf algebra 
$\mathsf{Q}(N,\beta)$ is given and it is conjectured that for $\beta = e^{\pi i N/4}$ there is a ribbon equivalence $\repfd\,\mathsf{Q}(N,\beta) \cong \rep \mathcal{V}_\mathrm{ev}(N)$,  
see \cite[Rem.\,6.10]{FGR2}.
Since $\TripAlg{p}$ for $p=2$ is isomorphic to $\mathcal{V}_\mathrm{ev}(N)$ for $N=1$ \cite{Kausch:1995py,Abe:2005}, 
we expect at least a twist-equivalence between $\Q$ and $\mathsf{Q}(N,\beta)$. In fact, if we choose $t=1$ we even have an isomorphism $\Q \to \mathsf{Q}(N,\beta)$ of ribbon quasi-Hopf algebras which acts on generators as 
(see \cite{FGR2} for notation)
\be
	E \mapsto \mathsf{f}^-\mathsf{K}
	~~ , \quad
	F \mapsto i \mathsf{f}^+ 
	~~ , \quad
	K \mapsto \mathsf{K} \ .
\ee
As a  consequence of this isomorphism we have that
the result~\cite[Thm.\,6.10]{FGR2} on torus mapping class group action applies to $\Q$ at $q=\mathrm{i}$: the $SL(2,\ZZ)$-action coming from pseudo-trace functions \cite{Arike:2011ab}
of $\TripAlg{2}$ agrees with the categorical one 
	 of \cite{Lyubashenko:1995,Lyubashenko:1994tm}.

\item
The quasi-Hopf algebra $\Q$ for odd values of $p$ has a particularly simple coproduct when taking the odd parameter $t$ to be equal to $p$,
\be\label{eq:cop-odd-p}
\Delta_p(E) = E\otimes K + K^p\otimes E\ , \qquad
\Delta_p(F) = F\otimes 1 +  K^{p-1}\otimes F
\ee
(so the standard coproduct is just modified by the central element $K^p$) and the co-associator~\eqref{eq:Phi-t} takes the form
\be
\Phi_p =  \one\otimes\one\otimes\one -2  \idem_1\otimes\idem_1\otimes \idem_1 \ .
\ee
It is central and corresponds to the generator of
$H^3(\mathbb{Z}_2,\CC^\times) = \langle \, [\omega] \, \rangle \cong \mathbb{Z}_2$. Here, the normalised 3-cocycle $\omega$ representing the generator has non-trivial value $\omega(1,1,1)=-1$ (we write $\mathbb{Z}_2 = \{0,1\}$ additively).
Our coassociator $\Phi_p$ is then simply
\be
\Phi_p = \sum_{i,j,k \in \mathbb{Z}_2} \omega(i,j,k) \cdot \idem_i \otimes \idem_j \otimes \idem_k \ .
\ee
\end{enumerate}
\end{remark}

\subsection{$\Q$ modified by abelian cocycles of $\ZZ_2$}

In this section we use abelian group cohomology to obtain four ribbon quasi-Hopf algebras $\Qb$ parametrised by $\zeta \in \CC$ with $\zeta^4=1$, such that $\Qb = \Q$ for $\zeta=1$. The material in this section is not used later and can be skipped. 

\medskip

The central idempotents $\idem_i$ from~\eqref{eq:idem-def} 
provide a decomposition of $\Q$ into the ideals
\begin{equation}\label{Q-decomp}
\Q = \idem_0\cdot\Q \oplus \idem_1\cdot\Q 
\end{equation}
and a corresponding $\ZZ_2$-decomposition of the category $\catQ_p$,
\be\label{eq:catQ-decomp}
\catQ_p  = \catQ^0_p  \oplus \catQ_p^1\ .
\ee
{}From the coproduct formulas
\begin{equation}\label{eq:cop-idem}
\Delta_t(\idem_0) = \idem_0\tensor\idem_0 + \idem_1\tensor\idem_1\ ,\qquad
\Delta_t(\idem_1) = \idem_0\tensor\idem_1 + \idem_1\tensor\idem_0\ 
\end{equation}
 we see that the tensor product $\tensor$ functor in $\catQ_p$ respects the $\ZZ_2$-grading~\eqref{eq:catQ-decomp}: 
 \be
 \begin{split}
 &U\tensor V\in \catQ^0_p \quad  \text{for} \quad (U,V)\in\catQ^i_p\times\catQ^i_p \ ,\\
 &U\tensor V\in \catQ^1_p \quad  \text{for} \quad (U,V)\in\catQ^i_p\times\catQ^{i+1}_p\ , 
 \end{split}
 \qquad \text{for} \; i\in\ZZ_2\ .
 \ee
 We can therefore modify the associator by   3-cocycles for $\mathbb{Z}_2$, as those satisfy the pentagon axiom. Furthermore, if one wants to take into account the braiding then one can modify  the braided monoidal structure in $\catQ_p$ by  abelian 3-cocycles
for $\mathbb{Z}_2$,	cf.\ Section~\ref{sec:ex-Cgr}.
Since we are interested in modifications up to   braided monoidal equivalences, it is enough to use classes from 
$H^3_{ab}(\mathbb{Z}_2,\CC^\times)$. 
 
{}From the classification via quadratic forms one sees that $H^3_{ab}(\mathbb{Z}_2,\CC^\times)
	\cong
\mathbb{Z}_4$.
Writing $\mathbb{Z}_2 = \{0,1\}$ additively and $\mathbb{Z}_4$ multiplicatively, a class $[\xi]\in H^3_{ab}(\mathbb{Z}_2,\CC^\times)$ is described by its representative  $\xi= (\omega,\sigma)$ where  $\sigma$ is a 2-cochain with only non-trivial value  $\sigma(1,1)=\zeta$ such that $\zeta^4=1$, and  $\omega$ is a 3-cocycle for group-cohomology such that $\omega(1,1,1)=\zeta^2$ and~$1$ else.
Multiplying the coassociator $\Phi_t$ by the 3-cocycle $\omega$,
\be
\Phi_{t,\zeta} := \Phi_t \cdot \sum_{i,j,k \in \mathbb{Z}_2} \omega(i,j,k) \cdot \idem_i \otimes \idem_j \otimes \idem_k \ ,
\ee
 and the $R$-matrix by the 2-cochain $\sigma$,
\be
R_{t,\zeta} := R_t \cdot \sum_{i,j \in \mathbb{Z}_2} \sigma(i,j) \cdot \idem_i \otimes \idem_j \ ,
\ee
we get (for a given $p\geq2$) four quasi-triangular quasi-Hopf algebras, which are not twist equivalent to each other, or four  inequivalent braided monoidal categories parametrised by the  $4$th roots of unity $\zeta$. The coproduct $\Delta_t$ stays of course non-modified.

We finally notice that the quasi-Hopf algebra $\Q$ from Remark~\ref{rem:p-odd}\,(3)
for odd values of~$p$ can be modified this way to a Hopf algebra with
\be
\Phi_{t,\pm\rmi} = \one\tensor\one\tensor \one\ ,
\ee
while having the same coproduct $\Delta_p$ in~\eqref{eq:cop-odd-p}. This also shows that $\Delta_p$ is actually co-associative (in contrast to general $\Delta_t$).

\section{Singlet and triplet vertex operator algebras}\label{sec:singlet}
In this section we discuss the singlet and triplet VOAs, and some aspects of their representation theory. We describe how the triplet VOA can be realised as a simple current extension of the singlet VOA. 
The main aim of this section is to illustrate how the passage from $\UQG{q}$ to $\Q$,
formulated in the previous section in Theorem~\ref{thm:equiv-1},
  mirrors the extension from the singlet to the triplet algebra. This relies on the conjecture 
that the representation category of the singlet VOA is
ribbon equivalent to $\catUfd_p$.

\subsection{The singlet and triplet VOAs $\SingAlg{p}$ and $\TripAlg{p}$}\label{sec:sing-trip-def}

The triplet VOA $\TripAlg{p}$ was first discussed in 
\cite{Kausch:1990vg,GK} 
and has since been extensively investigated. The singlet VOA $\SingAlg{p}$ 
and its representations were first considered in~\cite{Fl}, building on \cite{Kausch:1990vg}.
Here we will outline the definition of $\TripAlg{p}$ as a sub-VOA of a rank-one lattice VOA and that of $\SingAlg{p}$ as an invariant sub-VOA of $\TripAlg{p}$.

\medskip

The lattice construction of $\TripAlg{p}$ can be found in 
\cite{Kausch:1990vg, FHST, AM2},
here we follow \cite[Sec.\,2]{CRW}.
Let 
\be
	p \in \ZZ_{\ge 2} \ ,
\ee
set
\be\label{eq:alpha+-0_def}
\alpha_+=\sqrt{2p}
\quad , \quad
\alpha_-=-\sqrt{2/p}
\quad , \quad
\alpha_0=\alpha_++\alpha_- \ ,
\ee
and consider the lattice $L=\alpha_+\ZZ$ with dual lattice $L'=\alpha_+^{-1}\ZZ$ . Denote by $V_L$ the corresponding lattice VOA and by $\mathsf{H}$ its Heisenberg sub-VOA. We choose the Virasoro field in $\mathsf{H} \subset V_L$ such that the central charge is
\be 
c=1-6\frac{(p-1)^2}{p} \ .
\ee
With respect to this (twisted) Virasoro field the Fock module of highest-weight $\lambda$ has conformal weight
\be\label{eq:Heis_deformed-conf-weight}
h_\lambda = \tfrac12 \lambda(\lambda-\alpha_0) \ .
\ee

Let  $Q_-$  be  the zero-mode of a screening current corresponding to the highest-weight vector 
of the Fock module $\mathcal F_{\alpha_-}$ of weight $\alpha_-$. This means $Q_-$ maps the Fock module $\mathcal F_\lambda$ to $\mathcal F_{\lambda+\alpha_-}$ and the lattice VOA module $V_{L+\lambda}$ corresponding to the coset $L+\lambda \in L'/L$ to the module $V_{L+\lambda+\alpha_-}$.
 The triplet algebra is the kernel of the map $Q_-: V_L\rightarrow V_{L+\alpha_-}$ from the lattice VOA $V_L$ to its module $V_{L+\alpha_-}$ associated to the coset $L+\alpha_-$,
\be
\TripAlg{p} \,:=\,
\text{ker}\left(Q_- \colon V_L\rightarrow V_{L+\alpha_-}\right).
\ee
   The kernel of the  restriction of this map to the Heisenberg sub-VOA is the singlet algebra $\SingAlg{p}$ \cite{Ad} 
\be\label{eq:Mp-Q-}
\SingAlg{p} 
	\,:=\,
\text{ker}\left(Q_-\colon \mathcal F_0\rightarrow \mathcal F_{\alpha_-}\right).
\ee 
Thus altogether we have the inclusions
\be
\xymatrix{
   \mathsf{H} \ar@{}[r]|-*[@]{\subset}  
   & 
   V_L
   \\
   \SingAlg{p} \ar@{}[r]|-*[@]{\subset} \ar@{}[u]|-*[@]{\subset} 
   & 
   \TripAlg{p} \ar@{}[u]|-*[@]{\subset} \quad ,
}
\ee
where the Virasoro field is the same in all four VOAs.

The singlet VOA $\SingAlg{p}$ is strongly generated by the Virasoro field and one more field (hence the name singlet) of conformal dimension $2p-1$
\cite[Thm.~3.2]{Ad}. It is not $C_2$-cofinite as we will see in Section~\ref{sec:Singlet} that it has infinitely many simple modules. 
The triplet VOA $\TripAlg{p}$ is an extension of the singlet VOA $\SingAlg{p}$. It is strongly generated by the Virasoro field and three more fields (hence the name triplet)
of conformal dimension $2p-1$
\cite[Prop.\,1.3]{AM2}, see also \cite{FHST}.
The triplet VOA is $C_2$-cofinite \cite{Carqueville:2005nu,AM2}.

The lattice VOA $V_L$ carries a $gl(1)$ action
via the zero mode of the Heisenberg field,\footnote{
 $\TripAlg{p}$ actually carries an $sl(2)$-action, not only a $gl(1)$-action. This action integrates to an action of $PSL(2, \CC)$ which in turn is the full automorphism group of $\TripAlg{p}$ \cite{ALM}.}
where $u \in gl(1)$ acts on the Fock modules $\mathcal F_{\lambda} \subset V_L$  by multiplication with $u \lambda$.
By construction, only $\mathcal F_0$ is invariant under the $gl(1)$-action, so that $\SingAlg{p}$ is the invariant subalgebra:
\be\label{eq:sing-is-C*-inv-subspace}
	\SingAlg{p} = \TripAlg{p}^{gl(1)} \ .
\ee

\subsection{Modules of the singlet VOA $\SingAlg{p}$} \label{sec:Singlet}

Here we review the representation theory of $\SingAlg{p}$. 
Original works on $\SingAlg{p}$ and its modules are 
\cite{Fl,Ad,AM1, CM, CMR} and the case of $\SingAlg{2}$ is reviewed in \cite[Sec.\,3.3]{CR}.
Let $\rep \SingAlg{p}$ be the category of grading restricted generalised  $\SingAlg{p}$-modules 
(see \cite{HLZ} and \cite[Def.\,3.1]{CKM} for details 
-- in particular such modules have a direct sum decomposition into finite-dimensional generalised $L_0$-eigenspaces). 
A complete classification of simple modules in $\rep \SingAlg{p}$ is given in \cite{Ad} and we now introduce them. 

Let $\alpha_{\pm}$ and $\alpha_0$ be as in \eqref{eq:alpha+-0_def} and set 
\be
\alpha_{r, s} =\frac{1-r}{2}\alpha_+ +\frac{1-s}{2}\alpha_- \ .
\ee 
The typical $\SingAlg{p}$-modules
are denoted by $\SingTyp{\lambda}$ for real $\lambda$. 
The conformal dimension of the state of lowest conformal weight of $\SingTyp{\lambda}$ is
	given by \eqref{eq:Heis_deformed-conf-weight}.

The $\SingTyp{\lambda}$ are simple except if there is a pair of integers $(r, s)$ with $1\leq s\leq p$ such that $\lambda=\alpha_{r, s}$. In the latter instance we have the non-split short exact sequence
\begin{equation} \label{ses:singlet}
\dses{\SingAtyp{r,s}}{}{\SingTyp{\alpha_{r,s}}}{}{\SingAtyp{r+1, p-s}}\ ,
\end{equation}
with atypical simple $\SingAlg{p}$-modules
$\SingAtyp{r, s}$, see~\cite[Section\,2.3]{CRW}.
The conformal dimension of the state of lowest conformal weight of
	$\SingAtyp{r,s}$ is 	given by \eqref{eq:Heis_deformed-conf-weight}
	with $\lambda = \alpha_{r,s}$ if $r\geq 1$ and $\lambda = \alpha_{2-r,s}$ if $r\leq 0$.

\medskip

We denote by $\boxtimes$ the $P(1)$-tensor product of \cite{HLZ}, see \cite[Sec.\,3.2]{CKM} for a short review. 
The fusion rules of $\SingAlg{p}$-modules
are not known, except in some instances for $p=2$ \cite{AM3}. There has however been a computation using a conjectural Verlinde formula  \cite{CM} (see also \cite{RW1, RW2}) suggesting the following 
Grothendieck ring of  $\rep \SingAlg{p}$:
\begin{equation}\label{CM-fusion}
\begin{split}
[\SingTyp{\lambda}\fuse   \SingTyp{\mu}] &=  \sum_{\ell=0}^{p-1}[\SingTyp{\lambda+\mu+\ell\alpha_-}]\gc\\
 [\SingAtyp{r ,s} \fuse  \SingTyp{\mu}] &=
\sum_{\substack{\ell=-s+2\\ \ell+s=0\, \mathrm{mod}\, 2}}^{s}[\SingTyp{\mu-\frac{1}{2}(r\alpha_+ +\ell\alpha_- -\alpha_0)}]\gc\\
[\SingAtyp{r ,s} \fuse  \SingAtyp{r',s'}] &=
\sum_{\substack{\ell=|s-s'|+1\\ \ell+s+s'=1\, \mathrm{mod}\, 2}}^{\mathrm{min} \{ s+s'-1,p \}}[\SingAtyp{r+r'-1,\ell}] \\
& + \sum_{\substack{\ell=p+1\\ \ell+s+s'=1\, \mathrm{mod}\, 2}}^{s+s'-1}\Bigl([\SingAtyp{r+r'-2,\ell-p}]+
[\SingAtyp{r+r'-1,2p-\ell}]+ [\SingAtyp{r+r',\ell-p}]\Bigr).
\end{split}
\end{equation} 
In the case of $\SingAlg{2}$ these computations have been confirmed by determining explicit fusion rules \cite{AM3}.
The ribbon twist is not known but provided that we have a vertex tensor category structure on
$\rep \SingAlg{p}$ it is given by the action of $e^{2\pi i L_0}$ which acts as $e^{2\pi i h_\lambda}$ on $\SingTyp{\lambda}$ and as $e^{2\pi i h_{\alpha_{r, s}}}$ on $\SingAtyp{r,s}$.

It is expected that every object in $\rep \SingAlg{p}$ is of finite Jordan-H\"older length and that every finitely generated generalised grading-restricted $\SingAlg{p}$-module is $C_1$-cofinite
 (see \cite[Sec.\,6.3]{CMR}). 
If true, then by \cite{HLZ} the category  $\rep \SingAlg{p}$ can be given the structure of a vertex tensor category \cite[Thm.\,17]{CMR}. 
We would furthermore expect the category to be rigid, so that altogether:

\begin{conjecture}\label{conj:M-ribbon}
The category $\rep \SingAlg{p}$ is a rigid vertex tensor category (and so in particular a ribbon category).
 \end{conjecture}
 
	We denote by $\repsimple\SingAlg{p}$ the smallest full subcategory of $\rep \SingAlg{p}$ which contains all simple modules and which is complete with respect to taking tensor products, finite sums and subquotients. 

\begin{remark}
The reason for introducing  $\repsimple \SingAlg{p}$ is that typical singlet algebra modules allow for self-extensions that do not have an analogue in the category of weight modules of the unrolled quantum group. 
This is similar to Fock modules of the Heisenberg VOA which allow for self-extensions in which the zero-mode of the Heisenberg field do not act semisimply, cf.\ Remark~\ref{rem:HeisenbergVOA}\,(1). 
 In \cite[Sec.\,3.1]{CMR} analogous self-extensions for modules of $\UQG{q}$ are considered, but these are no weight modules anymore. 
Another more general type of $\SingAlg{p}$-modules, so-called Whittaker modules, were investigated in \cite{T}, but these are not even $\NN$-gradable.
\end{remark}
\begin{figure}[tb]
\begin{center}
\begin{tikzpicture}[scale=0.70][thick,>=latex,
nom/.style={circle,draw=black!20,fill=black!20,inner sep=1pt}
]
\node(left0) at  (-1,0) [] {$\SingStag{r,s}$:}; 
\node (top1) at (5,2.5) [] {$\SingAtyp{r, s}$};
\node (left1) at (2.5,0) [] {$\SingAtyp{r-1, p-s}$};
\node (right1) at (7.5,0) [] {$\SingAtyp{r+1, p-s}\ $,};
\node (bot1) at (5,-2.5) [] {$\SingAtyp{r, s}$};
\draw [->] (top1) -- (left1);
\draw [->] (top1) -- (right1);
\draw [->] (left1) -- (bot1);
\draw [->] (right1) -- (bot1);

\node (top2) at (12, 2) [] {$\SingTyp{\alpha_{r-1, p-s}}$};
\node (bot2) at (12, -2) [] {$\SingTyp{\alpha_{r, s}}$};
\draw [->] (top2) -- (bot2);

\end{tikzpicture}
\captionbox{\label{fig:Loewy-Mp} 
Loewy diagram of the $\SingAlg{p}$-module $\SingStag{r,s}$ from Proposition~\ref{prop:singstag}
in terms of simple (left diagram) and typical (right diagram) composition factors. Conjecturally, $\SingStag{r,s}$ is the
projective cover of the simple module $\SingAtyp{r,s}$.
}{\rule{12cm}{0cm}}
\end{center}
\end{figure}

\begin{proposition}\label{prop:singstag}
$\SingAlg{p}$
has logarithmic modules $\SingStag{r,s}$ satisfying the non-split short-exact sequence
\be\label{eq:P-is-log-SES}
\dses{\SingTyp{\alpha_{r,s}}}{}{\SingStag{r,s}}{}{\SingTyp{\alpha_{r-1,p-s}}} \ ,
\ee
and having a non-semisimple action of the zero-mode $L_0$ of $\SingAlg{p}$.
\end{proposition}

This proposition will be proved in Section~\ref{sec:Triplet} after giving some properties of triplet algebra modules that we will need.
Conjecturally, in $\repsimple \SingAlg{p}$
the $\SingStag{r,s}$ are the projective covers of the simple atypical modules $\SingAtyp{r,s}$ (cf.\ Conjecture~\ref{conj:M-C} below).
The Loewy diagram of $\SingStag{r,s}$ is given in Figure~\ref{fig:Loewy-Mp}.

\medskip

In order to study the triplet VOA $\TripAlg{p}$ as an extension of $\SingAlg{p}$ we need to allow for infinite direct sums. Let thus $\repsimpleinfinity \SingAlg{p}$ be the completion of $\repsimple \SingAlg{p}$ with respect to countably infinite direct sums and subquotients. If Conjecture~\ref{conj:M-ribbon} is true, then this category inherits a braided tensor category structure from $\rep \SingAlg{p}$.

\subsection{Modules of the triplet VOA $\TripAlg{p}$} \label{sec:Triplet}
Similar to the singlet case by a $\TripAlg{p}$-module we mean a generalised, grading restricted $\TripAlg{p}$-module.
The triplet VOA has up to isomorphism $2p$ simple modules $\TripIrr{r, s}$ parametrised by $r=1,2$ and $1\leq s\leq p$, and only two of them are projective, those for $s=p$.
 Furthermore, for $1\leq s\leq p-1$ and $r=1,2$ there are modules $\TripTyp{\alpha_{r,s}}$ satisfying the non-split exact sequence~\cite{FHST}
\begin{equation} \label{ses:triplet}
\dses{\TripIrr{r,s}}{}{\TripTyp{\alpha_{r,s}}}{}{\TripIrr{3-r, p-s}} \ .
\end{equation}
The projective simple modules  
$\TripIrr{r, p}$ agree with $\TripTyp{\alpha_{r, p}}$.
The lowest $L_0$ degree in
	$\TripTyp{\alpha_{r,s}}$ is 
	given by \eqref{eq:Heis_deformed-conf-weight}
	with $\lambda = \alpha_{r,s}$. 
The same applies to $\TripIrr{r,s}$.

The modules $\TripTyp{\alpha_{r,s}}$ can further be extended to the projective covers $\TripStag{r, s}$ 
with non-split exact sequence
\cite{Nagatomo:2009xp} (see also the earlier development in~\cite{GK,[FFHST],Feigin:2005xs})
\begin{equation} \label{ses:triplettyp}
\dses{\TripTyp{\alpha_{r,s}}}{}{\TripStag{r,s}}{}{\TripTyp{\alpha_{3-r, p-s}}} \ .
\end{equation}
This is best pictured in terms of the Loewy diagram given in Figure~\ref{fig:Loewy-Wp}.

\begin{figure}[tb]
\begin{center}
\begin{tikzpicture}[scale=0.70][thick,>=latex,
nom/.style={circle,draw=black!20,fill=black!20,inner sep=1pt}
]
\node(left0) at  (-1,0) [] {$\TripStag{r,s}$:}; 
\node (top1) at (5,2.5) [] {$\TripIrr{r, s}$};
\node (left1) at (2.5,0) [] {$\TripIrr{3-r, p-s}$};
\node (right1) at (7.5,0) [] {$\TripIrr{3-r, p-s}\ $,};
\node (bot1) at (5,-2.5) [] {$\TripIrr{r, s}$};
\draw [->] (top1) -- (left1);
\draw [->] (top1) -- (right1);
\draw [->] (left1) -- (bot1);
\draw [->] (right1) -- (bot1);

\node (top2) at (12, 2) [] {$\TripTyp{\alpha_{3-r, p-s}}$};
\node (bot2) at (12, -2) [] {$\TripTyp{\alpha_{r, s}}$};
\draw [->] (top2) -- (bot2);

\end{tikzpicture}
\captionbox{\label{fig:Loewy-Wp} Loewy diagram of $\TripAlg{p}$-modules
$\TripStag{r,s}$ in terms of 
	simple composition factors (left) and of the $\TripTyp{\alpha_{r, s}}$ in \eqref{ses:triplet} (right).
}{\rule{12cm}{0cm}}
\end{center}
\end{figure}

Fortunately in the instance of $\rep\TripAlg{p}$ we have a rigid vertex tensor category structure, as follows from $C_2$-cofiniteness  \cite{Carqueville:2005nu,AM2} and the works \cite{H,TW}. In particular:

\begin{theorem}
$\rep\TripAlg{p}$ is a $\CC$-linear abelian ribbon category.
\end{theorem}

The decomposition of the tensor products of 
simple and projective modules have been conjectured in \cite{FHST, GR2} and proved in \cite{TW}. Note that the computation of the
	decompositions
in \cite{TW} used the Nahm-Gaberdiel-Kausch algorithm 
	\cite{Nahm:1996zn,Gaberdiel:1996kx}
and it is work in progress \cite{KRi} that this agrees with the dimensions of spaces of intertwining operators
defined in the context of~\cite{HLZ}.

We note that the action of the nilpotent part of $L_0$ is always 
an intertwiner of VOA modules (since $e^{2 \pi \rmi m L_0}$ is an intertwiner for all $m \in \ZZ$).

\begin{proposition}[{\cite{AM4}}] \label{prop:triplogmod}
The modules $\TripStag{r, s}$ are logarithmic in the sense that the Virasoro-zero mode $L_0$ of $\TripAlg{p}$ does not act semi-simply
and its nilpotent part maps the top component of $\TripStag{r, s}$ onto the socle $\TripIrr{r, s}$.
\end{proposition}

\begin{proof}
For $r=1$ this has been proven in \cite[Cor.\,3.1]{AM4}. By \cite[Rem.\,3.1]{AM4} the case $r=2$ is related to $r=1$ by tensoring with the simple current $\TripIrr{2, 1}$. The non-semi-simple action of $L_0$ on an object of a vertex tensor category is equivalent to a non-semi-simple ribbon twist on that object. But by \cite[Lem.\,2.13]{CKL} together with the balancing axiom 
	\eqref{eq:balancing-def}, tensoring with
a simple current preserves this property, i.e.\ if the ribbon twist $\theta_X$ on the object $X$ is not semi-simple then the same is true for $\theta_{J \boxtimes X}$ with $J$ a simple current.
\end{proof}

Recall from Section~\ref{sec:sing-trip-def} that $\TripAlg{p}$ carries a $gl(1)$-action, and that $\SingAlg{p}$ is the 
invariant subalgebra of $\TripAlg{p}$, 
see \eqref{eq:sing-is-C*-inv-subspace}.
	For $\lambda \in L'$ 
we denote the one-dimensional representation 
$\rho_\lambda\colon gl(1) \rightarrow \text{End}(\CC)$, $u\mapsto u \lambda$ by $\CC_\lambda$. 
The full decomposition of $\TripAlg{p}$ and its modules into $gl(1)$-modules and modules over the subalgebra
 $\SingAlg{p}$ follows from the decomposition of  $\TripIrr{r, s}$ into bimodules over $sl(2)$ 
 and Virasoro VOA  as given in~\cite[Lem.\,3.5.2]{FGST3} and \cite[Sec.\,4]{ALM} together with  the decomposition of
$\SingAlg{p}$-modules into Virasoro 
modules~\cite[Eqs.\,(2.35)--(2.36)]{CRW}.
The resulting decompositions are:
\be \label{eq:tripassing}
\TripAlg{p} \cong \bigoplus_{k\in\ZZ} \CC_{\alpha_{2k+1,1}} \otimes \SingAtyp{2k+1, 1} \ .
\ee
Here we see the necessity of introducing 
$\repsimpleinfinity \SingAlg{p}$:  
$\TripAlg{p}$-modules viewed as $\SingAlg{p}$-modules are objects in this bigger category. 
All other modules that can be realised as submodules of representations of $V_{\sqrt{2p}\ZZ}$ have a similar $gl(1) \otimes \SingAlg{p}$-decomposition:
\be\label{eq:tripmoduledecomp}
\TripIrr{r, s} \cong \bigoplus_{k\in\ZZ} \CC_{\alpha_{2k+r,s}} \otimes \SingAtyp{2k+r, s} ~~, \qquad
\TripTyp{\alpha_{r, s}} \cong \bigoplus_{k\in\ZZ} \CC_{\alpha_{2k+r,s}} \otimes \SingTyp{\alpha_{2k+r, s}} \ .
\ee

Combining the conjectured fusion rules \eqref{CM-fusion} with the decomposition \eqref{eq:tripassing} suggests that $\TripAlg{p}$ is a simple current extension of infinite order of $\SingAlg{p}$. 
This is indeed true if Conjecture~\ref{conj:M-ribbon} 
is correct  \cite{CaM, CKLR}. 	The next theorem describes VOA extensions in terms of representations in general.

\begin{theorem}[{\cite{HKL,CKM}}] \label{thm:VOA-alg-in-cat-corr}
Let $V$ be a VOA and $\mathcal C$ a vertex tensor category of
 $V$-modules. Let $A \in \mathcal C$.
\begin{enumerate}
\item
There is a one-to-one correspondence between VOA structures on $A$ extending that of $V$ and structures of a commutative, associative algebra on $A$ with trivial action of the ribbon twist and injective unit.
\item 
${}_A\mathcal C^{\text{loc}}$ is $\CC$-linear braided monoidal equivalent to the vertex tensor category of representations of the VOA $A$ that lie in $\mathcal C$
as $V$-modules.
\end{enumerate}
\end{theorem}

Let $\mathcal C^{\oplus}$ be the completion of $\mathcal C$
with respect to countably infinite direct sums and subquotients. 
If $\mathcal C$ is a vertex tensor category, 
then $\mathcal C^{\oplus}$ inherits a vertex tensor category structure from $\mathcal C$. 
 Let $J$ be an invertible object of infinite order such that 
\be\label{eq:A-simple-current-infinite}
 A :=
 \bigoplus_{n\in\mathbb Z} J^{\otimes n}
\ee
 is a commutative, associative algebra object in $\mathcal C^{\oplus}$ with trivial action of the ribbon twist. Our example is
	$A\cong \TripAlg{p}$ 
and $J=\SingAtyp{3, 1}$.

Let $\mathcal{C}_J$ be the full subcategory of all objects in $\mathcal{C}$ that have trivial monodromy with~$J$ (or, equivalently, with $A$). 
	The so-called monodromy charge~\cite{Schellekens:1990xy} 
provides an easy criterion whether a simple object $X \in \mathcal{C}$ lies in $\mathcal{C}_J$, namely
\be\label{eq:XinCJ-crit}
	X \in \mathcal{C}_J
	\quad \Leftrightarrow \quad
	h_{J\fuse X} - h_X \in \ZZ \ , 
\ee
where $h$ denotes the conformal dimension, see e.g.\ \cite{tft3}
(recall that $A$ is assumed to have trivial ribbon twist, and hence so does $J$, that is, $h_J \in \mathbb{Z}$).
This criterion even applies more generally to indecomposable objects $X \in \mathcal{C}$ which are subquotients of a tensor product of several simple objects \cite{CKL}. By Lemma~\ref{lem:local-vs-induction} and Corollary~\ref{cor:fglocal-vs-induction}, induction by $A$ provides a braided monoidal functor 
\be\label{eq:induction-to-local}
	\mathcal{C}_J \longrightarrow \fgloc{A}{(\mathcal{C}^\oplus)} \ .
\ee

\begin{proposition}\label{prop:RepWp-local-fg-A-modules}
Suppose Conjecture~\ref{conj:M-ribbon} and the fusion rules \eqref{CM-fusion} are correct. 
Let  $\mathcal{C} = \repsimple \SingAlg{p}$ and
$A = \TripAlg{p}$. 
Then
\be\label{eq:RepWp-local-fg-A-modules}
\rep\TripAlg{p} \cong \fgloc{A}{(\mathcal{C}^\oplus)}
\ee
as ribbon categories.
\end{proposition}

We note that the result in~\eqref{eq:RepWp-local-fg-A-modules} is parallel to the result in Section~\ref{sec:ex-Cgr} where the category~\eqref{eq:C-gr-example-Mdef} for the lattice VOA was obtained as an extension of the Heisenberg VOA.

\begin{proof}
$A$ is of the form \eqref{eq:A-simple-current-infinite} for 
$J=\SingAtyp{3, 1}$. 
Since we assume Conjecture~\ref{conj:M-ribbon} to hold, 
$\repsimpleinfinity \SingAlg{p}$ is a vertex tensor category. We can then apply Theorem~\ref{thm:VOA-alg-in-cat-corr} to $\mathcal{C}^\oplus$ and our choice of $A$, which shows that ${}_A(\mathcal{C}^\oplus)^{\text{loc}}$ is a vertex tensor category.
However,  ${}_A(\mathcal{C}^\oplus)^{\text{loc}}$ contains infinite direct sums, and so is not itself equivalent to $\rep\TripAlg{p}$. Instead one has to restrict oneself to finitely-generated $A$-modules, as we will show now.
Let us describe the image of the functor~\eqref{eq:induction-to-local} in more detail.
Recall from Section~\ref{sec:Singlet} that the simple $\SingAlg{p}$-modules are given by the $\SingTyp{\lambda}$ with $\lambda$ not of the form $\alpha_{r,s}$, and by the $\mathcal{M}_{r,s}$. 
{}From the conformal weights in \eqref{eq:Heis_deformed-conf-weight}, from the conjectural fusion rules in \eqref{CM-fusion}, and from the criterion in \eqref{eq:XinCJ-crit} we learn that the simple $\SingAlg{p}$-modules in $\mathcal{C}_J$ 
	are precisely the $\mathcal{M}_{r,s}$.
Their induced modules are ($k'\in\ZZ, r=1, 2$ and $1\leq s\leq p$)
 \begin{equation}
 \begin{split}
 \text{Ind}(\SingAtyp{2k'+r, s})
 &= 
 \TripAlg{p} 
 \fuse \SingAtyp{2k'+r, s}  \cong \bigoplus_{k\in\ZZ} \SingAtyp{2k+1, 1}\fuse \SingAtyp{2k'+r, s} \\
 &\cong \bigoplus_{k\in\ZZ} \SingAtyp{2k+2k'+r, s} 
 \cong
   \TripIrr{r, s} \ .
\end{split}
 \end{equation}
Thus, every simple object of $\rep\TripAlg{p}$ is induced from a simple object in $\mathcal{C}_J$.
Moreover, since every object of $\rep\TripAlg{p}$ is 
a quotient of an iterated tensor product of completely reducible objects\footnote{
This is true if there is a simple projective object $Q$, such as $\TripIrr{1,p}$ and $\TripIrr{2,p}$ for $\TripAlg{p}$. Indeed, 
	since the tensor product is exact, tensoring anything with a projective object returns a projective object. In particular, $Q \otimes Q^*$ is projective, and
since $\mathrm{Hom}(Q \otimes Q^* , \one)$ is non-zero, $Q \otimes Q^*$ contains $P_{\one}$, 
the projective cover of $\one$, as a direct summand. Furthermore, $\mathrm{Hom}(P_{\one} \otimes U, U) \cong \mathrm{Hom}(P_{\one} , U \otimes U^*)$ is non-zero, so that $P_{\one} \otimes U$ contains $P_U$
as a direct summand. 
 Altogether, $P_U$ is a direct summand of $Q \otimes Q^*\otimes U$. 
 Since any object can be obtained as a quotient
 of a projective object, it can also be written as a quotient
 of the tensor product of $Q \otimes Q^*$ with  a completely reducible object.

As an aside we remark that every factorisable finite ribbon category over the complex numbers necessarily contains a simple projective object \cite[Thm.\,1.2]{GnR3}.
}
we have that every object of $\rep\TripAlg{p}$ is a quotient of an object induced from $\mathcal{C}_J$.
	By \eqref{eq:induction-to-local}, this shows that $\rep\TripAlg{p} \subset \fgloc{A}{(\mathcal{C}^\oplus)}$.
The conjectural fusion rules \eqref{CM-fusion} imply that every object in $\mathcal{C}$ has finite Jordan-H\"older length, and since induction is exact, we see that every object in $\fgloc{A}{(\mathcal{C}^\oplus)}$ has finite-dimensional generalised $L_0$-eigenspaces. Thus also $\rep\TripAlg{p} \supset \fgloc{A}{(\mathcal{C}^\oplus)}$, completing the proof of \eqref{eq:RepWp-local-fg-A-modules}
\end{proof}
 
\medskip

We conclude this section with the still-outstanding proof of Proposition~\ref{prop:singstag} about the $\SingAlg{p}$-modules 
$\SingStag{r, s}$. 

\begin{proof}[Proof of Proposition~\ref{prop:singstag}]
The triplet modules $\TripStag{r, s}$ have been constructed explicitely in \cite[Theorem 4.2]{Nagatomo:2009xp} by deforming the $\TripAlg{p}$ action on $\TripTyp{\alpha_{r, s}}\oplus \TripTyp{\alpha_{3-r, p-s}}$ following the ideas of
\cite{[FFHST]}, see also~\cite{AM5} and more recent results in~\cite{AM6, AM3}. 
From this construction it is clear that the modules $\TripStag{r, s}$ carry a semi-simple $gl(1)$-action that commutes with the $\SingAlg{p}$-subalgebra of $\TripAlg{p}$. 
In particular, 
the Heisenberg VOA submodule 
$\SingTyp{\alpha_{2k+r, s}} \oplus \SingTyp{\alpha_{2k+r-1, p-s}}$ 
of $\TripTyp{\alpha_{r, s}}\oplus \TripTyp{\alpha_{3-r, p-s}}$ 
becomes a $\SingAlg{p}$-module under the deformed action. 

	We denote this deformed $\SingAlg{p}$-module by $\SingStag{2k+r,s}$ and we will show that it satisfies the non-split exact sequence \eqref{eq:P-is-log-SES}.
We recall that $\TripTyp{\alpha_{r, s}}$ is a submodule of $\TripStag{r, s}$, see~\eqref{ses:triplettyp}. The $gl(1)$-action on $\TripStag{r, s}$
 is determined from \cite[Eq.\,(4.12)]{Nagatomo:2009xp} and is given by the 
weight decomposition \eqref{eq:tripmoduledecomp} 
on the submodule $\TripTyp{\alpha_{r, s}}$,
while the decomposition on the quotient  $\TripTyp{\alpha_{3-r, p-s}}$ is shifted
	by $(p-s)\alpha_-$.
The resulting decomposition of $\TripStag{r, s}$ as a $gl(1)\tensor\SingAlg{p}$ 
module is
\be
\TripStag{r, s} \cong \bigoplus_{k \in \ZZ}  \CC_{\alpha_{2k+r. s}} \otimes \SingStag{2k+r, s} \ .
\ee
with $\SingStag{2k+r, s}$ as defined above.
Indeed, the quotient module is
\be
\begin{split}
\TripStag{r, s}/\TripTyp{\alpha_{r, s}}\cong \TripTyp{\alpha_{3-r, p-s}}&\cong \bigoplus_{k\in\ZZ} \CC_{\alpha_{2k+r-1, p-s} +(p-s)\alpha_-} \otimes  \SingTyp{\alpha_{2k+r-1, p-s}} \\ &= \bigoplus_{k\in\ZZ} \CC_{\alpha_{2k+r, s}} \otimes  \SingTyp{\alpha_{2k+r-1, p-s}} \ .
\end{split}
\ee
In particular, $ \SingStag{2k+r, s}/\SingTyp{\alpha_{2k+r, s}} \cong  \SingTyp{\alpha_{2k+r-1, p-s}}$ as $\SingAlg{p}$-modules. 
The nilpotent part $L_0^{\text{nil}}$  of the Virasoro zero-mode $L_0$ of $\TripAlg{p}$ maps the top of $\TripStag{r, s}$ onto its socle. It follows that the (vector space-)direct summand 
$\TripTyp{\alpha_{3-r, p-s}}$ in $\TripStag{r, s}$
(that contains the top of $\SingStag{2k+r, s}$)
 is mapped non-trivially to  $\TripTyp{\alpha_{r, s}}$ under $L_0^{\text{nil}}$. The Virasoro algebra commutes with the $gl(1)$-action and so also 
	$\SingStag{2k+r,s}$ contains vectors that are
mapped non-trivially to  $\SingTyp{\alpha_{2k+r, s}}$ under $L_0^{\text{nil}}$.
 We have thus proven 
 that the $\SingAlg{p}$-modules $\SingStag{2k+r, s}$ satisfy the non-split short exact sequence~\eqref{eq:P-is-log-SES}.
\end{proof}

\subsection{Correspondence between VOAs and quasi-Hopf algebras}\label{sec:corr-VOA-qHopf}

Recall from Section~\ref{sec:unrolled+weight} the definition of the categories 
	$\catUfd_p$ and $\catU_p^\oplus$ 
of weight-modules over the unrolled quantum group,
and from Section~\ref{sec:Singlet} the definition of $\repsimple \SingAlg{p}$ and its completion $\repsimpleinfinity \SingAlg{p}$. 
We have the following conjecture \cite{CM, CGP, CMR}.${}^{\ref{thefootnote}}$

\begin{conjecture}\label{conj:M-C}
The category $\repsimple \SingAlg{p}$ is equivalent
  to $\catUfd_p$  as ribbon categories. 
 \end{conjecture}

Recall the notation for simple objects of $\catUfd_p$ and $\SingAlg{p}$ from Sections~\ref{sec:unrolled+weight} and~\ref{sec:Singlet}.
The conjectural equivalence $\catUfd_p \to \repsimple \SingAlg{p}$ acts on simple objects 
	and their projective covers
as
\be\label{eq:sing-Uq-func-obj}
\UqgTyp{\alpha} \mapsto \SingTyp{\frac{\alpha+p-1}{\sqrt{2p}}}
~~, \quad
\repS_{s,k} \mapsto \SingAtyp{1-k,s}
~~, \quad 
\UqgStag{s, k} \mapsto \SingStag{1-k, s} \ .
\ee
 This conjecture is motivated by the following: 
 \begin{enumerate}
 \item
  The Loewy diagrams of indecomposable objects agree.
  \item The ribbon twist on simple modules of $\catUfd_p$  agrees with the eigenvalue of $e^{2\pi i L_0}$ on the corresponding simple $\SingAlg{p}$-modules. 
 \item 
Tensor products in $\catUfd_p$ have been computed in \cite[Sec.\,8]{CGP}. 
In order to compare, note that our $\repS_{s,k}$ corresponds to 
$S_{s-1} \otimes \mathbb C^H_{kp}$ in \cite{CGP}, our $\UqgStag{s, k}$ corresponds to their $P_{s-1} \otimes \mathbb C^H_{kp}$, while the notation  $\UqgTyp{\alpha}$ for typical modules is the same here and in \cite{CGP}.
The conjectural Grothendieck fusion rules in \eqref{CM-fusion} match the tensor product in $\catUfd_p$ via the assignment \eqref{eq:sing-Uq-func-obj}, see \cite[Sec.\,5]{CMR}.  
\item 
The asymptotic behaviour of characters of $\SingAlg{p}$-modules allows one to define a quantity called the asymptotic dimension of a $\SingAlg{p}$-module \cite{CM, CMW, BFM}. 
These are expected to be related to open Hopf link invariants
(in analogy with rational CFTs),
 and indeed 
 \cite[Thm.\,1]{CMR} gives such a correspondence between  asymptotic dimensions of $\SingAlg{p}$-modules and open Hopf link invariants of objects in~$\catUfd_p$.
  \end{enumerate}
 
The next corollary gives the precise relationship between $\TripAlg{p}$ and $\Q$.

\begin{corollary}\label{cor:main}
Correctness of Conjectures~\ref{conj:M-ribbon} and~\ref{conj:M-C} implies that the categories  $\catTrip{p}$ and $\catQ_p=\repfd\,\Q$ are equivalent as ribbon categories.
\end{corollary}

\begin{proof}
	We abbreviate $\mathcal{W} = \TripAlg{p}$ and $\mathcal{M} = \SingAlg{p}$.
Let $\mathcal{E} \colon \repsimple \mathcal{M}\to\catUfd_p$ be the ribbon equivalence from Conjecture~\ref{conj:M-C}, which acts on objects as in~\eqref{eq:sing-Uq-func-obj}. We also write $\mathcal{E}$ for the induced equivalence $\repsimpleinfinity \mathcal{M}\to\catU^\oplus_p$.
Note that by point~3 above, correctness of Conjecture~\ref{conj:M-C} also implies that the fusion rules \eqref{CM-fusion} are valid.

Let $A := \mathcal{E}(\mathcal{W})$ be the image of the triplet algebra under $\mathcal{E}$. As $\mathcal{W}$ is a commutative algebra, so is $A$. We will now show that $A$ allows for an invariant non-degenerate pairing in the sense of \eqref{eq:cat-pairing-inv} and \eqref{eq:cat-pairing-nondeg}.

The triplet VOA $\mathcal{W}$ is self-contragredient, simple and of CFT-type, 
i.e.\ it allows for a non-degenerate symmetric invariant pairing in the VOA sense, see \cite[Cor.\,3.2]{Li:1994}. 
Write $\mathcal{W}'$ for the module contragredient to $\mathcal{W}$. 
A $\mathcal{W}$-invariant pairing on $\mathcal{W}$ is the same as an intertwiner 
$\varphi \in \mathrm{Hom}_{\mathcal{W}}(\mathcal{W},\mathcal{W}')$. Note that in particular $\varphi \in \mathrm{Hom}_{\mathcal{M}}(\mathcal{W},\mathcal{W}')$. 
We have the following isomorphisms between the above Hom-space and spaces of intertwining operators for $\mathcal{M}$-modules:
\be\label{eq:VOA-pairing-cat-pairing}
	\mathrm{Hom}_{\mathcal{M}}(\mathcal{W},\mathcal{W}')
	~\cong~
	I { \mathcal{W}' \choose \mathcal{W} ~ \mathcal{M} }
	~\overset{(*)}{\cong}~
	I { \mathcal{M}' \choose \mathcal{W} ~ \mathcal{W}'' }
	~\overset{(**)}{\cong}~
	I { \mathcal{M} \choose \mathcal{W} ~ \mathcal{W} } \ ,
\ee
see \cite[Rem.\,2.9]{Li:1994}. The isomorphism $(*)$ is given explicitly in \cite[Eq.\,(5.5.4)]{FHL-book}, and $(**)$ uses that $\mathcal{M}$ is self-contragedient and that the double-contragedient module is isomorphic to the original module (recall that $\mathcal{W}$ has finite-dimensional graded components).

Let $\iota$ be the image of $\varphi$ under these isomorphisms. Then $\mathcal{E}(\iota)$ defines a pairing on $A$. Using the isomorphisms in \eqref{eq:VOA-pairing-cat-pairing}, invariance of the pairing $\mathcal{E}(\iota)$ on $A$ follows from the fact that $\varphi$ is a 
$\mathcal{W}$-intertwiner, not just an $\mathcal{M}$-intertwiner.
Since the pairing on $\mathcal{W}$ is symmetric, non-degeneracy is equivalent to $\mathrm{ker}(\varphi)=\{0\}$.
{}From the explicit form of the isomorphisms in \eqref{eq:VOA-pairing-cat-pairing} one checks that $w \in \ker(\varphi)$ if and only if $I(w,z)v=0$ for all 
$v \in \mathcal{W}$. Since $\ker(\varphi)$ is a submodule, we see that non-degeneracy of the pairing in the VOA-sense translates into the statement that the pairing $\mathcal{E}(\iota)$ on $A$ is non-degenerate in the sense of
 \eqref{eq:cat-pairing-nondeg}.

Comparing \eqref{eq:algC} and \eqref{eq:tripassing} we conclude that $A \cong \Lambda$ as objects,
recall the correspondence in~\eqref{eq:sing-Uq-func-obj}.
But $A$ satisfies properties 1--3 from Proposition~\ref{prop:comm-alg-unique}, and hence carries a unique algebra structure (up to isomorphism) with these properties (recall the embedding $\mathcal{J}$ from Lemma~\ref{lem:H-in-Cp-embed}). The same holds for $\Lambda$, and hence we must have $A \cong \Lambda$ as algebras.

The algebra isomorphism $\mathcal{E}(\mathcal{W}) \cong \Lambda$ implies that $\mathcal{E}$ induces a ribbon equivalence $\mathcal{E}'$ of categories of local modules, so that altogether
\be
	\rep\mathcal{W}
	\overset{\text{Prop.\,\ref{prop:RepWp-local-fg-A-modules}}}{\cong}
	 {}_{\mathcal{W}(p)}\big(\repsimpleinfinity \mathcal{M}\big)^{\text{fg-loc}}
	~~\overset{\mathcal{E}'}\cong~~
	\fgloc{\algC}{{(\catU_{\mathit{p}}^\oplus)}}
	\overset{\text{Thm.\,\ref{thm:equiv-1}}}\cong
	\repfd\,\Q \ .
\ee
\end{proof}

The relationship between the VOAs $\SingAlg{p}\subset \TripAlg{p}$ and the (quasi-) Hopf algebras $\UQG{q}$ and $\Q$ is summarised in the following diagram, which is the precise version of \eqref{eq:intro-cat-diag} from the introduction:
\be\label{diag-QG-VOA}
\raisebox{\height}{
\xymatrix{
\catUfd_p
\ar[d]_[right]{\sim}^{~\text{Conj.\,\ref{conj:M-ribbon}, \ref{conj:M-C}}} 
&\ar@{~>}[rrr]^{\text{\begin{minipage}{5.5em}pass to local\\[-.5em] $\algC$-modules\end{minipage}}}
&&&&
\fgloc{\algC}{{(\catU_{\mathit{p}}^\oplus)}}
\ar[d]_[right]{\sim}^{~\text{Cor.\,\ref{cor:main}}} 
& \catQ_p
\ar[l]_{\quad\sim}^{\text{\qquad Thm.\,\ref{thm:equiv-1}}} 
\\
\repsimple \SingAlg{p}
&\ar@{~>}[rrr]^{\text{\begin{minipage}{7.5em}pass to represen-
\\[-.5em]
tations of $\TripAlg{p}$\end{minipage}}}
&&&&
\rep \TripAlg{p}
}}
\ee
Note that the lower horizontal arrow also requires Conjecture~\ref{conj:M-ribbon}.

\section{Proof of Theorem~\ref{thm:equiv-1}}\label{sec:proof}

In this section we prove our
 main theorem -- Theorem~\ref{thm:equiv-1} --
about the equivalence of the two ribbon categories $\catQ_p$ and $\Cmodl$.
Our  proof generalises the construction in Section~\ref{sec:ex-qHopf} and essentially follows steps \ref{plan-step1}--\ref{plan-step5} there. 

\subsection{Coproduct and antipode}

We will drop the subscript $t$ from the structural elements of $\Q$, e.g.\ we will write $\Delta$ instead of $\Delta_t$.
It is straightforward to check that $\Delta$ is an algebra homomorphism and $S$ is an algebra anti-homomorphism.
It remains to show the identities
\begin{equation}\label{eq:Salpha-1}
\sum_{(a)}S(a')\Salpha a'' =
\eps(a)
\Salpha\gc\qquad
\sum_{(a)}a'\Sbeta S(a'') =
\eps(a)
\Sbeta \ ,
\end{equation}
for all $a\in \Q$ and 
\begin{equation}\label{eq:Salpha-2}
\sum_{(\Phi)}S(\Phi_1)\Salpha \Phi_2\Sbeta  S(\Phi_3) = \one\gc\qquad
\sum_{(\Phi^{-1})}(\Phi^{-1})_1\Sbeta S((\Phi^{-1})_2)\Salpha  (\Phi^{-1})_3 = \one \gp
\end{equation}
The last two identities are straightforward and we skip their proof.

We first show~\eqref{eq:Salpha-1} on the generators. 
We define the linear map 
\be
P:=\mu\circ(S\tensor\id)\colon\; \Q\tensor\Q\to\Q
\ee
 where $\mu$ denotes the multiplication in $\Q$.
It is then straightforward to see that 
\be\label{eq:PEFK}
P\big(\Delta(F)\big)=0 = P\big(\Delta(E)\big) \ .
\ee
For the PBW basis elements
$F^{n} E^{m} K^j$ and defining $f:=F^{n-1} E^{m} K^j$, we get for $n>0$, \be
P\big(\Delta(F^{n}E^{m} K^j)\big) 
= P\big(\Delta(F)\Delta(f)\big)=\sum_{(F),(f)} S(F'f')F''f'' 
 =\sum_{(f)} S(f')P\big(\Delta(F)\big)f'' 
\stackrel{\eqref{eq:PEFK}}{=}
  0 
\ee
and similarly we calculate for $n=0$ and $m>0$. The only remaining case is when $n=m=0$,
i.e.\ the case $K^j$. But this is obvious as $K$ is group-like.
This proves the first identity in~\eqref{eq:Salpha-1}.
The second identity in~\eqref{eq:Salpha-1} can be shown analogously.

\subsection{The equivalence functors $\funF_\RS$ and $\funG_\RS$}

Recall that $\catU_p^\oplus$ is the category of at most countably-infinite dimensional weight modules of $\UQG{q}$, 
and that in~\eqref{eq:algC} we introduced 
	$\algC\in\catU_p^\oplus$, 
a commutative algebra
with multiplication $\mu_{\algC}$ given in~\eqref{eq:muC}.  Recall also that $\catQ_p$ denotes the category of finite-dimensional  
	$\Q$-modules. 
We  construct here a functor from $\catQ_p$ to the category $\Cmodl$ of finitely-generated local $\algC$-modules  and prove that it is an equivalence.
	This amounts to step~\ref{plan-step1} in Section~\ref{sec:ex-qHopf}.
 
	To define the functor we need to
``lift'' the $K$-grading on objects from $\catQ_p$ to an $H$-grading.
 The action of $H$ and $K$ are related as $K=q^H$ or $H = \frac{p}{i\pi}\log K$ and thus one has to choose a branch of the logarithm to fix  the $H$-action. Recall that $K$ generates $\ZZ_{2p}$ (and hence has $2p$ distinct eigenvalues) and fixing a branch of the logarithm amounts to choosing a set~$\RS$ of representatives of $\ZZ_{2p}$ in $\ZZ$.
This is therefore as in step~\ref{plan-step1} in Section~\ref{sec:ex-qHopf}:
for a subset $\RS\subset\ZZ$ satisfying~\eqref{eq:S} we provide a functor
\be\label{eq:funF}
\funF_\RS\colon\; \catQ_p \longrightarrow \Cmodl \ ,
\ee
which acts on objects and morphisms as
\be\label{eq:funF-2}
 M \,\longmapsto\, \algC\otimes_{\CC} M 
 \qquad , \qquad
 (M\xrightarrow{f}N) \,\longmapsto\, \id_\algC\otimes_{\CC} f \ .
\ee
Here, $\algC\otimes_{\CC} M$ is the tensor product of vector spaces. 

Recall from \eqref{eq:abbrev-x-xS} the notation $\re{x}_\RS$ for the representative  in $\RS$ of an integer $x$, and the abbreviation $\re{x}:= \re{x}_\RS$. 
	Specialising \eqref{eq:kappa}, it will be convenient to introduce the notation
\be\label{eq:sig}
	\kappa_{\pm}(a) = \kappa(a , \pm2) = \frac{a\pm2-\re{a\pm2}}{2p}\ \in \ \ZZ\ .
\ee
For $a\in \RS$,
recall the idempotents $\e{a}:=\e{\pi(a)}$ in~\eqref{eq:idemp-prim} with our convention after~\eqref{eq:M-decomp} and define the $\UQG{q}$ action on $\algC\otimes_{\CC} M$
for arbitrary $m\in M$ and $k\in\ZZ$ as follows (recall the basis in~\eqref{eq:algC}):
\begin{align}
K (1_k\otimes \e{a}m) &:= q^{a} \cdot 1_k\otimes \e{a}m\ ,\label{eq:funF-UqH-1}\\
H (1_k\otimes \e{a}m) &:= (2pk + a) \cdot 1_k\otimes \e{a}m\ ,\label{eq:funF-UqH-2}\\
E (1_k\otimes \e{a}m) &:= 1_{k+\kappa_+(a)
}\otimes \e{\re{a+2}}E \,m\ ,\label{eq:funF-UqH-3}\\
F (1_k\otimes \e{a}m) &:= 1_{k+\kappa_-(a)}\otimes \e{\re{a-2}}F \,m\ ,\label{eq:funF-UqH-4}
\end{align}
compare for the Cartan part with~\eqref{eq:C-gr-ex_K-action_to_H-action}.
	We stress that in~\eqref{eq:funF-UqH-1}--\eqref{eq:funF-UqH-4} the integer number $a$ must be taken from $\RS$. 
In particular, $H (1_k\otimes \e{a+2}m) = (2pk + \re{a+2}) \cdot 1_k\otimes \e{a+2}m$.
Having this in mind, it is then straightforward to check that the relations of the unrolled quantum group $\UQG{q}$ are satisfied. 
We show only the relation 
$E^p=F^p=0$. For this, we recursively compute the powers
\be\label{eq:funF-UqH-p1}
\begin{split}
E^n (1_k\otimes \e{a}m) &= 1_{k+ \frac{a+2n-\re{a+2n}}{2p}}\otimes \e{\re{a+2n}}E^n \,m\ ,\\
F^n (1_k\otimes \e{a}m) &= 1_{k+ \frac{a-2n-\re{a-2n}}{2p}}\otimes \e{\re{a-2n}}F^n \,m\ . % \label{eq:funF-UqH-p2}
\end{split}
\ee
By \eqref{EpFp-rel} the RHS is zero for $n=p$.

We  define the $\algC$-action on $\funF_\RS(M)=\algC\otimes_{\CC} M$ using the multiplication on the left as in~\eqref{eq:rho-LambdaM}.

\begin{lemma}
The map $\rho:=\rho_{\algC\otimes_{\CC} M}$ in~\eqref{eq:rho-LambdaM} is a morphism in 
	$\catU_p^\oplus$. 
Moreover, the pair $(\algC\otimes_{\CC} M, \rho)$ is an object in $\Cmodl$.
\end{lemma}
\begin{proof}
Using the coproduct $\Delta$ for $\UQG{q}$ and the action~\eqref{eq:funF-UqH-1}-\eqref{eq:funF-UqH-4} it is straightforward to check that the map $\rho_{\algC\otimes_{\CC} M}$ from~\eqref{eq:rho-LambdaM} intertwines the  action of $\UQG{q}$, i.e.\ that it is a morphism in $\catU_p^\oplus$.

We next show the locality. Using the braiding $\brC$ in $\catU_p^\oplus$ defined through the $R$-matrix in~\eqref{eq:R-H}, we see that the composition $\rho\circ \brC\circ \brC:=\rho_{\algC\otimes_{\CC} M}\circ \brC_{\algC\otimes_{\CC} M,\algC}\circ \brC_{\algC,\algC\otimes_{\CC} M}$ acts on the basis elements as
\begin{multline}
\rho\circ \brC\circ \brC
(1_t\otimes 1_k\otimes \e{a}m) = \rho_{\algC\otimes_{\CC} M}\circ q^{H\otimes H}(1_t\otimes 1_k\otimes \e{a}m) 
\\= \rho_{\algC\otimes_{\CC} M} q^{2pt(2pk+a)} (1_t\otimes 1_k\otimes \e{a}m)
=\rho_{\algC\otimes_{\CC} M} (1_t\otimes 1_k\otimes \e{a}m) \ .
\end{multline}
Thus the locality condition $\rho\circ \brC\circ \brC = \rho$ is satisfied.

It remains to check that $(\algC\otimes_{\CC} M, \rho)$ is finitely generated.
Since $M$ is finite-dimensional by definition of $\catQ_p$, also the $H$-eigenspaces as defined via \eqref{eq:funF-UqH-2} are finite-dimensional. The claim now follows from Proposition~\ref{prop:UHq-local-Lam-modules}.
\end{proof}

This finishes our definition of $\funF_\RS$ from~\eqref{eq:funF} and the proof that it is a  functor from $\catQ_p$ to $\Cmodl$.
We summarise:

\begin{proposition}\label{prop:funF}
For a set $\RS$ satisfying~\eqref{eq:S}, the formulas~\eqref{eq:funF-2} 
together with~\eqref{eq:funF-UqH-1}-\eqref{eq:funF-UqH-4} and~\eqref{eq:rho-LambdaM},
define $\funF_\RS$ as a $\CC$-linear functor $\catQ_p \to \Cmodl$.
\end{proposition}

For two different choices $\RS$ and $\RS'$ of representative sets, one can show analogously to the proof of Lemma~\ref{lem:etaSS} that $\funF_\RS$ and $\funF_{\RS'}$ are naturally isomorphic. The natural isomorphism is given by the same expression as in \eqref{eq:natis-FS-FS'}.
We can thus work with an arbitrary choice of $\RS$ (satisfying~\eqref{eq:S})
as the result does not depend on the choice, up to a natural isomorphism.

\medskip

	The main goal in this subsection is to show that $\funF_\RS$ is an equivalence. As preparation we need a number of lemmas.
For 
	$(V,\rho_V)\in \Cmod^\oplus$, 
we define the action $\rho(1_t)$ of the basis elements of $\algC$
as in~\eqref{eq:rho1t}.

\begin{lemma}\label{lem:Z-V}
Let 
	$(V,\rho_V)\in \Cmod^\oplus$.
 We then have the equality
(for any $v\in V$)
\begin{align}
\rho(1_t)\circ E (v) = E\circ \rho(1_t) (v)\ , \label{eq:1t-E}\\
\rho(1_t)\circ F (v) = F\circ \rho(1_t) (v)\ . \label{eq:1t-F}
\end{align}
\end{lemma}
\begin{proof}
	We will show only~\eqref{eq:1t-E}, the proof of~\eqref{eq:1t-F} is analogous.
The map $\rho_V$ is a morphism in 
	$\catU_p^\oplus$ 
and thus intertwines the $\UQG{q}$ action on $V$.
 In particular, this means  the equality $\rho_V \circ \Delta(E) = E \circ \rho_V$ and $\Delta(E) = \one\otimes E + E\otimes K$ acts on $\algC\otimes V$. As $E$ acts trivially on $\algC$, so $E\cdot 1_t = 0$, we have
\be
\rho_V(1_t\otimes E\cdot V) = E\circ \rho_V (1_t\otimes V)\ , \qquad t\in \ZZ\ , 
\ee
and this gives~\eqref{eq:1t-E}.
\end{proof}

Combining Propositions~\ref{prop:Lambda-loc-mod}  and~\ref{prop:UHq-local-Lam-modules}  we have the following.

\begin{lemma}\label{lem:V-loc-decomp}
Any local $\algC$-module $V$ has the decomposition into 
 $H$-eigenspaces $V_x$ of integer eigenvalues~$x$:
\be\label{eq:V-decomp}
V = \bigoplus_{k\in\ZZ}\bigoplus_{a\in \RS} V_{2pk+a}
\ee
and the basis elements $1_t\in\algC$ act by isomorphisms of vector spaces:
\be\label{eq:iso-Vx}
\rho(1_t): \quad V_{2pk+a} \xrightarrow{\; \sim \;} V_{2p(k+t)+a}\ .
\ee
\end{lemma}

\begin{lemma}\label{lem:V-f}
 A morphism $f\colon V\to W$  in $\Cmodl$ 
	has the 
decomposition $f = \oplus_{n\in\ZZ} f_n$ and the property
\be\label{eq:1_t-f}
f_{n+2pk}\circ\rho_V(1_k)(\e{n} v) = \rho_W(1_k)\circ f_n(\e{n}v)\ ,
\ee
where $\rho_V\colon \algC\otimes V\to V$ and $\rho_W\colon \algC\otimes W\to W$ are the corresponding $\algC$ actions on $V,W\in\Cmodl$.
\end{lemma}
\begin{proof}
As each morphism $f$ in $\Cmodl$ is the intertwiner of the $H$ action and by Lemma~\ref{lem:V-loc-decomp} the eigenvalues of $H$ are integer, we obviously have the decomposition  $f = \oplus_{n\in\ZZ} f_n$ where $f_n$ is the restriction of $f$ to $V_n\subset V$ and $f_n\colon V_n\to W_n$. We 
then use the condition on $f$ to be an intertwiner of the $\algC$-actions  $\rho_V$ and $\rho_W$:
\be
f\circ \rho_V = \rho_W \circ (\id_{\algC}\otimes f)\ .
\ee
The latter can be written on the $n$th grade components as
\be
f_{n+2pk} \circ \rho_V(1_k\tensor \e{n}v) = \rho_W \bigl(1_k\tensor f_n(\e{n}v)\bigr)
\ee
which gives~\eqref{eq:1_t-f} using the definition in~\eqref{eq:rho1t}.
\end{proof}

We now arrive at the main statement of this subsection.

\begin{proposition}\label{thm:funF-equiv}
The functor $\funF_\RS\colon \catQ_p \to \Cmodl$ defined in Proposition~\ref{prop:funF} is an equivalence of $\CC$-linear categories.
\end{proposition}

\begin{proof}
We prove the equivalence by introducing another functor
\be
\funG_\RS: \quad \Cmodl \to \catQ_p
\ee
that is an inverse to $\funF_\RS$. On objects, 
we define it similarly to~\eqref{eq:funG-Cartan}
 (recall that by Lemma~\ref{lem:V-loc-decomp} we have~\eqref{eq:V-decomp}):
\be\label{eq:funG}
\funG_\RS: \quad V=\bigoplus_{k\in\ZZ}\bigoplus_{a\in \RS} V_{2pk+a} \; \mapsto \; V_{(0)}:=\bigoplus_{a\in \RS} V_{a} \ .
\ee
The  action of $\Q$-generators
on $V_{(0)}$ is given by $K=q^H$ and the following maps (recall~\eqref{eq:sig})
\begin{align}
\dE &\,:=\, \big[\, V_a \xrightarrow{\,E\,} V_{a+2} \xrightarrow{\;\rho(1_{-\kappa_{+}(a)})\;} V_{\re{a+2}} \,\big]\ ,
\nonumber \\
\dF &\,:=\, \big[\, V_a \xrightarrow{\,F\,} V_{a-2} \xrightarrow{\;\rho(1_{-\kappa_{-}(a)})\;} V_{\re{a-2}} \,\big] \ .
\label{eq:EF-in-D-def}
\end{align}
We check below that these maps together with $K$ action indeed satisfy the defining relations of $\Q$. Let us note that
 we use the notation $\dE$ and $\dF$ for the representation of the generators $E$ and $F$ of	$\Q$
(i.e. on objects in $\catQ_p$) in order to distinguish from the action of 
	$E$ and $F$ of $\UQG{q}$ on objects from $\catU_p^\oplus$.
We have thus well-defined maps on $\funG_\RS(V)=V_{(0)}$ and the relations $K\circ\dE = q^2 \dE \circ K$ and $K\circ\dF = q^{-2} \dF \circ K$ are straightforward to check. For the commutator $[\dE,\dF]$, we  compute first the composition on each $V_a$:
	for $a\in \RS$ we have
\be
\dE\circ \dF (\e{a}v) = \dE\circ \rho(1_{-\kappa_{-}(a)}) \circ F(\e{a}v) 
\overset{(*)}= 
\rho(1_{\kappa_{-}(a)}) \circ E\circ \rho(1_{-\kappa_{-}(a)}) \circ F(\e{a}v) \ ,
\ee
where	in $(*)$ 
we used that $-2p\kappa_+(\re{a-2}) = \re{\re{a-2}+2} - (\re{a-2}+2) = 2p \kappa_-(a)$.
	Applying
then~\eqref{eq:1t-E} and~\eqref{eq:1t-F} from Lemma~\ref{lem:Z-V} we obtain
\be
\dE\circ \dF (\e{a}v) = E F(\e{a}v) 
\ee
and therefore $[\dE,\dF] = \frac{K-K^{-1}}{q-q^{-1}}$. We also compute the powers 
\begin{align}
\dE^n (\e{a}v) &= \rho\bigl(1_{\frac{\re{a+2n}-(a+2n)}{2p}}\bigr)\circ E^n(\e{a}v)\ ,\\
\dF^n (\e{a}v) &= \rho\bigl(1_{\frac{\re{a-2n}-(a-2n)}{2p}}\bigr)\circ F^n(\e{a}v)\ 
\end{align}
and so $\dE^p=\dF^p=0$ and therefore we have shown that the defining relations of	$\Q$
hold indeed.

On morphisms $f\colon V\to W$, we define 
\be
\funG_\RS\colon  f \mapsto  f_{(0)} :=  \bigoplus_{a\in \RS} f_a
\ee
 with $f_{(0)}$ the restriction of $f$ to the ``fundamental'' component $V_{(0)}\subset V$. As morphisms in $\Cmodl$  respect the grading by~$H$, we have $f_{(0)}\colon  V_{(0)}\to W_{(0)}$. We have only got to show that $f_{(0)}$ is an $\Q$
 intertwiner, or that it commutes with $\dE$ and $\dF$. Using
Lemma~\ref{lem:V-f}, 
we have a sequence of the equalities
\begin{multline}\label{eq:f0-E} 
f_{\re{a+2}}\circ \dE (\e{a} v) 
\overset{\eqref{eq:EF-in-D-def}}{=} 
f_{(a+2) - 2p\kappa_{+}(a) }\circ \rho_V(1_{-\kappa_{+}(a)}) \circ E (\e{a}v)\\
\overset{\eqref{eq:1_t-f}}{=} 
\rho_W(1_{-\kappa_{+}(a)}) \circ f_{a+2}\circ E(\e{a}v) 
\overset{(*)}= 
 \rho_W(1_{-\kappa_{+}(a)}) \circ E\circ f_{a}(\e{a}v) 
\overset{\eqref{eq:EF-in-D-def}}{=} 
 \dE \circ f_a (\e{a}v)\ ,
\end{multline}
	where in $(*)$ we used the intertwining property of $f$ for $E\in\UQG{q}$ action, which is $f_{a+2}\circ E = E\circ f_{a}$. 
As~\eqref{eq:f0-E} is true for each $a\in \RS$, it is also true for $f_{(0)} = \oplus_{a\in \RS} f_a$ and therefore we have in total $f_{(0)}\circ \dE = \dE \circ f_{(0)}$. A similar computation shows the equality $f_{(0)}\circ \dF = \dF \circ f_{(0)}$, and therefore $f_{(0)}$ is an $\Q$
 intertwiner indeed.
  This finishes our definition of the functor $\funG_\RS$.

We next show that the composition $\funG_\RS\circ\funF_\RS$ is naturally isomorphic to the identity functor $\id_{\catQ_p}$.
It goes along the lines of Section~\ref{sec:ex-qHopf}.
On objects, the composition is
\be
\funG_\RS\circ\funF_\RS\colon  \quad M \mapsto \bigl(\algC\tensor_{\CC} M\bigr)_{(0)} 
\ee
where the image is obviously  $\CC 1_0\tensor M$, as a  vector space.
We use then the definition of $\dE$ and $\dF$, and
the action~\eqref{eq:funF-UqH-3}, \eqref{eq:funF-UqH-4}
to compute
\begin{align}
\dE (1_0\tensor \e{a}m) &= \rho_{\algC\tensor_{\CC} M}(1_{-\kappa_{+}(a)}) \circ E  (1_0\tensor \e{a}m)  
\nonumber \\
&= \rho_{\algC\tensor_{\CC} M}(1_{-\kappa_{+}(a)})   (1_{\kappa_{+}(a)}\tensor \dE \e{a}m) = 
1_0\tensor  \dE \e{a}m
\end{align}
and similarly
\be
\dF (1_0\tensor \e{a}m)  = 1_0\tensor  \dF \e{a}m\ .
\ee
We have thus the natural transformation
\begin{align}
\funG_\RS\circ\funF_\RS \xrightarrow{\;\cdot\;} \id_{\catQ_p}: \quad &\funG_\RS\circ\funF_\RS(M)\to M
\nonumber \\
&1_0\tensor m \mapsto m
\label{eq:GF-nat}
\end{align}
for any $m\in M$ (the naturality of this map is obvious). It is also clear that~\eqref{eq:GF-nat} is  an  isomorphism.

For the other direction,  we show that the composition $\funF_\RS\circ\funG_\RS$ is naturally isomorphic to the identity functor $\id_{\Cmodl}$. On objects, the composition is
(compare with~\eqref{eq:FG-Cartan})
\be
\funF_\RS\circ\funG_\RS: \quad V \mapsto \algC\tensor_{\CC} V_{(0)} \ .
\ee
The natural transformation in this case is  
\begin{align}
\funF_\RS\circ\funG_\RS \xrightarrow{\;\cdot\;} \id_{\Cmodl}: \quad &\funF_\RS\circ\funG_\RS(V)\to V
\nonumber\\
&1_k \tensor v
  \mapsto \rho_V(1_k)(v)\ , \qquad k\in\ZZ, \; v\in V_{(0)}\ .
\label{eq:FG-nat}
\end{align}
It is a straightforward check that this map is a morphism in $\Cmodl$, 
i.e.\ an $\UQG{q}$ intertwiner that also intertwines the $\algC$ action. The map in~\eqref{eq:FG-nat}  is  an isomorphism because of Lemma~\ref{lem:V-loc-decomp}. 
The naturality of~\eqref{eq:FG-nat} is also easy to see. We have thus shown that both $\funF_\RS$ and $\funG_\RS$ are equivalences.
This finishes our proof of the theorem.
\end{proof}

\subsection{Ribbon equivalence between $\catQ_p$ and $\Cmodl$}
The category $\Cmodl$ is a braided monoidal category with a ribbon twist. And we have shown in 
	Proposition~\ref{thm:funF-equiv} 
that	it
is equivalent,	as a $\CC$-linear category, 
to the representation category of the restricted quantum group $\UresSL2$.  The idea of this  section is to extend this equivalence up to braided monoidal categories level, of course, with all the necessary modifications of the standard monoidal structure on $\rep \UresSL2$, as this category with the initial monoidal structure can not be braided.

	We will denote the structure  maps \eqref{qHopf-cat-data} for
	$\repQ = \rep \Q$ 
	by $\assoc^{\mathcal{D}}$, $c^{\mathcal{D}}$ and $\theta^{\mathcal{D}}$, and
we will show below that these linear maps are indeed morphisms in $\catQ_p$, that they are natural and  that $\assoc^{\mathcal{D}}$ gives an associator, $c^{\mathcal{D}}$ a braiding and $\theta^{\mathcal{D}}$ a twist in $\repQ$.  The idea of the proof is to endow the equivalence functor $\funF\colon \catQ_p \to \Cmodl$ (we  use the abbreviation $\funF:= \funF_\RS$ and omit $\RS$ for brevity) with a certain	 multiplicative structure $\funF_{M,N}\colon\funF(M) \tensor_{\algC} \funF(N) \xrightarrow{ \sim } \funF(M\tensor N)$ and to show that it solves the coherence conditions. As explained in Remark~\ref{rem:solve-pent-hex-autom} this implies that $\assoc^{\mathcal{D}}$ solves the pentagon, etc.,	and thus proves both parts of Theorem~\ref{thm:equiv-1} together. 

The twist-equivalence of $\Q$ for different values of $t$ in part 1
of Theorem~\ref{thm:equiv-1} 
 follows from part 2 and the fact that the underlying equivalence $\funF$ does not change (only its multiplicative structure does). This implies 
 that the monoidal structures on  $\repQ$ for different values of $t$ are related via the identity functor equipped with a suitable monoidal structure, which in turn is the same as twist-equivalence (cf.\ the proof of \cite[Prop.\,16.1.6]{ChPr}).
 
\medskip

As in step~\ref{plan-step2} in  Section~\ref{sec:ex-qHopf}, we begin with introducing  a family of $\CC$-linear maps
\begin{align}\label{eq:isoF-MN}
\funF_{M,N}: \quad \funF(M) \tensor_\algC \funF(N) &\longrightarrow \funF(M\tensor N)
\nonumber \\
(1_k\tensor \e{a}m)\tensor_\algC (1_l\tensor \e{b}n)
&\longmapsto 
(-1)^{al} \zeta_{a,b}\, 1_{k+l + \kappa(a,b)}\tensor (\e{a}m\tensor \e{b}n) \ ,
\end{align}
where $\kappa(a,b)$ is defined in~\eqref{eq:kappa}.
	The proof that this map is well-defined and an isomorphism is as in Lemma~\ref{lem:tfun-fact}. To obtain the quasi-Hopf algebra structure in Section~\ref{sec:Q-def} we need to make the following choice, for $a,b\in \RS$,
\be\label{eq:zeta-def}
\zeta_{a,b} = 
\begin{cases}
1 ~~ &; ~ a \; \text{even} \\
q^{t b/2}  ~~ &; ~ a \; \text{odd} 
\end{cases}
\qquad .
\ee
Here $t$ is the odd integer fixed in Section~\ref{sec:Q-def}.

\begin{lemma}\label{lem:FMN-iso-mod}
	The isomorphisms $\funF_{M,N}$ are morphisms in $\Cmodl$.
\end{lemma}
\begin{proof}
We start by showing that $\funF_{M,N}$  intertwines the action of $\UQG{q}$. We first notice that the action of the generators 
on $\funF(M) \tensorC \funF(N)$ is given via  the coproduct formulas~\eqref{eq:coprod} and~\eqref{eq:H-coprod}.
This follows from the fact that $\funF(M) \tensorC \funF(N)$ is a quotient of $\funF(M) \tensor \funF(N)$ in $\catU_p^\oplus$. 

The check that $\funF_{M,N}$ intertwines the $H$-action literally  repeats the first part of the proof of Proposition~\ref{prop:funF-kappa}, where the necessity of  the shift by $\kappa(a,b)$ in the $\algC$ part becomes evident.
Since on weight modules we have $K = q^H$, it also follows that $\funF_{M,N}$ intertwines the $K$-action.

For the $E$-action we have 
\begin{multline}\label{eq:F-MN-Eaction}
(1_k\tensor \e{a}m)\tensorC (1_l\tensor \e{b}n)
 \xrightarrow{\;(\pi\tensor \pi)\Delta(E)\;} (1_{k}\tensor \e{a}m)\tensorC E(1_l\tensor \e{b}n) + E(1_{k}\tensor \e{a}m)\tensorC K(1_l\tensor \e{b}n)  \\
 = (1_{k}\tensor \e{a}m)\tensorC (1_{l+\kappa_+(b)}\tensor \e{\re{b+2}}E\, n) + (1_{k+\kappa_+(a)}\tensor \e{\re{a+2}}E \,m)\tensorC q^b (1_l\tensor \e{b}n)
 \end{multline}
 where $\pi\tensor \pi$  stands for  the product $\pi_{\funF(M)}\tensor \pi_{\funF(N)}$
 of the corresponding representations of $\UQG{q}$ and in the last equality we  used~\eqref{eq:funF-UqH-1} and~\eqref{eq:funF-UqH-3}. Then the map $\funF_{M,N}$ is
 \begin{align}
\text{RHS of}\, \eqref{eq:F-MN-Eaction} \xrightarrow{\;\funF_{M,N}\;} (-1)^{al} 1_{k+l + \kappa_+(a+b)}\tensor
\Bigl(& (-1)^{a\kappa_{+}(b)}\zeta_{a,\re{b+2}}\,  \bigl(\e{a}m\tensor\e{\re{b+2}}E\, n\bigr) \nonumber\\
&+  q^b \zeta_{\re{a+2},b}  \bigl(\e{\re{a+2}}E \,m\tensor \e{b}n\bigr) \Bigr)\ . \label{eq:funF-E-act}
\end{align}
This finally gives  the image of composition $\funF_{M,N}\circ(\pi\tensor \pi)\Delta(E)$ on $(1_k\tensor \e{a}m)\tensorC (1_l\tensor \e{b}n)$.
On the other side, we get
\be
E\circ\funF_{M,N}\colon\; (1_k\tensor \e{a}m)\tensorC (1_l\tensor \e{b}n) \mapsto  (-1)^{al} \zeta_{a,b}1_{k+l + \kappa_+(a+b)}\tensor \e{\re{a+b+2}} E(\e{a}m\tensor\e{b}n),
\ee
then applying $\Delta_t(E)$ for the action on $\e{a}m\tensor\e{b}n$ and for the choice of $\zeta_{a,b}$ in~\eqref{eq:zeta-def} we get indeed RHS of~\eqref{eq:funF-E-act}.

We then similarly check that the $F$-action also commutes with the morphisms $\funF_{M,N}$. 
It remains to show that the maps $\funF_{M,N}$ intertwine the 
	$\Lambda$-action.
This part repeats the second part of the proof of Proposition~\ref{prop:funF-kappa}.
We thus conclude that $\funF_{M,N}$ intertwine both the $\UQG{q}$- and $\algC$-actions and are therefore  morphisms  in $\Cmodl$.
\end{proof}

We can now state the following refinement of 	Proposition~\ref{thm:funF-equiv}.

\begin{proposition}\label{thm:funF-br-equiv}
The functor $\funF\colon \catQ_p \to \Cmodl$ defined in Proposition~\ref{prop:funF} and equipped with the isomorphisms $\funF_{M,N}$ in~\eqref{eq:isoF-MN} is a ribbon equivalence. 
\end{proposition}

\begin{proof}~
\\
$\bullet$~\textsl{$\funF$ is monoidal:}
We have to check  commutativity of the diagram~\eqref{eq:transport-assoc-diag},
or equivalently the equation~\eqref{eq:transport-assoc-diag-eq} where $\Phi$ should be replaced by $\Phi_t$. 
Since in terms of $H$-graded vector spaces, $\funF$ and $\funF_{M,N}$ are the same in Section~\ref{sec:ex-qHopf} and here, and since $\Phi_t$
 only involves~$K$, evaluating \eqref{eq:transport-assoc-diag-eq} gives again \eqref{eq:Phi-Z2p}. After substituting \eqref{eq:zeta-def} the coefficient in \eqref{eq:Phi-Z2p} becomes
\be
 (-1)^{a\kappa(b,c)} \ffrac{ \zeta_{a,b} \zeta_{\re{a+b},c}}{\zeta_{b,c} \zeta_{a,\re{b+c}}} = 
\begin{cases}
q^{-t c}, & \qquad a,b \; \text{are odd},\\
1, & \qquad  \text{otherwise}.\\
\end{cases}
\ee
Recall then from~\eqref{eq:idem-e} that the central idempotent $\idem_0$ is the projector onto even $a\in\ZZ_{2p}$ and $\idem_1$ is to odd values of $a$. We can therefore write
\be\label{eq:Phi-t-idem}
\Phi_t = 
\idem_0\tensor \idem_0 \tensor \one + 
\idem_0\tensor \idem_1 \tensor \one + 
\idem_1\tensor \idem_0\tensor \one +  
\idem_1\tensor \idem_1\tensor \sum_{c\in\ZZ_{2p}} q^{-tc} \e{c}\ ,
\ee
where the four terms correspond to the four possible parities of $(a,b) \in \ZZ_{2p}^{\times2}$.
Using~\eqref{eq:idemp-prim} the expression in~\eqref{eq:Phi-t-idem} is resummed into~\eqref{eq:Phi-t}.

\newcommand{\ev}{\mathrm{ev}}
\newcommand{\odd}{\mathrm{odd}}

\medskip

\noindent
$\bullet$~\textsl{$\funF$ is braided:}
We need to check 
commutativity of
the  diagram~\eqref{eq:transport-braiding-via-functorequiv},
or equivalently the condition in~\eqref{eq:transport-br-diag-eq} where $R$ should be replaced by $R_t$.
Recall from discussion around~\eqref{eq:braiding-descends} that the braiding in $\mathcal{M}=\Cmodl$ is inherited from the one in $\catU_p^\oplus$ where it is defined via the $R$-matrix in~\eqref{eq:R-H} as $c_{U,V} = \tau_{U,V} \circ R$. $c^{\mathcal{M}}$ is then the unique solution of commutativity of the right quadrant of the diagram in~\eqref{eq:braiding-descends}, i.e.\ $c^{\mathcal{M}}_{\funF(U),\funF(V)}=\tilde c_{\funF(U),\funF(V)}$.

The calculation for RHS of~\eqref{eq:transport-br-diag-eq} is the same as for the Cartan part in~\eqref{eq:br-diag-RHS} and for our choice in~\eqref{eq:zeta-def} is
\begin{align}
 &(1_k\tensor \e{a} u)\tensorL (1_l\tensor \e{b} v)
\nonumber \\
 & \qquad \longmapsto  \bigl(\delta_{a,\ev} + (-1)^l q^{\half tb}\delta_{a,\odd}\bigr) 1_{k+l +\kappa(a,b)} \tensor \sum_{(R_t)}(R_t)_2 \cdot \e{b}v\tensor(R_t)_1 \cdot \e{a}u\ .
\label{eq:braid-eq-RHS}
\end{align}
We emphasise  that  we assume here and below that $a,b\in \RS$, otherwise $q^{\half tb}$ would not be well-defined here.

The calculation of the LHS of~\eqref{eq:transport-br-diag-eq} is different due to the presence of the nilpotent part in the $R$-matrix~\eqref{eq:R-H}, and we do it again. We first calculate the image of the braiding:
 \begin{align}
 &(1_k\tensor \e{a} u)\tensorL (1_l\tensor \e{b} v)
\nonumber \\ 
&
\xymatrix{ \ar@{|->}[rr]^{c^{\mathcal{M}}_{\funF(U),\funF(V)} } && }
 (-1)^{al+bk}\sum_{n=0}^{p-1}q^{\half(a+2n)(b-2n)+\half n(n-1)} \ffrac{(q-q^{-1})^n}{[n]!}
 \nonumber \\
 &\hspace*{8em}  \times \bigl(1_{l+ \kappa(b,-2n)}\otimes  \e{\re{b-2n}}
 F^n.v\bigr) \tensorL \bigl( 1_{k+ \kappa(a,2n)}\otimes \e{\re{a+2n}}E^n.u\bigr) \ ,
\label{eq:braid-eq-cM}
\end{align}
 where  we used~\eqref{eq:funF-UqH-p1}.
 Then the image of $\funF_{V,U} $ on RHS of~\eqref{eq:braid-eq-cM} is
\begin{align} \label{eq:braid-eq-FVU}
 \text{RHS of}~\eqref{eq:braid-eq-cM}& 
 \xymatrix{ \ar@{|->}[r]^{ \funF_{V,U} }& }
(-1)^{al+bk}\sum_{n=0}^{p-1}q^{\half(a+2n)(b-2n)+\half n(n-1)} \ffrac{(q-q^{-1})^n}{[n]!}  \\
&  \times \bigl(\delta_{b,\ev} + (-1)^{k+\kappa(a,2n)}q^{\half t \re{a+2n}}\delta_{b,\odd}\bigr)1_{k+l + \kappa(a,b)}\tensor \bigl(\e{\re{b-2n}}F^n.v\tensor \e{\re{a+2n}}E^n.u\bigr)
 \nonumber
 \end{align}
where we used in the expression for $\kappa\bigl(\re{b-2n},\re{a+2n}\bigr)$ that $\re{\re{b-2n}+\re{a+2n}}=\re{a+b}$. 

Comparing the both sides of~\eqref{eq:transport-br-diag-eq}, which is RHS of~\eqref{eq:braid-eq-RHS} and RHS of~\eqref{eq:braid-eq-FVU}, we have thus to check the equality for $a,b\in \RS$
\begin{align}\label{eq:braid-eq-fin}
\sum_{(R_t)}(R_t)_2 \cdot \e{b}v\tensor(R_t)_1 \cdot \e{a}u &=\sum_{n=0}^{p-1}q^{\half(a+2n)(b-2n)+\half n(n-1)} \ffrac{(q-q^{-1})^n}{[n]!} \\
&\times  \bigl(\delta_{a,\ev} + q^{-\half tb}\delta_{a,\odd}\bigr) 
 \bigl(\delta_{b,\ev} + (-1)^{\kappa(a,2n)}q^{\half t \re{a+2n}}\delta_{b,\odd}\bigr)\nonumber\\
&\times  \bigl(F^n \e{b}v\tensor E^n\e{a}u\bigr)\nonumber
\end{align}
using $R_t$ from~\eqref{eq:R-quasiH}, we also used here $\e{\re{b-2n}}F^n = F^n \e{b}$.
We note that the second line here can be simplified using that $t$ is odd and $\kappa(a,2n)$ is integer: we can replace $(-1)^{\kappa(a,2n)}$ by $(-1)^{t\kappa(a,2n)}$ and thus
\be\label{eq:second-line-sim}
(-1)^{\kappa(a,2n)}q^{\half t \re{a+2n}} = q^{\half t(a+2n)} \ .
\ee

Recall then that $\idem_0$ is the projection onto eigenspaces with even $a$ and $\idem_1$ is onto odd values of $a$. 
To check~\eqref{eq:braid-eq-fin} for the four possible cases of parity of $(a,b)$ it is convenient to rewrite the expression for $R_t$ in \eqref{eq:R-quasiH} in terms of central idempotents as
\begin{multline}\label{eq:Rt-4-sectors}
R_t= \frac{1}{p}\sum_{n,s,r=0}^{p-1} \frac{(q-q^{-1})^n}{[n]!}q^{\frac{n(n-1)}{2}+2n(s-r)-2sr} \,E^nK^s\otimes F^n K^r\\
\times (\idem_0\otimes \idem_0 +q^{t r}  \idem_0\otimes \idem_1 + q^{-t(n+s)} \idem_1\otimes \idem_0 + q^{\frac{t^2}{2} +t r -t(n+s)}  \idem_1\otimes \idem_1)\ .
\end{multline}
We illustrate 	the verification of~\eqref{eq:braid-eq-fin} 
only in the case when both $a$ and $b$ are odd (the other case are treated similarly). 
	On the LHS of~\eqref{eq:braid-eq-fin} we have after a simple calculation (and after flipping the tensor factors)
\be\label{eq:Rt-ea-eb}
R_t(\e{a}\tensor \e{b}) = \sum_{n=0}^{p-1} \ffrac{(q-q^{-1})^n}{[n]!} q^{\half n(n-1)} q^{nb-na+nt-2n^2} q^{\half(at -bt + ab)} E^n\e{a} \tensor F^n \e{b} \ ,
\ee
where we used only the $\idem_1\tensor \idem_1$ contribution 
	from~\eqref{eq:Rt-4-sectors}, 
applied~\eqref{eq:K-en} and evaluated the sum over $s$. 
	On the RHS 
of~\eqref{eq:braid-eq-fin} we have the contribution for $\delta_{a,\odd}\delta_{b,\odd}$ 
(using~\eqref{eq:second-line-sim} and applying the flip):
\be\label{eq:ab-odd}
\delta_{a,\odd}\delta_{b,\odd}\colon \; 
\sum_{n=0}^{p-1}\ffrac{(q-q^{-1})^n}{[n]!} q^{\half n(n-1)} q^{nb-na+nt-2n^2}  q^{\half(at -bt + ab)} E^n\e{a} \tensor F^n \e{b} \ ,
\ee
and we see that   \eqref{eq:ab-odd} agrees with~\eqref{eq:Rt-ea-eb}.
We have thus proven that the equality~\eqref{eq:transport-br-diag-eq} holds and thus $\funF$ is braided.

\begin{remark}
In the proof we require $a,b \in \RS$, but we note 
that by construction the RHS of \eqref{eq:braid-eq-fin} must be independent of the choice of representatives $a,b$. This is not immediately obvious from the explicit expression, so let us verify it in the case where $a$ and $b$ are odd, i.e.\ in \eqref{eq:ab-odd}. It is only in the third $q$-exponent one have to check: we first  write it  in the form
$q^{\half({a}(t+{b}) -{b}t)}$ -- it is now clear that it does not depend on the choice of the representative~${a}$ as $t+{b}$ is even. Then we rewrite it as
 $q^{\half({a}t + {b}({a} - t))}$ and it is clearly independent of the choice of the representative ${b}$ as ${a}-t$ is even.  
\end{remark}

\medskip

\noindent
$\bullet$~\textsl{$\funF$ is ribbon:}
We finally check the ribbon twist in $\catQ_p$, similarly to step~\ref{plan-step5} in Section~\ref{sec:ex-qHopf}.  Recall that the twist on 
	$M \in \catU_p^\oplus$ 
is $\theta^{\catU}_M = \ribbon^{-1}.(-)$ 
with $\ribbon$ defined  in the second line of~\eqref{eq:R-H}, and for $U \in \catQ_p$ is $\theta^{\catQ}_U = \ribbon_t^{-1}.(-)$ with $\ribbon_t$ from~\eqref{eq:ribbon-quasiH}.  
The required
 equation~\eqref{eq:transport-twist} for the inverse twist
on $1_k \otimes \e au\in\funF(U)$ then takes the form
\be
	\ribbon \bigl(1_k \otimes \e{a} u\bigr) = 1_k \otimes \ribbon_t \,\e au \ .
\ee
LHS of this equation is
\be
\ribbon \bigl(1_k \otimes \e{a} u\bigr) =   \sum_{n=0}^{p-1} (-1)^n  \ffrac{(q-q^{-1})^{n}}{[n]!} 
 q^{-\frac{3}{2} n-\half (n^2+ a^2)} 1_k\tensor \e{a} K^{p-1-n}F^nE^n.u\ ,
\ee
where we first used~\eqref{eq:funF-UqH-p1}, then~\eqref{eq:funF-UqH-2}, then rewrote $S(F^n)=(-1)^n q^{n(n-1)}K^n F^n$ and used again~\eqref{eq:funF-UqH-p1}. We now take the sum over $a\in\ZZ_{2p}$ for both sides, so the LHS above is just $\ribbon \bigl(1_k \otimes  u\bigr)$, while for the RHS we 
	repeat step $(*)$ in \eqref{eq:ribbon-Z2p} to carry out the sum over $a$.
The result is
\be
	\ribbon \bigl(1_k \otimes  u\bigr) = \ffrac{1-\rmi}{2\sqrt{p}}  \sum_{n=0}^{p-1} \sum_{l=0}^{2p-1} \ffrac{(q-q^{-1})^{n}}{[n]!} 
 q^{np -\frac{3}{2} n+\half(l^2 - n^2)} \, 1_k\tensor K^{p-1+l-n}F^nE^n.u\ .
\ee
Finally, substituting $l\mapsto j+p+1+n$ and replacing the sum $\sum_{j=-p-1-n}^{p-2-n}$ just by $\sum_{j\in\ZZ_{2p}}$ we get indeed the equality
$\ribbon \bigl(1_k \otimes  u\bigr)  =  1_k \otimes  \ribbon_t . u$,  cf.~\eqref{eq:ribbon-quasiH}.
\end{proof}

\subsection{Factorisability}

In Section~\ref{sec:conventions} we gave two of the equivalent characterisations of factorisability: no non-trivial transparent objects and non-degeneracy of $\hat{\mathcal{D}}$. In the proofs of Propositions~\ref{prop:Z-2p-quasiH} and~\ref{prop:C-transp-obj} we used the first criterion. We could do the same here, but we prefer to use this opportunity to illustrate the second criterion.
	Thus, to show factorisability we will check that
the element $\hat{\mathcal{D}} \in \Q \otimes \Q$ defined  in~\eqref{eq:Q-for-fact}  provides a non-degenerate copairing, or in other words has an expansion $\hat{\mathcal{D}}= \sum_I f_I\tensor g_I$ over two bases $f_I$ and $g_I$ in $\Q$.

We first calculate the element $X$ from~\eqref{eq:Rem-X} to be $X=\one\tensor\one  + \idem_0\tensor \idem_1 \cdot(K^{-t}-\one)$. 
Since  $\mu \circ(S\tensor\id) \bigl(\Delta(\idem_1)\bigr)=0$ (recall the coproduct formula~\eqref{eq:cop-idem} and that $S(\idem_i)=\idem_i$) only the
	summand
$\one\tensor \one$ in $X$ contributes into the expression for $\hat{\mathcal{D}}$, 
and after a simple calculation we get
\be\label{eq:hat-D-Q}
\hat{\mathcal{D}} = \sum_{(W)} S(W_3) W_4  \otimes S(W_1) W_2  = \sum_{(M)} \Bigl(S(M_2)\tensor \idem_0 M_1 +  K^{-t} S(M_2)K^t \tensor \idem_1 M_1\Bigr).
\ee
We then rewrite the $M$-matrix~\eqref{eq:M}
 using the idempotents~\eqref{eq:idemp-prim}:
\begin{align}
  M_t 
  =  \sum_{m,n=0}^{p-1}  \sum_{i=0}^{2p-1}
  \ffrac{(\q - \q^{-1})^{m + n}}{[m]! [n]!}\,
  & \q^{ \half n(n - 1) - \half m(m + 1+2i)}  
  \left( \delta_{i+m,\mathrm{ev}} + \delta_{i+m,\mathrm{odd}}
  \,q^{t(m-n)} 
  \right)
  \nonumber\\
  & \times  \e{i-m}  F^{m} E^{n} \tensor K^i E^{m} F^{n}\ .
\label{eq:M-e}
\end{align}
The resulting $\hat{\mathcal{D}}$ obviously has the desired decomposition over product of two bases labelled by $I=\{(i,m,n)\} = \ZZ_{2p} \times \ZZ_p \times \ZZ_p$. Namely, if we  introduce
\begin{align}
f_{i,m,n}&:=
  \e{i-m}  F^{m} E^{n} \ ,
  \nonumber\\
g_{i,m,n}&:= 
\ffrac{(\q - \q^{-1})^{m + n}}{[m]! [n]!}\,
  \q^{ \half n(n - 1) - \half m(m + 1+2i)}  
  \left( \delta_{i+m,\mathrm{ev}} + \delta_{i+m,\mathrm{odd}}
  \,q^{t(m-n)} 
  \right)
  K^i E^{m} F^{n} \ ,
\end{align}
we have $M_t=  \sum_{i,m,n} f_{i,m,n}\tensor g_{i,m,n}$.
On the RHS of~\eqref{eq:hat-D-Q}, $S$ maps the basis $g_{i,m,n}$ to a basis, as $S$ is invertible, and the conjugation by $K^t$ does the same. The two central idempotents $\idem_0$ and $\idem_1$ in the second tensor factor  just provide decompositions in the two ideals corresponding to~\eqref{Q-decomp}, where we used the obvious fact that 
$\{ f_{i,m,n}\idem_k \,|\, i+m+k \text{ even} \}$ 
is a basis of $\idem_k\cdot\Q$.

This finishes the proof of our main theorem.

\appendix

\section{Proof of Proposition~\ref{prop:comm-alg-unique}}\label{app:proof-Prop}

{}From the outset, by Proposition~\ref{prop:all-cont-braidings} we may simplify our lives and set $r=1$. The result for general $r$ is recovered by the appropriate braided monoidal equivalence. 

Let $A \in \mathcal{H}^\oplus$ be an algebra satisfying 1--3. By condition 3,
we can write $A = \bigoplus_{s \in U} \mathbb{C}_s$ for some subset $U \subset \mathbb{C}$.
Since $A$ is unital, we have $0 \in U$.
We may assume (by passing to an isomorphic algebra if necessary) that the unit is $\eta(1) = 1_0$. The product $\mu$ of $A$ can be described by constants $t_{a,b} \in \mathbb{C}$, $a,b \in U$, via
\be
	\mu(1_a \otimes 1_b) = t_{a,b} \, 1_{a+b} \ .
\ee
We adopt the convention that $t_{a,b}=0$ if either of $a,b$ is not in $U$.
The structure constants must satisfy, for all $a,b,c \in U$:
\begin{align}
&t_{a,b+c} \, t_{b,c} = t_{a+b,c} \, t_{a,b} &\text{(associativity)}
\nonumber\\
&t_{a,b} = t_{b,a} \,e^{\frac{\pi i}2 ab} &\text{(commutativity)}
\nonumber\\
&t_{0,a} = 1 = t_{a,0}  &\text{(unitality)}
\end{align}

Any invariant pairing on $A$ is of the form $\varpi = \eps \circ \mu$, where $\eps : A \to \one$ is non-zero. Up to a non-zero factor, the pairing is hence given by $\varpi(1_a,1_b) = \delta_{a+b,0} \, t_{a,b}$. Non-degeneracy requires
\be\label{eq:alg-unique-aux1}
	s \in U ~\Rightarrow~ -s \in U
	\text{ and }
	t_{s,-s} \neq 0 \ .
\ee

It will be convenient to abbreviate $\mu(1_a \otimes 1_b)$ by just $1_a1_b$.  The rebracketing $(1_{-a}1_a)(1_b1_{-b}) = (1_{-a}(1_a1_b))1_{-b}$ results in,
	for $a,b \in U$,
\be
	0
	\overset{\text{non-deg.}}\neq
	t_{-a,a} t_{b,-b}
	\overset{\text{assoc.}}=
	t_{a,b} t_{-a,a+b} t_{b,-b} \ .
\ee
In particular we learn the important constraint
\be\label{eq:alg-unique-aux2}
	a,b \in U ~\Rightarrow~ a+b \in U \text{ and } t_{a,b} \neq 0 \ .
\ee
Together with \eqref{eq:alg-unique-aux1} this shows that $U$ is a subgroup of $\mathbb{C}$.

Next we exploit commutativity. Applying the commutativity condition twice gives $t_{a,b} = t_{a,b} e^{ \pi i ab}$. Since $t_{a,b} \neq 0$ 
	for $a,b\in U$,
we obtain the condition 
\be
	a,b \in U ~\Rightarrow~ ab \in 2 \mathbb{Z} \ .
\ee
For $a=b$ we learn that $a^2 \in 2\mathbb{Z}$. In particular, $U$ is discrete, and for all $a \in U$
	we have $a \in \mathbb{R} \cup i \mathbb{R}$.
In fact, we even have $U \subset \mathbb{R}$ or $U \subset i\mathbb{R}$, for otherwise there would be nonzero $a,b \in U$ with $a \in \mathbb{R}$ and $b \in i\mathbb{R}$, and by \eqref{eq:alg-unique-aux2} then also $a+b \in U$, a contradiction. 

For $U = \{0\}$ we obtain $A = \mathbb{C}_0$. Consider now the case $U \neq \{0\}$.
Since $U$ is discrete, there is a nonzero element closest to $0$, call it $\gamma$. Since $U$ is a subgroup also $\mathbb{Z} \gamma \subset U$.
Suppose $\mathbb{Z} \gamma \subsetneq U$ and let $z \in U \setminus \mathbb{Z} \gamma$. Since $U \subset \mathbb{R}$ or $U \subset i\mathbb{R}$ there is $\lambda \in \mathbb{R}$ such that $z = \lambda \gamma$. A suitable $\mathbb{Z}$-linear combination of $z$ and $\gamma$ hence produces a nonzero element of $U$ which is closer to $0$ than $\gamma$, a contradiction. Thus 
\be
U = \gamma\mathbb{Z} \ .
\ee

Next we study the effect of algebra isomorphism. Let $A'$ be an algebra satisfying 1--3 with the same underlying object as $A$ and denote by $t'_{a,b}$ its structure constants. We may again take $t'_{a,0}=1=t'_{0,a}$. It is easy to see that $A$ and $A'$ are isomorphic as algebras iff there are $\varphi_a \neq 0$, $a \in U$,  such that for all $a,b \in U$,
$t'_{a,b} \,\varphi_a \,\varphi_b = \varphi_{a+b} \,t_{a,b}$, or, equivalently,
\be\label{eq:alg-unique-aux3}
	t'_{a,b} = \frac{\varphi_{a+b}}{\varphi_a \,\varphi_b} \, t_{a,b} \ .
\ee
More conceptually, $t$ is a 2-cocycle in group cohomology, and the above operation changes it by a 1-cocycle.
Note that necessarily $\varphi_0=1$. 
As suggested by \eqref{eq:alg-unique-aux3}, we may think of $t'_{a,b}$ as being determined by $t_{a,b}$ and $\varphi_a$, giving an algebra $A'$ isomorphic to $A$.

Fix now $(\varphi_{m\gamma})_{m \in \mathbb{Z}}$ 
	recursively
such that $\varphi_0=1$ and 
	$\varphi_{a+\gamma}/(\varphi_\gamma\varphi_a) = 1/t_{\gamma,a}$. 
Then $t'$ as determined by \eqref{eq:alg-unique-aux3} satisfies $t'_{\gamma,a} = \frac{\varphi_{\gamma+a}}{\varphi_\gamma \,\varphi_a} \, t_{\gamma,a} = 1$. That is, after possibly passing to an isomorphic algebra, we may assume that
\be\label{eq:alg-unique-aux4}
	t_{\gamma,a} = 1 \quad
	\text{for all} ~~ a \in U \ .
\ee

We now use associativity in the form $1_\gamma(1_{m\gamma}1_{n\gamma}) = (1_\gamma 1_{m\gamma})1_{n\gamma}$ to get $t_{\gamma,(m+n)\gamma} t_{m\gamma,n\gamma} = t_{(m+1)\gamma,n\gamma} t_{\gamma,m\gamma}$. Together with \eqref{eq:alg-unique-aux4} this becomes
\be
	t_{m\gamma,n\gamma} = t_{(m+1)\gamma,n\gamma} 
\quad
	\text{for all} ~~ 
	m,n \in \mathbb{Z} \ .
\ee
But this implies that $t_{m\gamma,n\gamma} = t_{\gamma,n\gamma} = 1$ for all $m,n$. 

Finally, since all $t_{a,b}$ are equal to 1, the commutativity condition becomes $ab \in 4 \mathbb{Z}$, or, equivalently, $\gamma^2 \in 4 \mathbb{Z}$. This proves the proposition.

\newcommand{\botpr}{\mathsf{a}}
\newcommand{\leftpr}{\mathsf{b}}

\section{Tensor products with the modules $\modO^{\pm}_s(1;\lambda)$}\label{app:O}

It is known that the representation category $\rep \UresSL2$ is not braidable~\cite{KS} because of the non-commutativity for $p>2$ of tensor products of a certain class of modules paramterised by points on $\CC\mathbb{P}^1$. (While for $p=2$ the tensor product is commutative the category nevertheless does not allow for a braiding~\cite{GR1}.)
These modules are denoted by $\modO^{\pm}_{s}(n,\lambda)$, for integer $n\geq1$ and $\lambda\in \CC\mathbb{P}^1$, and were introduced independently in~\cite{Xi,Feigin:2005xs}. We first recall the action in an explicit basis and then show examples of tensor product decompositions for $n=1$ where the non-commutativity is apparent. 
We will then turn to  a similar decomposition with respect to the new coproduct~\eqref{eq:cop-new}, that is, for the quasi-Hopf modification $\Q$, in order to see how the new coproduct $\Delta_t$ solves the issue of non-commutativity.

Let
$1\leq s\leq p{-}1$, $a=\pm$, and
$\lambda\in \CC\mathbb{P}^1$, so we will write 
	$\lambda\,{=}\,[\lambda_1:\lambda_2]$.  
The $\modO^{a}_{s}(1,\lambda)$
module is $p$-dimensional and has  a basis (we follow here~\cite{Feigin:2005xs})
\begin{equation}\label{Nbasis}
  \{\botpr_j\}_{0\le n\le s-1}\cup\{\leftpr_k\}_{0\le k\le p-s-1}
\end{equation}
with the $\UresSL2$-action given by
\begin{align}
  K\botpr_j&=aq^{s-1-2j}\botpr_j, \quad
  K\leftpr_k=-aq^{p-s-1-2k}\leftpr_k,
  \quad 0\le j\le s{-}1,~0\le k\le p{-}s{-}1,\notag\\
  E\botpr_j&=
  \begin{cases}
    a[j][s-j]\botpr_{j-1}, &1\le j\le s{-}1,\\
    \lambda_2 \leftpr_{p-s-1}, &j=0,\\
  \end{cases}
  \\
  E\leftpr_k&=-a[k][p-s-k]\leftpr_{k-1}, \quad 0\le k\le p{-}s{-}1,\notag\\
  \intertext{where we set $\leftpr_{-1}=0$, and} F\botpr_j&=
  \begin{cases}
    \botpr_{j+1}, &0\le j\le s{-}2,\\
    \lambda_1 \leftpr_0, &j=s{-}1,\\
  \end{cases}
  \\
  F\leftpr_k&=\leftpr_{k+1}, \quad 0\le k\le p{-}s{-}1,\notag
\end{align}
where we set $\leftpr_{p-s}=0$. 
One can check that different choices of representatives $\lambda_1,\lambda_2$ for $\lambda\,{=}\,[\lambda_1:\lambda_2]$ yield isomorphic modules, so that the notation $\modO^{a}_{s}(1,\lambda)$ is justified.

Let  $\modX^{-}_1$ denote the one-dimensional $\UresSL2$-module with the trivial action by $E$ and $F$ while  $K$ acts by $-1$.
Using the basis above it is straightforward to establish 
the isomorphisms~\cite[Prop.\ 3.4.1]{KS}:
\begin{equation}
\begin{split}\label{eq:modO-X}
\modO^{\pm}_s(1;\lambda)\otimes \modX^-_1&\cong \modO^{\mp}_s(1;-\lambda)\ ,\\
\modX^-_1\otimes\modO^{\pm}_s(1;\lambda)&\cong \modO^{\mp}_s(1;-(-1)^p\lambda)\ .
\end{split}
\end{equation}
 Of course, the tensor products in~\eqref{eq:modO-X} are not the only examples of the non-commutativity -- see more in~\cite[Sec.\ 3.4]{KS}. 

Using now the new  coproduct $\Delta_t$ (with the corresponding tensor product denoted by~$\tensor_t$) we get the following isomorphisms 
of $\Q$ modules
\begin{equation}\label{eq:modO-X-2}
\modO^{\pm}_s(1;\lambda)\otimes_{t} \modX^-_1\cong \modX^-_1\otimes_{t}\modO^{\pm}_s(1;\lambda)\cong \modO^{\mp}_s(1;-\lambda)
\end{equation}
for any odd values of $t$.
The first isomorphism can be taken to be the action with $\tau\circ R_t$ where $\tau$ is the flip of vector spaces and $R_t$ is introduced in~\eqref{eq:R-quasiH}.

%%%%%%%%%%%%%%%%%%%%%%%%%%%%%%%%%%%%     bibliography
\newcommand\arxiv[2]      {\href{http://arXiv.org/abs/#1}{#2}}
\newcommand\doi[2]        {\href{http://dx.doi.org/#1}{#2}}
\newcommand\httpurl[2]    {\href{http://#1}{#2}}

\end{document}